\documentclass{amsart}
\usepackage{amsthm,stmaryrd,enumerate}
\usepackage{amssymb}
\usepackage{epsfig}
\usepackage{verbatim}
\usepackage{hyperref}
\usepackage{mathrsfs}
\usepackage[all]{xy}
\usepackage{color}
\newcommand{\RT}{{R-\!T}}
\newcommand{\dg}[1]{{\widetilde{#1}}}
\newcommand{\TT}{{\mathsf Z}}
\newcommand{\Gr}{{G}}
\newcommand{\V}{\mathscr{V}}
\newcommand{\E}{\mathscr{E}}
\newcommand{\lk}{\operatorname{lk}}
\newcommand{\lkG}{{\lk^\Gr}}
\newcommand{\ev}{\operatorname{ev}}

\newcommand{\brk}[1]{{\left\langle{#1}\right\rangle}}
\newcommand{\p}[1]{{\left({#1}\right)}}
\newcommand{\FK}{\ensuremath{\mathbb{K}} }
\newcommand{\ve}{\varepsilon}
\newcommand{\pow}[1]{^{\bullet#1}}
\newcommand{\eg}{\stackrel\bullet=}
\newcommand{\ro}{r}
\newcommand{\g}{\ensuremath{\mathfrak{g}}}
\newcommand{\h}{\ensuremath{\mathfrak{h}}}
\newcommand{\rk}{{n}}
\newcommand{\roots}{R}
\newcommand{\Uqg}{\ensuremath{\mathcal U}}
\newcommand{\UqgH}{\ensuremath{\Uqg^{H}}}
\newcommand{\La}{{L_W}}
\newcommand{\Xr}{{X_r}}
\newcommand{\X}{{\dg X}}
\newcommand{\srcol}{{\C\setminus \Xr}}
\newcommand{\A}{{\mathsf A}}
\newcommand{\AZ}{{\mathsf{A}_{Jones}}}
\newcommand{\cG}{{\mathscr L}}
\newcommand{\cGZ}{{\mathscr{L}_{Jones}}}
\newcommand{\cM}{{\mathscr M}}
\newcommand{\cMZ}{{\mathscr{M}_{RT}}}
\newcommand{\Nr}{{\mathsf N}_\ro}
\newcommand{\Nrd}{{\mathsf N}_{\ro,\delta}}
\newcommand{\PP}{{\mathsf P}}

\newcommand{\Hr}{H_r}
\newcommand{\qr}{{q}}
\newcommand{\coh}{\omega}
\newcommand{\vp}{\varphi}

\newcommand{\e}{{\operatorname{e}}}
\newcommand{\slt}{{\mathfrak{sl}(2)}}

\newcommand{\UsltH}{{U_q^{H}\slt}}
\newcommand{\Ubar}{{\wb U_q^{H}\slt}}
\newcommand{\unit}{\ensuremath{\mathbb{I}}}
\newcommand{\cat}{\mathscr{C}}
\newcommand{\catd}{\mathcal{D}}
\newcommand{\D}{\mathscr{D}}
\newcommand{\Id}{\operatorname{Id}}
\newcommand{\bp}[1]{{\left(#1\right)}}
\newcommand{\bk}[1]{{\left\langle#1\right\rangle}}
\newcommand{\qn}[1]{{\left\{#1\right\}}}
\newcommand{\qN}[1]{{\left[#1\right]}}
\newcommand{\qd}{{\mathsf d}}
\newcommand{\qdim}{\operatorname{qdim}}
\newcommand{\End}{\operatorname{End}}
\newcommand{\Hom}{\operatorname{Hom}}
\newcommand{\Int}{\operatorname{Int}}
\newcommand{\sign}{\operatorname{sign}}
\newcommand{\tr}{\operatorname{tr}}
\newcommand{\C}{\ensuremath{\mathbb{C}} }
\newcommand{\Z}{\ensuremath{\mathbb{Z}} }
\newcommand{\R}{\ensuremath{\mathbb{R}} }
\newcommand{\N}{\ensuremath{\mathbb{N}} }
\newcommand{\wt}{\widetilde}
\newcommand{\wb}{\overline}
\newcommand{\gcolor}{\varphi}
\newcommand{\ms}[1]{\mbox{\tiny$#1$}}
\newcommand{\nsimeq}{{\not\simeq}}
\newcommand{\qbin}[2]{\left[\begin{array}{c}
      #1 \\
      #2 \end{array}\right]}
\newcommand{\sjtop}[6]{\left|\begin{array}{ccc}#1 & #2 & #3 \\#4 & #5 &
      #6\end{array}\right|} 
\newcommand{\epsh}[2]
         {\begin{array}{c} \hspace{-1.3mm}
        \raisebox{-4pt}{\epsfig{figure=#1,height=#2}}
        \hspace{-1.9mm}\end{array}}

\newtheorem{teo}{Theorem}[section]
\newtheorem{lemma}[teo]{Lemma}
\newtheorem{prop}[teo]{Proposition}

\newtheorem{question}[teo]{Question}

\theoremstyle{definition}
\newtheorem{defi}[teo]{Definition}
\newtheorem{rem}[teo]{Remark}
\newtheorem{example}[teo]{Example}

\theoremstyle{remark}

\renewcommand{\qedsymbol}{\fbox{\theteo}}

\newcounter{exo} \newcounter{numexercice}
\renewcommand{\theexo}{\arabic{exo}} 

\begin{document}
\title[Non-semi-simple surgery invariants of 3--manifolds]
{Quantum invariants of 3--manifolds via link surgery presentations and
  non-semi-simple categories }

\author[Costantino]{Francesco Costantino}
\address{Institut de Recherche Math\'ematique Avanc\'ee\\
  Rue Ren\'e Descartes 7\\
 67084 Strasbourg, France}
\email{costanti@math.unistra.fr}

\author{Nathan Geer}
\address{Mathematics \& Statistics\\
  Utah State University \\
  Logan, Utah 84322, USA} 
\thanks{The first author's research was supported by French ANR
  project ANR-08-JCJC-0114-01. Research of the second author was
  partially supported by NSF grants DMS-0968279 and DMS-1007197. 
  The authors would like to thank Utah State University, LMAM, Universit\'e de  Bretagne-Sud and Institut de Recherche Math\'ematique Avanc\'ee for all their  generous hospitality.}\
\email{nathan.geer@usu.edu}

\author{Bertrand Patureau-Mirand}
\address{UMR 6205, LMBA, universit\'e de Bretagne-Sud, universit\'e
  europ\'eenne de Bretagne, BP 573, 56017 Vannes, France }
\email{bertrand.patureau@univ-ubs.fr}

\begin{abstract}
  In this paper we construct invariants of 3--manifolds ``\`a la
  Reshetikhin-Turaev'' in the setting of non-semi-simple ribbon tensor
  categories.  We give concrete examples of such categories which lead
  to a family of 3--manifold invariants indexed by the integers.  We
  prove that this family of invariants has several notable features,
  including: they can be computed via a set of axioms, they
  distinguish homotopically equivalent manifolds that the standard
  Witten-Reshetikhin-Turaev invariants do not, and they allow the
  statement of a version of the Volume Conjecture and a proof of this
  conjecture for an infinite class of links.
\end{abstract}

\maketitle

\section{Introduction}

\subsection{Historical Overview} The
Witten-Reshetikhin-Turaev 3--manifold invariants are extremely
important and intriguing objects in low-dimensional topology.  These
invariants can be computed combinatorially and lead to Topological
Quantum Field Theories (TQFTs) and representations of mapping class
groups.  They have been studied extensively but still have a
mysterious topological significance.  Reshetikhin and Turaev \cite{RT}
gave the first rigorous construction of these invariants which have
become known as quantum invariants of 3--manifolds.  Their proof uses
surgery to reduce the general case to the case of links in $S^3$ then
applies certain quantum invariants of links associated to quantum
$\slt$ at a root of unity
 (see \cite{RT0}).
The reduction of the topology of 3--manifolds to the theory of links in
$S^3$ is well-known: any closed orientable connected 3--manifold is
obtained by surgery on some framed link in $S^3$.  Two manifolds $M_L$
and $M_{L'}$ obtained by surgery on $L$ and $L'$, respectively, are
homeomorphic if and only if the framed links $L$ and $L'$ may be
related by a series of Kirby moves (see \cite{Ki}).

Roughly speaking, the construction of the quantum invariant of
3--manifolds defined by Re\-she\-tikhin and Turaev can be described as
follows (for more details see \cite{RT}).  Consider the quotient
$\overline{U}$ of $U_q(\slt)$ defined by setting $E^r=F^r=0$ and
$K^r=1$, where $q$ is a root of unity of order $2r$ and $E,F$ and $K$
are the generators of $U_q(\slt)$.  Any finite dimensional
$\overline{U}$-module $V$ decomposes as $V\cong \oplus_{i=1}^n V_i
\oplus W$ where $V_i$ is a simple $\overline{U}$-module with non-zero
quantum dimension and $W$ is a $\overline{U}$-module with zero quantum
dimension.  By quotienting the category of finite dimensional
$\overline{U}$-modules by $\overline{U}$-modules with zero quantum
dimension one obtains a modular category $\catd$.  Loosely speaking, a
modular category is a semi-simple ribbon category with a finite number
of isomorphism classes of simple objects satisfying some axioms.  Let
$M$ be a manifold obtained by surgery on $L$.  If the $i^{\text{th}}$
component of $L$ is labeled by a simple module $V_i$ of $\catd$ then
consider the weighted link invariant
\begin{equation}\label{E:WeightedSum}
  \bp{\prod_i\qdim_\catd(V_i)}F(L),
\end{equation}
where $\qdim_\catd$ is the quantum dimension and $F$ is the
Reshetikhin-Turaev link invariant associated to $\catd$, see
\cite{RT0}.  The invariant of $M$ is the finite sum of such weighted
link invariants over all possible labelings of $L$.  The Kirby moves
correspond to certain algebraic identities.  Using the semi-simplicity
of $\catd$ it can be shown that the weight sum is preserved by the
Kirby moves.  
The Reshetikhin-Turaev 3--manifold invariant  construction can be generalized to semi-simple (modular) ribbon categories.   Obstructions to applying this construction to any ribbon  
category  $\cat$ include: 

\begin{enumerate}[\hspace{14pt}  (Ob1)]
\item \label{I:ObPiv1A} zero quantum dimensions, 
\item  \label{I:ObPiv4A}  infinitely many isomorphism classes of simple objects in $\cat$, 
\item  \label{I:ObPiv3A}  non-semi-simplicity of $\cat$.
\end{enumerate}

Most of the research on quantum invariants is based on semi-simple
 categories.  Apart from a few exceptions (\cite{He,Ku,Vz}),
the work that has been done on quantum invariants of links and
3--manifolds coming from non-semi-simple representation theory has
been centered on examples related to quantum $\slt$.  Such work has
been initiated by the independent and seminal works of Akutsu,
Deguchi, Ohtsuki \cite{ADO}, Kashaev \cite{Ka}, Viro \cite{Vi} and
others.  In \cite{MM}, Murakami and Murakami showed that the
Akutsu-Deguchi-Ohtsuki (ADO) invariants \cite{ADO} are related to Kashaev's
invariants \cite{Ka}.  The ADO invariants are also related to the multivariable
Alexander polynomials (see \cite{jM}).  More recently, Murakami and
Nagatomo \cite{MN} defined ``logarithmic invariants,'' Benedetti and
Baseilhac \cite{BB} extended Kashaev's construction to a ``quantum
hyperbolic field theory'', Kashaev and Reshetikhin \cite{KR} defined
tangle invariants from non semi-simple categories and Andersen and
Kashaev \cite{AK} constructed a TQFT out of quantum Teichm\"uller
theory.

This body of work is deep and requires new techniques involving
algebra, topology, geometry and mathematical physics.  Moreover, these
invariants are related to well known problems, including the Volume
Conjecture (see \cite{Ka, MM}).  

\subsection{Purpose of this paper}
To our knowledge no one has constructed a quantum 3--manifold invariant based on link surgery presentations arising from the non-semi-simple categories of representations of quantum groups.  The purpose of this paper is to fill this gap 
by defining Reshetikhin-Turaev type 3-manifold invariants from categories with Obstructions (Ob\ref{I:ObPiv1A})--(Ob\ref{I:ObPiv3A}).  
In particular, for each integer $r\geq 2$, we define two invariants $\Nr$ and $\Nr^0$ of triples  (a closed oriented 3-manifold $M$,
a colored link $T$ in $M$, an element $\coh$ of $H^1(M\setminus T; \C/2\Z)$), under certain admissibility conditions.  As explained in the next subsection these invariants have the following general interpretations (see Figures \ref{F:VisRepLink} and \ref{F:VisRep3man}):

\begin{itemize}
\item The invariants $\Nr(S^3,T,\coh)$ contain the multivariable
  Alexander polynomial, Kashaev's invariant and the ADO invariant of
  $T$. 
  Thus, $\Nr$ can be considered an \emph{extension} of these
  invariants to 3-manifolds other than $S^3$, via a modified version
  of the quantum invariant construction.\vspace{5pt}

\item The invariant $\Nr^0(S^3,T,0)$ is the colored Jones polynomial
  of $T$ at a root of unity.  Thus, $\Nr^0$ can be considered an
  extension of this invariant to general 3-manifolds.  It is an open
  question if $\Nr^0$ is related to the standard Reshetikhin-Turaev
  quantum invariant.
\end{itemize}

This approach is useful and powerful for several reasons, including:
\begin{description}
\item[Computable and new]  We compute our invariants for a number of examples.  In particular,  we show that $\Nr$ distinguishes
homotopically equivalent manifolds that the standard
Reshetikhin-Turaev invariants do not
(see Subsections \ref{SS:DistLens} and \ref{SS:torusKnots}).  
\item[Volume Conjecture]  We formulate a version of the Volume Conjecture for links in arbitrary manifolds and give a proof
of this conjecture for the so called fundamental hyperbolic links (see
Subsection \ref{SS:VolCon}).  When $M=S^3$ this conjecture is the usual Volume Conjecture.
\item[Recovers standard {\RT}] Our construction applied to modular
categories arising from simple Lie algebras yields invariants which
are equal to the usual Reshetikhin-Turaev invariants (see
Remark \ref{R:RTrecoved}).
\item[Properties]  The invariants of this paper have useful properties including formulas for the connected sum (see Subsection \ref{S:consum}).
\item[General] The construction produces  invariants from
non-semi-simple representation theory associated to any quantum simple
Lie algebra (see Section \ref{S:otherQG}).
\end{description}

\subsection{Summary of results}
We will now describe the results of this paper in the context of
quantum $\slt$.
Fix a root of unity $q=\exp(\frac{i\pi}{r})$ 
where $r\geq 2$ is an odd or twice an odd integer.
Let $\UsltH$ be the quantization of $\slt$ defined in
Subsection \ref{SS:QUantSL2H}.  This quantization has five generators:
$E,F, H, K, K^{-1}$ where $H$ can be considered as the logarithm of
$K$.  Let $\Ubar$ be $\UsltH$ modulo the relations $E^\ro=F^\ro=0$.
We consider the category $\cat$ of finite dimensional $H$-weight
modules over $\Ubar$ such that $K$ acts as the operator $q^H$.  Unlike
the modular category $\catd$ discussed above we do not require that
$K^r=1$ instead $K^r$ acts as a scalar on a simple module of $\cat$.
Also unlike $\catd$ here we do not take a quotient of the category of
$\Ubar$-modules.
The category $\cat$ is a ribbon category 
with Obstructions (Ob\ref{I:ObPiv1A})--(Ob\ref{I:ObPiv3A}).  
The isomorphism classes of simple modules of $\cat$ are indexed by $\C$.   We will consider two subsets of simple modules.  First, let $\A$ be the set of simple modules indexed by $(\C \setminus \Z) \cup \{kr
: k\in \Z\}\subset \C$.  Each module in $\A$ has a vanishing quantum dimension.  
Note that in this paper we use a non-standard indexing (middle weight notation) of the simple modules of $\cat$; in particular, the module of $\A$ indexed by $0\in \C$ is known as ``Kashaev's module'' and is crucial in Murakami and
Murakami's reformulation of the Volume Conjecture \cite{MM}.  
Second,  let $\AZ$ be the set of simple modules in $\cat$ indexed by $\{1,2,...,r-1\}$.  
The set $\AZ$ coincides with the set of ``standard'' simple modules $\{V_i\}$ over $\overline{U}$ with non-zero quantum dimension used in the construction of the Reshetikhin-Turaev invariants (see Equation \ref{E:WeightedSum}).  

To explain our approach, let us first 
discuss the underlying link invariants.  Let $\cGZ$ be the sets of isotopy classes of oriented, framed links in $S^3$ whose components are colored by modules in $\AZ$.  Similarly, let $\cG$ be the sets of isotopy classes of oriented, framed links in $S^3$ whose components are colored by modules in $\A\sqcup \AZ$ such that at least one color is in $\A$.  
Let $F_r:\cGZ \sqcup \cG\to \C$ be the Reshetikhin-Turaev link invariant  (see \cite{RT0, Tu}).  
By restricting $F_r$ to $\cGZ$ one obtains 
the colored Jones polynomial at the root of unity $q=\exp(\frac{i\pi} {r})$.  On the other hand, $F_r$ vanishes on its restriction to $\cG$ (since the quantum dimension of a module in $\A$ is zero). 
In \cite{GPT}, the second two authors and Turaev give a construction to overcome this vanishing: we show that one can replace the quantum dimension $\qdim$ with a new modified dimension $\qd$.   
The resulting invariant, denoted $F_r':\cG\to \C$ is not defined on $\cGZ$ but it is related to other interesting invariants: (1) if $L$ is a  link whose components are colored by the Kashaev module then $F_r'(L)$ is Kashaev's invariant \cite{Ka} (a key ingredient of the Volume Conjecture), (2) for general links it coincides with the ADO invariant and (3) if $r=2$ it coincides with the multivariable Alexander polynomial. 

We can see $F'_r$ as a kind of first order extension of $F_r$ as follows:
consider the invariant $\wb F_r$ on $\cGZ\sqcup\cG$ with values in $\C[h]/(h^2)$ defined by $$\wb F_r(L)=\left\{
  \begin{array}{l}
    F_r(L)\text{ if }L\in\cGZ\\
    hF'_r(L)\text{ if }L\in\cG
  \end{array}\right. .$$  Then, for all $L_1,L_2\in \cGZ\sqcup\cG$, we have
for the split link $L_1\sqcup L_2$:
\begin{equation}
  \label{eq:barF}
  \wb F_r(L_1\sqcup L_2)
  =\wb F_r(L_1)\wb F_r(L_2).
\end{equation}
In particular, the invariant $F'_r$ recovers $F_r$ as follows: $F_r$
is equal to the
map 
\begin{equation}\label{E:F'recF}
\cGZ\sqcup \cG\to \C \;\;\;\; \text{ defined by } \;\;\;\; L\mapsto \frac{F'_r(L\sqcup o)}{F'_r(o)}
\end{equation}
where $o$ is an unknot colored by any element of $\A$.  
Note here that if $L\in  \cG$ then $F'_r(L\sqcup o)=0$.  
The relationship between the link invariants discussed in this
paragraph are visually represented in Figure \ref{F:VisRepLink}.
\begin{figure}[t]  
 \centering
 $\epsh{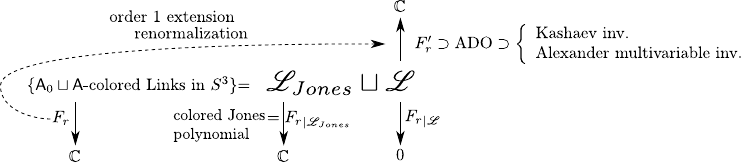}{20ex}$
  \caption{Visual representation of underlying link invariants}
  \label{F:VisRepLink}
\end{figure}
One of the main points of this paper is to develop 3-manifold invariants $\Nr^0$ and $\Nr$ with an analogous relationship (see Figure \ref{F:VisRep3man}).   

With this in mind consider the following sets:  
$$\cMZ=\{(M,\coh)|M\ {\rm closed\ oriented\ 3-manifold\ and}\ \coh\in H^1(M;\Z/2\Z)\}$$ 
$$\cM=\{(M,\coh)|M\ {\rm closed\ oriented\ 3-manifold\ and}\ \coh\in H^1(M;\C/2\Z)\setminus H^1(M;\Z/2\Z)\}.\footnote{The informal notation here means ``cohomology classes which are not in the image of the natural map $H^1(M;\Z/2\Z)\to H^1(M;\C/2\Z)$ induced by the universal coefficient theorem''. Also pairs are considered up to the action induced by the set of oriented preserving diffeomorphisms.}$$
The standard {\RT} invariants are refined by Kirby and Melvin in \cite{KM} to a map $Z_r:\cMZ \to \C$.   This map can be considered as a map $Z_r:\cMZ\sqcup \cM \to \C$ which is $0$ on $\cM$. 
One of the starting point of this paper is the observation that this is not an artificial ``extension by $0$'', but rather that it is natural that $(Z_r)|_\cM=0$ and that a suitable modification recovers new and interesting information on $\cM$ exactly as $F'_r$ does on $\cG$.

Let us explain this in more detail.     The  standard {\RT} 3-manifold invariant  can be extended to pairs $(M,\coh)\in\cMZ\sqcup \cM$ as follows.  First, $\coh$ induces a $\C/2\Z$-valued coloring on a surgery presentation $L$ of $M$:  the color of a component $L_i$ is $\coh([m_i])$ where $m_i$ is a meridian of $L_i$ and $[m_i]$ is the corresponding homology class  (see Subsection \ref{SS:CohomGcoloring} for details).   Using the quotient map $\C\to \C/2\Z$ such a $\C/2\Z$-valued coloring on $L$ can be ``lifted'' to a $\C$-valued coloring of $L$ or equivalently a $\A\sqcup \AZ$-valued coloring of $L$.  For such a coloring one can consider the  weighted  
link invariant given in Equation \eqref{E:WeightedSum}.   If $(M,\coh)\in\cMZ$ then there are only a finite number of lifts and the sum over all possible lift is the standard {\RT} 3-manifold invariant of $M$.  On the other hand, if $(M,\coh)\in\cM$ then there are infinite number of lifts.  For any of these lifts the weighted link invariant given in Equation \eqref{E:WeightedSum} is zero, since all modules in $\A$ have $0$ quantum dimension.  Thus, the natural extension of {\RT}  invariant to $\cM$ is trivial.  This phenomenon is a generalization of what happens when extending $F_r$ to $\cG$ for links in $S^3$. 

To overcome this problem and extend the {\RT} invariants to $\cM$ in a non-trivial way, we start by replacing $F_r$ and $\qdim$ in Equation \eqref{E:WeightedSum} with $F_r'$ and the modified dimension $\qd$, respectively (see Subsection \ref{SS:RibbonGraphs}).  Suppose $\coh$ induces a $(\C\setminus\Z)/2\Z$-valued coloring of the components of $L$.  As explained above, such a coloring can be lifted to an infinite number of $\A$-valued colorings of~$L$.  For one of these $\A$-valued colorings consider the weighted link invariant
\begin{equation}\label{E:WeightedSumWithdF'}
  \bp{\prod_i\qd(V_i)}F'_r(L),
\end{equation}
where the $i^{\text{th}}$ component of $L$ is labeled by a simple module $V_i\in\A$.  
We choose a finite number of such weighted invariants and prove that a normalization of their sum leads to  
an invariant $\Nr:\cM\to \C$.  Here the choice of the weighted invariants is not canonical, however we prove that the invariant is independent of the choice.   The invariant is analogous to $F'$ in the case of links in $S^3$.  In particular, it  is not defined on $\cMZ$.  

We also define a second invariant $\Nr^0:\cM\sqcup \cMZ\to
\C$ as a kind of order $0$ extension of $\Nr$: $$\Nr^0(M,\coh)=\frac{\Nr \big(
  (M,\coh)\#(M',\coh')\big)}{\Nr(M',\coh')}$$ where $\Nr$ does not vanish on
$(M',\coh')\in\cM$  (for further details on the notion of connected sum of elements of $\cM$ see the discussion preceding Theorem \ref{T:N0sl2}).
This definition is analogous to the extension of the {\RT} link invariant
expressed in Equation \eqref{E:F'recF}.  In particular, $\Nr^0$ restricted to
$\cM$ is zero and  one can consider the invariant $\wb \Nr$ on $\cMZ\sqcup\cM$ with values in $\C[h]/(h^2)$ defined by $$\wb \Nr=\left\{
  \begin{array}{l}
    \Nr^0(M,\omega)\text{ if }\omega\in H^1(M,\Z/2\Z)\\
    h\Nr(M,\omega)\text{ else }
  \end{array}\right..$$
   Then for all $(M_1,\omega_1),(M_2,\omega_2)\in \cMZ\sqcup\cM$ we have
\begin{equation}
  \label{eq:barF}
  \wb 
  \Nr((M_1,\coh_1)\#(M_2,\coh_2))=\wb \Nr(M_1,\coh_1)\wb \Nr(M_2,\coh_2).
\end{equation}  
It is open question if $\Nr^0$ is equal to the standard {\RT} 3-manifold
invariant (see \cite{CGP2} for some results in this direction).
\begin{figure}[t]  
 \centering
 $\epsh{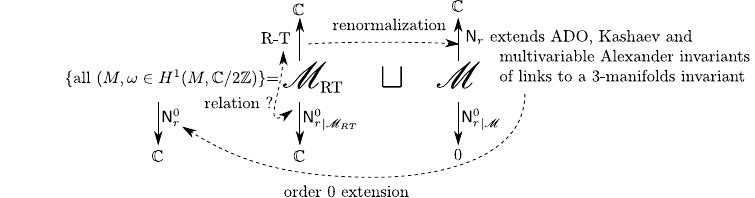}{24ex}$
  \caption{Visual representation of 3-manifold invariants}
  \label{F:VisRep3man}
\end{figure}  The 3-manifold invariants discussed above are represented in Figure \ref{F:VisRep3man}.  Also, the 
existence of the main invariants of this paper, in the case when $M$ contains no link (or more generally trivalent graph), can be expressed in the following theorem:
\begin{teo}\label{T:NrNr0Summary}
There exists maps $\Nr:\cM\to \C$ and $\Nr^0:\cM\cup \cMZ\to \C$ with the following properties:
\begin{enumerate}
\item  If $(M,\coh)\in \cM$, then there exists a link surgery presentation $L$ of $M$ such that $\coh$ induces a coloring on $L$ whose values are all in  $(\C\setminus \Z)/2\Z$.  As described above such a coloring leads to an invariant $\Nr(M,\coh)$ which is defined as a normalization of a sum of a finite number of the weighted link invariants given in Equation \eqref{E:WeightedSumWithdF'}.
\item If $(M,\coh)\in \cM$ then $\Nr^0(M,\coh)=0$.  In addition, if $(M',\coh')\in \cM\sqcup \cMZ$ then
  $$\Nr((M,\coh)\#(M',\coh'))=\Nr(M,\coh)\Nr^0(M',\coh').$$
\end{enumerate}
\end{teo} 

\subsection{Subtleties of the construction.} In this subsection we discuss
three technical and difficult points of our construction.  
We now consider the full invariant $\Nr$ defined on some triple $(M,T,\coh)$
where $T$ is a complex colored framed link or trivalent graph in $M$.
First, not all
surgery presentations can be used to compute $\Nr(M,T,\coh)$.  For this reason
we introduce the notion of a \emph{computable} surgery presentation.  In the
case when $T=\emptyset$ this means the coloring takes values in $(\C\setminus
\Z)/2\Z$, see part 1 of Theorem \ref{T:NrNr0Summary}.  If $T\neq \emptyset$,
it may happen that a triple does not admit a computable presentation, however
we can modify the construction as follows to obtain a non-trivial invariant.
A \emph{$H$-stabilization} is a connected sum of the Hopf link with a
component of $T$.  The definition of $\Nr$ can naturally be extended by first
doing a $H$-stabilization then re-normalizing the invariant of the new triple
to account for such a stabilization.  Thus, if $T\neq \emptyset$ the invariant
$\Nr$ is defined for all triples $(M,T,\coh)$.

The second delicate point of the construction is how to select finitely many modules to color the components of a computable surgery presentation $L$.   In the standard case the set of modules $V_i$ very loosely depends on $\coh\in H^1(M;\Z/2\Z)$: basically the parity of the weights in $V_i$ should coincide with the parity of $\coh$ on the meridians of the link. This allows one to select a finite set of modules for each component of $L$ and to color the component with a formal linear combination all these possible choices, called the \emph{Kirby color}.
In our case the weights of the modules belong to $\C$ and they should coincide modulo $2\Z$ with the value of $\coh$ on the meridians of the link. 
The problem here is  there is not a unique Kirby color but rather we can define an infinite set of Kirby colors for each weight in $\C/2\Z$. Hence we need to prove that the value of the invariant does not depend on which of these Kirby colors one choses on each component of the link. 

The third technical point in our construction is that while performing Kirby calculus moves relating two computable surgery presentations one may pass through some non-computable presentations. So a suitable work must be done to avoid  non-computable presentations.  Moreover the presence of the cohomology classes forces us to consider more seriously the topological meaning of a sequence of handle slides relating two presentations and in particular to deal with the (isotopy class of) diffeomorphisms it induces. Fortunately as already remarked by Gompf and Stipsicz, Kirby's Theorem already contains all the informations concerning this problem: see Subsection \ref{sub:kirbygompf} for further discussion of this issue.

\subsection{Open questions and further developments}

The work of this paper leads to a number of interesting open problems.
First, it is natural to expect that our invariants extend to certain
kinds of Topological Quantum Field Theories as the {\RT} invariants
extend to the TQFTs first constructed in \cite{BHMV}.  If such a TQFT
exists then it should give rise to a new set of quantum
representations of the Mapping Class Groups of surfaces.  Second, as
shown in Section \ref{sec:examples}, the invariants may be used to
prove a version of the Volume Conjecture for certain pairs $(M,T)$
where $T$ is a link in a manifold $M$.  Thus, it is natural to ask if
this conjecture is true for other pairs $(M,T)$, in particular, in the
case when $T$ is empty.  Finally, a natural problem is to relate the
invariants of this paper to those defined in \cite{GPT2}; an analogous
relationship exists between the usual {\RT} invariants and Turaev-Viro
invariants.  Giving such a relationship would be interesting also
because it would relate the invariants of this paper with the
generalized Kashaev invariants of \cite{GKT}, via the results of
\cite{GP2}.

\subsection{Organization of the paper.}
To make the paper accessible to
more readers, in Section~\ref{S:sl2NrInv} we provide a definition of our invariant in the case of $\Ubar$ using a set of axiomatic properties defining $F'_r$ on links in $S^3$ and then we detail its properties. This allows the
reader to be able to compute and use our invariants immediately.  Then
in Section~\ref{sec:examples} we give several computations and
properties of these invariants.
In Section~\ref{sec:generalconstruction} we introduce a general
categorical construction and state the main results of the paper.
Section~\ref{S:ProofsT-coh-adm} contains the proofs of the results
stated in Section~\ref{sec:generalconstruction}.  In
Section~\ref{S:ProofsOfNrManInv} we prove that the combinatorial
invariants of Section~\ref{S:sl2NrInv} are examples of the general
theory defined in Section~\ref{sec:generalconstruction}.
Section~\ref{S:otherQG} is devoted to showing that all quantized
simple Lie algebras lead to examples of the general theory of this
paper.

\section{The invariants $\Nr$ and $\Nr^0$}\label{S:sl2NrInv}
\subsection{Notation} \label{SS:NotationNr} All manifolds in the
present paper are oriented, connected and compact and all the
diffeomorphisms preserve the orientations unless explicitly stated.
By a graph we always mean a finite graph with oriented edges (we allow
loops and multiple edges with the same vertices).  Given a set $Y$, a
graph is said to be \emph{$Y$-colored} if it is equipped with a map
from the set of its edges to $Y$.  A \emph{framed} graph $\Gamma$ in
an oriented manifold $M$ is an embedding of $\Gamma$ into $M$ together
with a vector field on $\Gamma$ which is nowhere tangent to $\Gamma$,
called the \emph{framing}
We assume that all edges of a vertex are tangent to the same plane
which does not contain the vector of the framing. The framing is seen
up to homotopy of vector fields constantly transverse to these tangent
planes and so, together with the orientation of the manifold, gives a
cyclic ordering of the edges of any vertex.

Let $r$ be an integer greater or equal to $2$ and let
$\qr=\e^{i\pi/r}$.  For $x\in \C$, we use the notation $q^x$ for
$\e^{x i\pi/r}$ and let $\qn x=q^{x}-q^{-x}$.  For all $ x,y\in \C$
with $x-y\in \{0,1,\ldots, r-1\}$, let $\qbin{x}{y}=\prod_{j=1}^{x-y}
\frac{\{x+1-j\}}{\{j\}}$.  Set $\Xr=\Z\setminus r\Z \subset \C$ and
define the \emph{modified dimension} $\qd:\srcol\to\C$ by
\begin{equation}\label{E:Def_qd}
  \qd(\alpha)=(-1)^{r-1}{\qbin{\alpha+r-1}{\alpha}}^{-1}
  =(-1)^{r-1}r\frac{\sin(\frac{\pi \alpha}r)}{\sin(\pi \alpha)}.
\end{equation}
Finally, let 
\begin{equation}
  \label{eq:Hr}
  \Hr=\{1-\ro,3-\ro,\ldots,\ro-3,\ro-1\}
\end{equation}
and for any $x\in \C$ let $\wb{x}\in \C/2\Z$ be the class of $x$
modulo $2\Z$.

\subsection{Axiomatic definition of the invariant $\Nr$ of graphs in
  $S^3$}\label{S:axiom}
In \cite{ADO}, Akutsu, Deguchi and Ohtsuki define generalized
multivariable Alexander invariants, which contain Kashaev's
invariants (see \cite{Kv,MM}).  In \cite{jM}, Jun Murakami gives a
framed version of these link invariants using the universal $R$-matrix
of quantum $\slt$ and calls them the colored Alexander invariant.
Here we consider a generalization of these invariants (when $r=2$,
such a generalization was already considered by Viro in \cite{Vi}).

As above, let $r\in \Z$ with $r\geq 2$ and let $\qr=\e^{i\pi/r}$.  Let
$\cG$ be the set of oriented trivalent framed graphs in $S^{3}$ whose
edges are colored by element of $\srcol$.  In Subsection
\ref{SS:NfromSl2} we will show that the following axioms define an
invariant $\Nr :\cG\to\C$.  Let $T, T'\in\cG$.
\begin{subequations}
\renewcommand{\theequation}{$\mathsf N$ \alph{equation}}
\begin{enumerate}
\item Let $e$ be an edge of $T$ colored by $\alpha$.  If $T'$ is
  obtained from $T$ by changing the orientation of $e$ and its color
  to $-\alpha$ then $\Nr(T')=\Nr(T)$.  In other words,
  \begin{equation}
    \label{eq:ch-orient}
    \Nr\left(\epsh{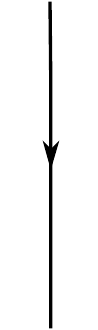}{8ex}\put(-11,3){\ms{\alpha}}\right)=
    \Nr\left(\,\epsh{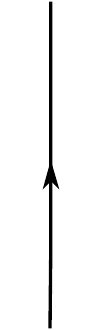}{8ex}\put(-17,0){\ms{-\alpha}}\right).
  \end{equation}
\item Let $\alpha, \beta, \gamma$ be the colors of a vertex $v$ of
  $T$.  If all the orientations of edges of $v$ are incoming and
  $\alpha+ \beta+ \gamma$ is not in $\Hr$ then $\Nr(T)=0$, i.e.
  \begin{equation}
    \label{eq:heights}
    \Nr\left(\epsh{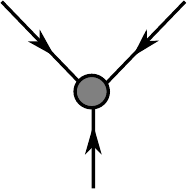}{6ex}
      \put(-25,7){\ms{\alpha}}\put(-4,7){\ms{\beta}}
      \put(-11,-5){\ms{\gamma}}\right)
    =0\text{ if }\alpha+\beta+\gamma\notin \Hr.
  \end{equation}
\item If $ T''$ denotes the connected sum of $T$ and $T'$ along
  an edge colored by $\alpha$ then $\Nr(T'')=\qd(\alpha)^{-1}\Nr(T)\Nr(T')$:
  \begin{equation}\label{eq:con-sum}
    \Nr\left(\epsh{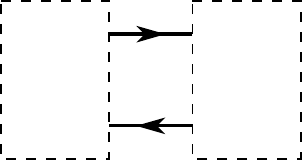}{6ex}\put(-43,0){$T$}\put(-12,0){$T'$}
      \put(-26,12){\ms{\alpha}}\put(-26,-10){\ms{\beta}}\right)
    =\delta_\alpha^{\beta}\,\qd(\alpha)^{-1}\,    
    \Nr\left(\epsh{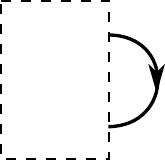}{6ex}\put(-21,0){$T$}\put(0,1){\ms{\alpha}}
      \,\right)
    \Nr\left(\epsh{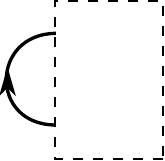}{6ex}\put(-12,0){$T'$}\put(-31,1){\ms{\alpha}}
      \,\right).
  \end{equation}
\item The invariant $\Nr$ has the following normalizations:
  \begin{equation}\label{eq:unknot-Theta}
    \Nr\left(\epsh{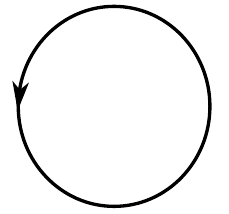}{6ex}\put(-31,3){\ms{\alpha}}\right)=\qd(\alpha),
    \qquad \Nr\left(\epsh{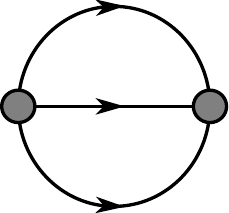}{6ex}\right)=1,
    \qquad \Nr\left(\epsh{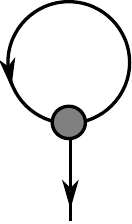}{6ex}\right)=0,
  \end{equation}
  here we assume that the ``$\Theta$'' graph is colored with any
  coloring which is not as in \eqref{eq:heights}.
\item If $T''$ denotes the connected sum of $T$ and $T'$ along a
  vertex with compatible incident colored edges then
  $\Nr(T'')=\Nr(T)\Nr(T')$:
  \begin{equation}\label{eq:con-sumv}
    \Nr\left(\epsh{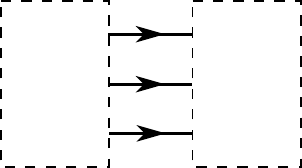}{6ex}\put(-42,0){$T$}\put(-12,0){$T'$}
    \right)=  
    \Nr\left(\epsh{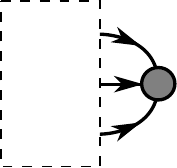}{6ex}\put(-23,0){$T$}
    \,\right)
    \Nr\left(\epsh{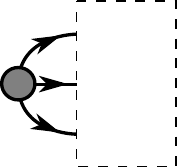}{6ex}\put(-12,0){$T'$}
    \,\right).
  \end{equation}
\item \label{I:zeroSplitGraphs} $\Nr$ is zero on split graphs:
  $\Nr(T\sqcup T')=0$.
\item The following relations hold whenever all appearing colors are
  in $\srcol$:
   \begin{equation}
    \label{eq:twist}
    \Nr\left(\,\epsh{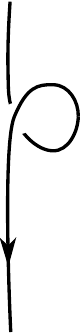}{8ex}\put(-13,-7){\ms{\alpha}}\right)=
    \qr^{\frac{\alpha^{2}-(\ro-1)^{2}}2}\Nr\left(\epsh{fig10}{8ex}
      \put(-11,3){\ms{\alpha}}\right),
  \end{equation}
  \begin{equation}
    \label{eq:twistv}
    \Nr\left(\,\epsh{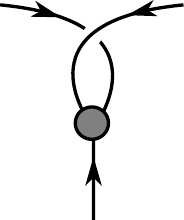}{6ex}
      \put(-21,16){\ms{\alpha}}\put(-7,17){\ms{\beta}}
      \put(-9,-7){\ms{\gamma}}\right)
    =q^{\frac{\gamma^2-\alpha^2-\beta^2+(\ro-1)^2}4}\Nr\left(\epsh{fig8}{6ex}
      \put(-25,7){\ms{\alpha}}\put(-4,7){\ms{\beta}}
      \put(-11,-5){\ms{\gamma}}\right),
  \end{equation}
  \begin{equation}
    \label{eq:hopf}
   \Nr\left(\,\epsh{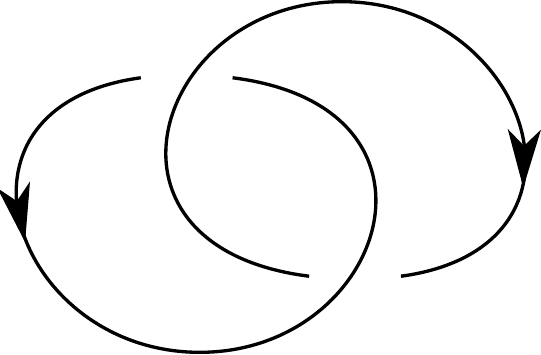}{6ex}
      \put(-36,-2){\ms{\alpha}}\put(-7,2){\ms{\beta}}\right)
    =(-1)^{r-1}rq^{\alpha\beta},
  \end{equation}
  \begin{equation}\label{eq:fusion}
    \Nr\left(\epsh{fig10}{8ex}\put(-11,3){\ms{\alpha}}
      \epsh{fig10}{8ex}\put(-11,3){\ms{\beta}}\right)=
    \sum_{\gamma\in\alpha+\beta+\Hr}\qd(\gamma)\Nr\left(
      \epsh{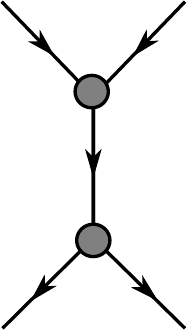}{8ex}\put(-16,3){\ms{\gamma}}\,\right),
  \end{equation}
  \begin{equation}\label{eq:6j}
    \Nr\left(\epsh{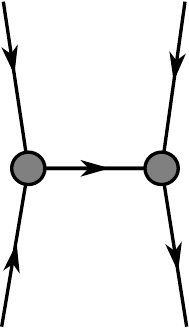}{8ex}
      \put(-13,4){\ms{j_3}}\put(-18,13){\ms{j_2}}\put(-8,13){\ms{j_4}}
      \put(-18,-8){\ms{j_1}}\put(-8,-8){\ms{j_5}}\,\right)=
    \sum_{j_6\in j_1+j_5+\Hr} \qd(j_6)
    \sjtop{j_1}{j_2}{j_3}{j_4}{j_5}{j_6}
    \Nr\left(\epsh{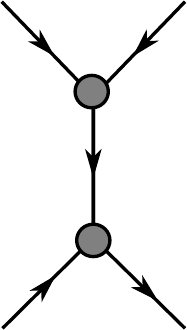}{8ex}
      \put(-18,3){\ms{j_6}}\put(-21,13){\ms{j_2}}\put(-4,13){\ms{j_4}}
      \put(-21,-8){\ms{j_1}}\put(-4,-8){\ms{j_5}}\,\right).
  \end{equation}
\end{enumerate}
\end{subequations}
Here the $6j$-symbol
$\sjtop{j_1}{j_2}{j_3}{j_4}{j_5}{j_6}=\Nr\left(\epsh{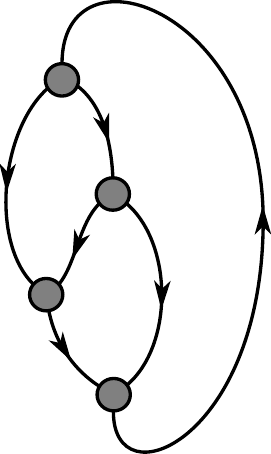}{12ex}
  \put(-36,10){\ms{j_1}}\put(-28,2){\ms{j_2}}\put(-28,-15){\ms{j_3}}
  \put(-12,-7){\ms{j_4}}\put(-18,14){\ms{j_6}}\put(-8,3){\ms{j_5}}\right)$
is given by:
\begin{eqnarray*}
\sjtop{j_1}{j_2}{j_3}{j_4}{j_5}{j_6}=
(-1)^{r-1+B_{165}} \, 
\dfrac{\{B_{345}\}! \, \{B_{123}\}!}
{\{B_{246}\}! \, \{B_{165}\}!} \, 
\qbin{j_3+r-1}{A_{123}+1-r} \, 
{\qbin{j_3+r-1}{B_{354}}}^{-1}\times 
\\
\times\sum_{z = m}^{M}\,
(-1)^z\qbin{A_{165}+1}{j_5+z+r} \, 
\qbin{B_{156}+z}{B_{156}} \,
\qbin{B_{264}+B_{345}-z}{B_{264}} \,
\qbin{B_{453}+z}{B_{462}}
\end{eqnarray*}
where $A_{xyz}=\frac{j_x+j_y+j_z+3(r-1)}{2}$,
$B_{xyz}=\frac{j_x+j_y-j_z+r-1}{2}$,
$m=\rm{max}(0,\frac{j_3+j_6-j_2-j_5}{2})$ and
$M=\rm{min}(B_{435},B_{165})$.

\begin{rem}
  It can be shown that the above axioms determine the value of $\Nr(T)$ for
  any $T\in \cG$ (cf \cite[Proposition 4.6]{GP0} for a similar proof).  In
  particular, the axioms can be used to reduce $\Nr(T)$ to a linear
  combination of $6j$-symbols which are in turn determined by the above
  formula.  In Subsection \ref{SS:NfromSl2} we will prove that these axioms
  are consistent.
\end{rem}
\begin{rem}\label{R:admcol}
  The $\srcol$-coloring $c$ of $T\in\cG$ can be seen as a complex
  $1$-chain.  Its boundary $\delta c$ is a $0$-chain, i.e. a map from
  the set $\V$ of trivalent vertices of $T$ to $\C$.  If $\delta c$
  has a value in $\C\setminus\Hr$, then $\Nr(T)=0$ because of Axioms
  \eqref{eq:ch-orient} and   \eqref{eq:heights}.  For a trivalent graph
  $T$ and  map $\delta:\V\to\Hr$ let $C(T,\delta)$ be the set of
  colorings of $T$ with boundary $\delta$.  Then $C(T,\delta)$ is a
  Zariski open subset of an affine subspace of
  $\C^{\{\text{edges(T)}\}}$.  As a function on
  $C(T,\delta)$, $\Nr$ is holomorphic.  Finally, if the number of
  edges of $T$ is greater or equal 2 then $\Nr$ extends continuously
  to the closure of $C(T,\delta)$.
\end{rem}
\begin{rem}
  The formulas for the $6j$-symbols above have been computed in
  \cite{CM}.  In particular, one can recover these symbols from
  Formula 1.16 of \cite{CM} where a color $a\in \mathbb{C}$ in
  \cite{CM} corresponds to the color $2a+1-r$ above.  In the case when
  $r$ is odd, these $6j$-symbols are computed in \cite{GP} but have a
  different normalization.  There are several choices which give
  different normalizations of this invariant.  One of them concerns
  the choice of $\qd$: For any complex number $c$, one can replace
  $\qd$ with $\wt\qd=c\qd$ and define $\wt\Nr$ on a non empty graph
  $T$ with $v$ vertices by $\wt\Nr(T)=c^{1-v/2}\Nr(T)$.  Then $\wt\Nr$
  satisfies an equivalent set of axioms.
\end{rem}

We consider the following computation which will be useful later.  Let
$\epsilon=1$ if $r$ is even and $0$ otherwise and let $\sigma$ be a formal
linear combination of colors, $\sigma=\sum_{k\in
  \Hr}\qd(\alpha+k)[\alpha+k]$.  For a graph with an edge colored
by $\sigma$, we formally expand such a color as
\begin{align*}
  \Nr\left(\epsh{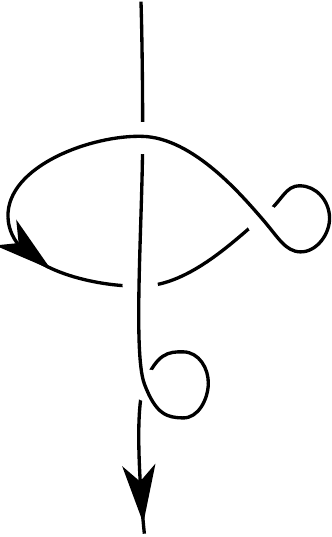}{12ex}\put(-33,-2){\ms{\sigma}}
    \put(-16,-20){\ms{\alpha+\epsilon}}\ \right)&= \sum_{k\in
    \Hr}\qd(\alpha+k)\Nr\left(\epsh{fig1}{12ex}\put(-37,-3){\ms{\alpha\!+\!k}}
    \put(-16,-20){\ms{\alpha+\epsilon}}\ \right)\\
  &=\sum_{k\in \Hr}\qd(\alpha+k)(-1)^{r-1}r\left(\qr^{(\ro-1)^{2}}
    \qr^{-\frac12(\alpha+\epsilon)^{2}}\qr^{-\frac12(\alpha+k)^{2}}\right)
  \qr^{(\alpha+\epsilon)(\alpha+k)}\qd(\alpha+\epsilon)^{-1}
  \Nr\left(\epsh{fig10}{8ex}\put(-3,3){\ms{\alpha+\epsilon}}\  \ \ \right)\\
  &=\sum_{k\in
    \Hr}(-1)^{r-1}r\frac{\qn{\alpha+k}}{\qn{\alpha+\epsilon}}\qr^{(\ro-1)^{2}}
  \qr^{\epsilon k-\frac12 k^{2}-\frac12 \epsilon^{2}}
  \Nr\left(\epsh{fig10}{8ex}\put(-3,3){\ms{\alpha+\epsilon}}\ \ \ \right)\\
  &=\frac{(-1)^{r-1}r\qr^{(\ro-1)^{2}-\frac12
      \epsilon^2}}{\qn{\alpha+\epsilon}} \left(\qr^{\alpha}\sum_{k\in
      \Hr}\qr^{\epsilon k-\frac12 k^{2}+k} -\qr^{-\alpha}\sum_{k\in
      \Hr}\qr^{\epsilon k-\frac12 k^{2}-k}\right)
  \Nr\left(\epsh{fig10}{8ex}\put(-3,3){\ms{\alpha+\epsilon}}\ \ \ \right)\\
  &=\frac{(-1)^{r-1}r\qr^{(\ro-1)^{2}-\frac12
      \epsilon^2}}{\qn{\alpha+\epsilon}} \left(\qr^{\alpha} S_+
    -\qr^{-\alpha} S_- \right)
  \Nr\left(\epsh{fig10}{8ex}\put(-3,3){\ms{\alpha+\epsilon}}\ \ \
  \right).
\end{align*}
where $S_\pm=\sum_{k\in \Hr} \qr^{\epsilon k-\frac12 k^2\pm k}$. Setting $l=\frac{k+(r-1)}{2}$ one has $$S_\pm=\sum_{l=0}^{\ro-1}\qr^{\epsilon(2l-(r-1))-2l^2-\frac12 (\ro-1)^2+2(\ro-1)l\pm (2l-(r-1))}=q^{-\frac12 (\ro-1)^2\mp (\ro-1)-\epsilon(r-1)}\sum_{l=0}^{\ro-1}\qr^{-2l^{2}+(2\epsilon-2\pm 2)l}$$ which is a generalized quadratic Gauss sum and
is computed for example in \cite[Chapter 1]{BEW}.  Since $-l^2+\epsilon l=-(l-1)^2-2(l-1)-1+\epsilon(l-1)+\epsilon$ and $q=\exp(\frac{i\pi}{r})$ then
$$S_+=q^{-\frac12 (r-1)^2-(r-1)-\epsilon(r-1)}\sum_{l=0}^{\ro-1}\qr^{-2l^{2}+2\epsilon l}=q^{-\frac12 (r-1)^2+(r+1)-\epsilon(r-1)}\sum_{l=0}^{\ro-1}\qr^{-2(l-1)^{2}-4(l-1)-2+2\epsilon( l-1)+2\epsilon}$$
The last quantity equals $q^{2\epsilon}S_-$, thus we have:
$$\Nr\left(\epsh{fig1}{12ex}\put(-33,-2){\ms{\sigma}}
  \put(-16,-20){\ms{\alpha+\epsilon}}\right)=
(-1)^{\ro-1}\ro\qr^{(\ro-1)^{2}-\frac12
  \epsilon^2}\left(\frac{q^{\alpha+2\epsilon}-q^{-\alpha}}
  {\qn{\alpha+\epsilon}}S_-\right)\
\Nr\left(\epsh{fig10}{8ex}\put(-3,3){\ms{\alpha+\epsilon}}\ \ \
\right)
=\Delta_-\Nr\left(\epsh{fig10}{8ex}\put(-3,3){\ms{\alpha+\epsilon}}\ \
  \ \right)$$ where \begin{equation}\label{eq:delta}
  \Delta_-=\left\{
  \begin{array}{lll}
    0&&\text{ if $\ro\equiv0$ mod }4\\
    i(r\qr)^{\frac32}&&
    \text{ if $\ro\equiv1$ mod }4\\
  (i-1)(r\qr)^{\frac32}&&
    \text{ if $\ro\equiv2$ mod }4\\ 
    - (r\qr)^{\frac32}&&
    \text{ if $\ro\equiv3$ mod }4\\
  \end{array}\right.\end{equation}
  
Similarly,
$$\Nr\left(\epsh{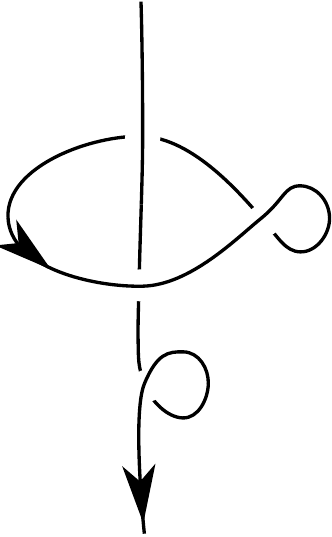}{12ex}\put(-33,-2){\ms{\sigma}}
    \put(-16,-20){\ms{\alpha+\epsilon}}\right)
  =\Delta_+\Nr\left(\epsh{fig10}{8ex}\put(-3,3)
    {\ms{\alpha+\epsilon}}\ \ \ \right)$$
where $\Delta_+$ is the complex conjugate of $\Delta_-$.

\subsection{Cohomology classes}\label{SS:CohomGcoloring}
Here we give a characterization of a cohomology group which will be
used throughout this paper.
Let $(\Gr,+)$ be an abelian group.  Let $M$ be a compact connected
oriented 3--manifold and $T$ a framed graph in $M$.  Let $L$ be an
oriented framed link in $S^3$ which represents a surgery presentation
of $M$.  Consider the cohomology group $H^{1}(M,\Gr)\simeq
\Hom(H_1(M,\Z),\Gr)$.
The meridians $\{m_i\}_{i=1\ldots n_L}$ of the components of $L$
generate $H_1(M,\Z)$ and their relations are given in these generators
by the columns of the symmetric linking matrix $\lk=(\lk_{ij})$.
Consequently, $H^{1}(M,\Gr)=\{(\phi_i)\in \Gr^{n_L}: \lk.\phi=0\}$.

As we will now explain the cohomology group of $M\setminus T$ can be
described in terms of the homology of $L\cup T$.  Let $e_1,\ldots,
e_{n_L}$ be the oriented edges of $L$, $e_{n_L+1},\ldots, e_{n_L+n_T}$
be the oriented edges of $T$ and $m_i$ be the oriented meridian of
$e_i$.  We have $H^{1}(M\setminus T,\Gr)\simeq \Hom(H_1(M\setminus
T,\Z),\Gr)$ where $H_1(M\setminus T,\Z)$ is generated by the meridians
of all edges of $L\cup T$. Given a regular planar projection of $L\cup
T$, we can define a linking matrix for the $n_L+n_T$ edges of $L\cup
T$ (which is not an isotopy invariant) as follows: for $i,j \in
\{1,\ldots, {n_L+n_T}\}$ let $\lk_{ij}$ be the algebraic number of
crossings between the edges $e_i$ and $e_j$ with the edge $e_j$ above
the edge $e_i$.  Graphically, this can be represented by
$$\lk_{ij}\left(\epsh{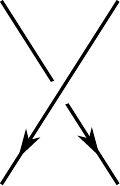}{6ex}
  \put(-20,-5){\ms{e_j}}\put(-3,-5){\ms{e_i}}\right)=1\qquad
\lk_{ij}\left(\epsh{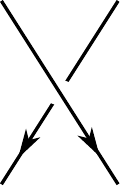}{6ex}\put(-20,-5){\ms{e_i}}\put(-3,-5){\ms{e_j}}
\right)=-1\qquad
\lk_{ij}\left(\epsh{fig26}{6ex}\put(-20,-5){\ms{e_i}}\put(-3,-5){\ms{e_j}}
\right)=\lk_{ij}\left(\epsh{fig27}{6ex}
  \put(-20,-5){\ms{e_j}}\put(-3,-5){\ms{e_i}}\right)=0$$
Thus we may present:
$$H_1(M\setminus T,\mathbb{Z})=\langle[m_i]| \sum_{j=1}^{n_L+n_T} 
\lk_{ij}[m_j]=0\ {\rm{and}}\ r_v=0\ {\text{\rm{for all}}} \ i\in
\{1,\ldots n_L\}\ {\rm{and}}\ v\rangle$$ 
where $v$ ranges over all the vertices of $T$ and $r_v$ is the sum of
the meridians of the edges incoming to $v$ minus the sum of the
meridians of the edges outgoing from $v$.

Let 
\begin{equation}\label{eq:Phi}
\Phi:H^{1}(M\setminus T,\Gr)\to H_1(L\cup
T,\Gr)=H_1(L,\Gr) \oplus H_1(
T,\Gr)
\end{equation} 
be the map sending a cohomology class $\coh$ to the chain $\sum_i
\coh(m_i) e_i$; clearly $\Phi$ is injective.  A cycle $x=\sum_i x_i
e_i $ representing a class in $H_1(L\cup T,{\Gr})$ is in $\operatorname{Im}(\Phi)$ if
and only if $\sum_{j=1}^{n_L+n_T} \lk_{ij} x_j=0\in\Gr,$ for all $i\in
\{1,\ldots n_L\}$.

For $\coh\in H^{1}(M\setminus T,\Gr)$ and $L\cup T$ as above, the map
$g_\coh$ defined on the set of edges of $L\cup T$ with values in $\Gr$
defined by $g_\coh(e_i)=\coh(m_i)$ is called the \emph{$\Gr$-coloring}
of $L\cup T$ induced by $\coh$.

\subsection{The 3--manifolds invariants $\Nr$ and $\Nr^0$}\label{SS:sl2NandN^0}
In the rest of this section we use the notation of the last subsection
with the additive group $\Gr=\C/2\Z$.  We also assume that the integer
$r$ is not congruent to $0$ mod $4$.  Let $\X=\Z/2\Z\subset \C/2\Z$.
Let $M$ be a compact, connected oriented 3--manifold and $T$ a framed
trivalent graph in $M$ whose edges are colored by elements of
$\srcol$.  Let $\coh\in H^{1}(M\setminus T,\C/2\Z)$.  If $M$ is
presented as an integral surgery over a link $L$ in $S^3$ then $\coh$
induces a $\C/2\Z$-coloring $g_\coh$ on $L\cup T$.  We say that
$(M,T,\coh)$ is {\em a compatible triple} if for each edge $e$ of $T$
its $\srcol$-color $c(e)$ satisfies $c(e)+r-1\equiv g_\coh(e)\ {\it
  mod}\ (2\Z)$.  Note that this definition does not depend on the
surgery presentation of $M$.  A surgery presentation via $L\subset
S^3$ for a compatible tuple $(M,T,\coh)$ is \emph{computable} if one
of the following two conditions holds: (1) $g_\coh(e) \in \C/2\Z
\setminus \X $ for all edges $e$ of $L$, or (2) $L=\emptyset$ and
$T\neq \emptyset$.
 
\newcommand{\Span}{\operatorname{Span}}
Recall the set $\Hr=\{1-r,3-r,\ldots,r-1\}$ defined in \eqref{eq:Hr}.
For $\alpha\in \C\setminus\Z$ we define the Kirby color
$\Omega_{{\alpha}}\in \Span_\C<[x] \;|\;x\in \C>$ by
\begin{equation}
  \label{eq:Om}
  \Omega_{{\alpha}}=\sum_{k\in \Hr}\qd(\alpha+k)[\alpha+k].
\end{equation}
If $\wb{\alpha}$ is the image of $\alpha$ in $\C/2\Z$ we say
that $\Omega_{{\alpha}}$ has degree $\wb{\alpha}$.  We can ``color'' a
knot $K$ with a Kirby color $\Omega_{{\alpha}}$: let
$K({\Omega_{{\alpha}}})$ be the formal linear combination of knots
$\sum_{k\in \Hr} \qd(\alpha+k) K(\alpha+k)$ where $K(\alpha+k)$ is the
knot $K$ colored with $\alpha+k$.  If
$\wb{\alpha}\in\C/2\Z\setminus\Z/2\Z$, by $\Omega_{\wb{\alpha}}$, we
mean any Kirby color of degree $\wb\alpha$.

We will prove all of the theorems and propositions of this section in Section
\ref{S:ProofsOfNrManInv}.
 
\begin{teo}\label{T:sl2CompSurgInv}
  If $L$ is a link which gives rise to a computable surgery
  presentation of a compatible triple $(M,T,\coh)$ then
  $$\Nr(M,T,\coh)=\dfrac{\Nr(L\cup T)}{\Delta_+^{p}\ \Delta_-^{s}}$$
  is a well defined topological invariant (i.e. depends only of the
  homeomorphism class of the triple $(M,T,\coh)$), where $(p,s)$ is
  the signature of the linking matrix of the surgery link $L$ and for
  each $i$ the component $L_i$ is colored by a Kirby color of degree
  $g_{\coh}(L_i)$.
\end{teo}

Next we show that computable surgery presentations exist in several
situations.  To do this we need the following definition.  The
inclusion $\Z/2\Z\subset\C/2\Z$ induces an injective map
$$H^1(M\setminus T,\Z/2\Z)\cong\Hom(H_1(M\setminus
T,\Z),\Z/2\Z)\hookrightarrow\Hom(H_1(M\setminus T,\Z),\C/2\Z)\cong
H^1(M\setminus T,\C/2\Z).$$ We say that a cohomology class $\coh\in
H^1(M\setminus T,\C/2\Z)$ is {\em integral} if it is in the image of
this map.

\begin{prop}\label{P:ExistCompSurgSl2}
  Let $(M,T,\coh)$ be a compatible triple where the cohomology class
  $\coh$ is not integral.  Then there exists a surgery presentation of
  $(M,T,\coh)$ which is computable.  In particular, the triple
  $(M,T,\coh)$ has a computable surgery presentation if $T\neq
  \emptyset$ and $T$ has an edge whose color is in $\C\setminus \Z$.
\end{prop}
The proposition (whose proof is constructive) implies that  when $\coh$
is not integral one may construct a computable presentation of
$(M,T,\coh)$ and hence apply Theorem \ref{T:sl2CompSurgInv}. 
The proposition does not apply to the following case: when $\coh$ is integral,
$T$ is non-empty and all the edges of $T$ are colored by a multiple of
$r$.  The following construction allows one to extend the definition of $\Nr$ to
this case and provides a method for computing $\Nr$ 
when computable presentations exist.  With this in mind, we consider
the situation when $T\neq \emptyset$.
\begin{defi}[$H$-Stabilization]
  Recall the notation $\Xr=\Z\setminus r\Z \subset \C$.  Let
  $H(\alpha,\beta)$ be a long Hopf link in $\mathbb{R}^3$ whose circle
  component is colored by $\alpha\in \srcol$ and whose long component
  is colored by $\beta \in \C \setminus \Xr$.  Let $(M,T,\coh)$ be a
  compatible triple where $T\neq \emptyset$, $e$ is an edge of $T$
  colored by $\beta$ and $m$ is the meridian of $e$.  A
  $H$-\emph{stabilization} of $(M,T,\coh)$ along $e$ is a compatible
  triple $(M,T_H,\coh_H)$ such that
  \begin{itemize}
  \item $T_H=T\cup m$ and $m$ is colored by $\alpha \in \C \setminus
    \Xr$,
  \item $\coh_H$ is the unique element of $H^1(M\setminus (T\cup
    m);\C/2\Z)$
    such that $\coh_H(m)=\wb{\alpha+r-1}$ is equal to the
    image of $\alpha+r-1$ in $\C/2\Z$ and $(\coh_H)|_{M\setminus
      (T\cup D)}=\coh$ where $D$ is a disc bounded by $m$ and
    intersecting once $e$.
  \end{itemize}
\end{defi}

\begin{teo}\label{T:T_adm-Nr}
  If $(M,T,\coh)$ is a compatible triple and $T\neq \emptyset$ then
  there exists a $H$-stabilization of $(M,T,\coh)$ admitting a
  computable surgery presentation.  Let $(M,T_H,\coh_H)$ be such a
  $H$-stabilization and let $L$ be a link which gives rise to a
  computable surgery presentation of $(M,T_H,\coh_H)$ then
  $$\Nr(M,T,\coh)=\dfrac{\qd(\beta)\Nr(L\cup T_H)}
  {(-1)^{\ro-1}rq^{\alpha\beta}\Delta_+^{p}\ \Delta_-^{s}}$$ 
  is a well defined topological invariant (i.e. depends only of the
  homeomorphism class of the triple $(M,T,\coh)$), where a component
  $L_i$ of $L$ is colored by a Kirby color of degree $g_{\coh}(L_i)$
  and $(p,s)$ is the signature of the linking matrix of the surgery
  link $L$.  Moreover, if an edge of $T$ is colored by an element of
  $\C\setminus \Z$ then the above quantity coincides with the
  invariant of Theorem \ref{T:sl2CompSurgInv}.
\end{teo}

Theorems \ref{T:sl2CompSurgInv} and \ref{T:T_adm-Nr} do not apply when
$\coh$ is integral and $T$ is empty. To cover this case we define a
second invariant $\Nr^0$.  This invariant naturally completes the
definition of $\Nr$ in the following sense: $\Nr^0(M,T,\coh)$ is
defined for any triple $(M,T,\coh)$ but it is zero if a computable
surgery presentation of $(M,T,\coh)$ exists.  To do so, we define the
connected sum of two triples.

Let $(M_1, T_1, \coh_1)$ and $(M_2, T_2, \coh_2)$ be compatible
triples. 
Let  $M_3=M_1\#M_2$ is the connected sum along balls not intersecting $T_1$ and
$T_2$.  Let $T_3=T_1\sqcup T_2$.  For $i=1,2$, let $B_i^3$ be a 3--ball in
$M_i\setminus T_i$.  Then we have  
$$
H_1(M_3\setminus T_3;\Z)\cong H_1(M_1\setminus (B_1^3\sqcup
T_1);\Z)\oplus H_1(M_2\setminus (B_2^3\sqcup T_2);\Z)\cong
H_1(M_1\setminus T_1;\Z)\oplus H_1(M_2\setminus T_2;\Z)
$$ 
where the first isomorphism is induced by a Mayer-Vietoris sequence
and the second isomorphism comes from excision.  These maps induce an
isomorphism
$$
H^1(M_3\setminus T_3,\C/2\Z)\to H^{1}(M_1\setminus T_1,\C/2\Z)\oplus
H^{1}(M_2\setminus T_2,\C/2\Z).
$$
Let $\coh_3$ be the unique element of $H^{1}(M_3\setminus T_3,\C/2\Z)$
such that $\coh_3$ restricts through the above isomorphism to both
$\coh_1$ and $\coh_2$.  Define the connected sum of $(M_1, T_1,
\coh_1)$ and $(M_2, T_2, \coh_2)$ as
\begin{equation}
  \label{eq:consum}
  (M_1, T_1, \coh_1)\#(M_2, T_2, \coh_2)=(M_3,T_3,\coh_3).
\end{equation}

For $\alpha \in \srcol$, let $u_\alpha$ be the unknot in $S^3$ colored
by $\alpha$.  Let $\coh_\alpha$ be the unique element of
$H^1(S^3\setminus u_\alpha;\C/2\Z)$ such that
$(S^3,u_\alpha,\coh_\alpha)$ is a compatible triple.  Remark that
$\Nr(S^3,u_\alpha,\coh_\alpha)=\qd(\alpha)$.
\begin{teo}\label{T:N0sl2} Let $(M,T,\coh)$ be a compatible triple.  
  Define
  $$
  \Nr^{0}(M,T,\coh)=\frac{\Nr((M,T,\coh)\#(S^3,u_\alpha,\coh_\alpha))}
  {\qd(\alpha)}.
  $$ 
  Then $\Nr^{0}(M,T,\coh)$ is a well defined topological invariant
  (i.e. depends only of the homeomorphism class of the compatible
  triple $(M,T,\coh)$).  Moreover, if $(M,T,\coh)$ or a
  $H$-stabilization of $(M,T,\coh)$ has a computable surgery
  presentation then $\Nr^0(M,T,\coh)=0$.
\end{teo}

Notice that the above theorems can be used to define a possibly
non-zero invariant for every compatible triple $(M,T,\coh)$.  In
particular, if $T\neq \emptyset$ or if $\coh$ is not integral then the
invariant $\Nr$ is well defined (in this case $\Nr^0$ is defined but
zero), otherwise, $\Nr^0$ is well defined and possibly non-zero (for
example it is non-zero on the Poincar\'e sphere, see Subsection
\ref{SS:Non-trivPoincare}).
In Section \ref{S:ProofsOfNrManInv} we will prove all of the theorems of this
section.  In Section \ref{sec:examples} we give several examples that show the
invariants $\Nr$ and $\Nr^0$ are computable and non-trivial.  For other 
properties of these invariants see Section \ref{SS:MainResults}.

\section{Examples and applications}\label{sec:examples}

\subsection{Distinguishing lens spaces}\label{SS:DistLens}
In Remark 3.9 of \cite{Je} it was observed that the standard
Reshetikhin-Turaev 3--manifold invariants arising from quantum $\slt$
cannot distinguish the lens spaces $L(65,8)$ and $L(65,18)$.  In this
subsection, we show that for $r=3$ and $T=\emptyset$, the invariant
$\Nr$ can distinguish these two manifolds.  Thus, the invariants
defined in this paper are independent of the standard
Reshetikhin-Turaev 3--manifold invariants.  Interestingly, this result
depends on the cohomological data of the invariant.  Since it can be
difficult to compare two elements of the isomorphic spaces
$H^1(L(65,8);\C/2\Z)$ and $ H^1(L(65,18);\C/2\Z)$ we consider the sum
over the whole cohomological space (except the zero element):
$$
S_\ro(L(p,p'))=\sum_{\coh \in H^1(L(p,p');\C/2\Z),\ \coh\neq 0}
\Nr(L(p,p'),\emptyset,\coh)
$$
where this sum is finite since $H^1(L(p,p');\C/2\Z)$ is isomorphic to
$\mathbb{Z}/p\mathbb{Z}$.

Let $C(a_1,\ldots, a_n)$ be the chain link whose components are
oriented unknots $L_1,\ldots, L_n$ with framings $a_1,\ldots, a_n$
(i.e. $L_i$ and $L_{j}$ form a Hopf link if $j=i\pm 1$ and a split
link otherwise; see the two exemples of Figure \ref{F:Lens}) and let
\begin{figure}
 \begin{minipage}[b]{.4\linewidth}
  \centering
    $\epsh{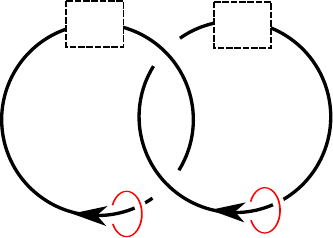}{15ex} \put(-29,25){-8}\put(-67,25){8} $
    {\color{red}\put(-64,-37){\ms{\frac{2k}{65}}}\put(-26,-37)
      {\ms{\frac{-16k}{65}}}}
  \\[4ex]{a. L(65,8)}
 \end{minipage} 
 \begin{minipage}[b]{.4\linewidth}
  \centering
  $\epsh{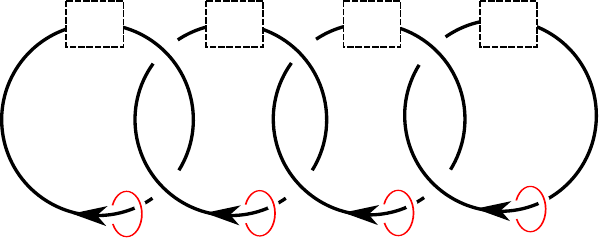}{15ex}  \put(-29,25){-3}\put(-63,25){2}\put(-100,25){3}
  \put(-139,25){4}$
    {\color{red}\put(-137,-37){\ms{\frac{2\ell}{65}}}\put(-102,-37)
      {\ms{\frac{-8\ell}{65}}}\put(-64,-37){\ms{\frac{22\ell}{65}}}
      \put(-26,-37){\ms{\frac{-36\ell}{65}}}}
  \\[4ex]{b. L(65,18)}
 \end{minipage}
 \caption{Surgery presentation of lens spaces.  A dashed box with an
   integer $n$ represents $n$ full positive twists.  The values in
   $\C/2\Z$ of a cohomology class on meridians of the link are shown
   in red. \label{F:Lens}}
\end{figure}
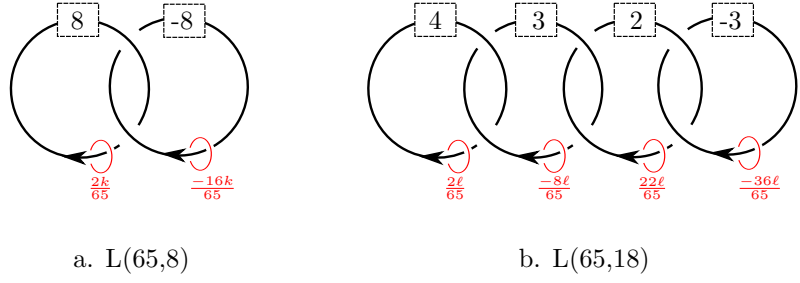
$m_i$ be the oriented meridian of $L_i$. It is a standard fact that
surgery on $C(a_1,\ldots, a_n)$ provides $L(p,p')$ when
$\frac{p}{p'}=a_1-\frac{1}{a_2-\cdots -\frac{1}{a_n}}$.  

Suppose that the components of $C(a_1,\ldots, a_n)$ are colored by the
complex numbers $\beta_1,\ldots,\beta_n$.  Using the axiom
\eqref{eq:twist} to simplify the framing, then axioms
\eqref{eq:con-sum}, \eqref{eq:hopf} one computes by induction
$$\Nr(C(a_1,\ldots, a_n))=\prod_{j=2}^{n-1}\qd(\beta_j)^{-1}
\prod_{j=1}^{n}q^{a_j\frac{\beta_j^2-(r-1)^2}2}\prod_{j=1}^{n-1}(-1)^{\ro-1}\ro q^{\beta_j\beta_{j+1}}$$

Let $\lk$ be
the linking matrix of $C(a_1,\ldots, a_n)$ with respect to the
components $L_i$.
Color $L_i$ with $\Omega_{\alpha_i}$ where $\alpha_i\in \mathbb{C}$ is such
that $\sum_{j} \lk_{ij}\wb{\alpha_j}=0\in\C/2\Z$ for all $i$.
The coloring $\{\wb\alpha_i\}_i$ represents the cohomology class $\coh\in
H^1(L(p,p');\C/2\Z)$ determined by
$\coh([m_i])=\wb{\alpha_i}$. 
Expanding the Kirby color $\Omega_{\alpha_i}$ defined by Equation
\eqref{eq:Om} one gets that $\Nr(L(p,p'),\emptyset,\coh)$ is equal to:
\begin{equation}\label{eq:chain}
  \sum_{k_1,\ldots k_n\in\Hr}\frac{d(\alpha_1+k_1)d(\alpha_n+k_n)
    \prod_{j=1}^{n}q^{a_i\p{\frac{(\alpha_i+k_i)^2-(\ro-1)^2}{2}}}
    \prod_{j=1}^{n-1}(-1)^{\ro-1}\ro q^{(\alpha_j+k_j)(\alpha_{j+1}+k_{j+1})}}
  {\Delta_+^{n_+}\Delta_-^{n_-}} 
\end{equation}
where $n_+$ (resp. $n_-$) is the number of positive (resp. negative)
eigenvalues of $\lk$.

Since $H_1(L(p,p');\Z)$ is isomorphic to $\mathbb{Z}/p\mathbb{Z}$
and $[m_1]$ is a generator of this homology group then for any
$\coh\in H^1(L(p,p');\C/2\Z)$ the value $\coh([m_1])\in \C/2\Z$ is an
integer multiple of $\frac{2}{p}$.  Conversely, if ${\alpha_1}\in
\C$ is any integer multiple of $\frac{2}{p}$ we can find a unique
cohomology class $\coh\in H^1(L(p,p');\C/2\Z)$ such that
$\coh([m_1])=\wb{\alpha_1}$.  Thus, elements of
$H^1(L(p,p');\C/2\Z)$ are in one to one correspondence with integer
multiples of $\frac{2}{p}$ in $\C/2\Z$.  Moreover, for $\coh\in
H^1(L(p,p');\C/2\Z)$ one can compute the values
$\wb{\alpha}_i=\coh([m_i])$ by solving the linear equations $\sum
lk_{i,j}\wb {\alpha_j}=0$ (in $\C/2\Z$) in terms of $\wb{\alpha}_1$.
Combining this fact with Equation \eqref{eq:chain} it is easy to use
computer algebra software to numerically compute $S_\ro(L(p,p'))$.  In
particular, one can show that $S_3(L(65,8))\neq S_3(L(65,18))$ (for
instance we use the surgery presentations over $C(8,-8)$ and
$C(4,3,2,-3)$, 
as in Figure \ref{F:Lens}).  
This was checked independently by the first and third authors.
 
%
%
\subsection{On the volume conjecture}\label{SS:VolCon}
In \cite{Co} the first author proved a version of the volume
conjecture for an infinite class of hyperbolic links called
``fundamental hyperbolic links''.  In this section we show that a
similar conjecture holds for the invariants introduced in this paper.
The structure of our proof follows very closely that of \cite{Co} but we are dealing with different invariants; in particular to relate the asymptotic behavior of our invariants to hyperbolic volumes we use the results proved in \cite{CM}. 
The main advantage of this new approach is that now one may formulate
a version of the volume conjecture for any link or knot in a
3--manifold using our invariants. This was not possible using the
approach of \cite{Co} where the fact that the ambient manifold was a
connected sum of $S^1\times S^2$ or a sphere was used in an essential way.

\newcommand{\Ga}{\Gamma}
\begin{defi}[Fundamental hyperbolic links]\label{def:fundlinks}
  A \emph{fundamental hyperbolic link} $L$ is a link contained in a
  3--manifold $M$ diffeomorphic to a connected sum of $k\geq 2$ copies
  of $S^2\times S^1$ and obtained by the following procedure:
\begin{enumerate}
\item let $\Ga$ be a graph in $S^3$ with $k-1$ four-valent vertices and
  let $T\subset {\Ga}$ be a maximal tree;
\item in a diagram of $\Gamma$ replace each vertex of ${\Ga}$ by the
  diagram\
  \raisebox{-0.3cm}{\includegraphics[width=0.7cm]{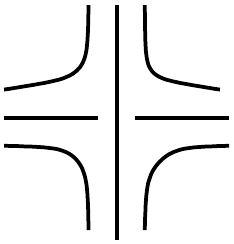}};
\item connect the new three-uples of boundary points (with any
  permutation) following the edges of ${\Ga}$ and let $L$ be the
  resulting link in $S^3$;
\item put $0$-framed meridians around each of the three-uple of
  strands passing along the edges of ${\Ga}\setminus T$;
\item perform surgery on the $0$-framed meridians to realize $L$ as a
  link in the connected sum of $k$ copies of $S^2\times S^1$.
\end{enumerate}
\end{defi}

A fundamental hyperbolic link has no natural orientation but such a
link may be equipped with a natural framing (see \cite{CT} for more
details).  However, for our purposes we can choose arbitrarily an
orientation and a framing.  Indeed, we will consider the norm of the
invariants of triples $(M,L,\coh_r)$ where the induced coloring on $L$
is $0$ and so by Axiom \eqref{eq:ch-orient} the choice of an
orientation is irrelevant and by Axiom \eqref{eq:twist} the framing
changes the value of the invariant by a multiplicative factor of the
form $q^{\pm\frac{(r-1)^2}{2}}$ and so does not change the norm of the
invariant.
\begin{figure}
\includegraphics[width=7cm]{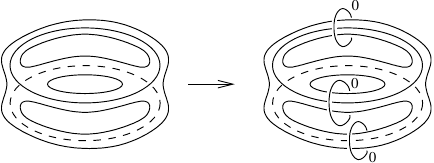}
\caption{An instance of construction of fundamental hyperbolic
  link.}\label{fig:link}
\end{figure}
\begin{example}
  Let ${\Ga}$ be the graph with $2$ vertices and $4$ edges connecting
  them.  Apply the procedure of Definition \ref{def:fundlinks} to
  $\Ga$.  In this procedure one can connect the three-uples of
  boundary points in such a way that one obtains the link in left side
  of Figure \ref{fig:link}. In the right side of the figure the
  0-framed meridians have been added.  After surgery the resulting
  link is contained in the connected sum of $3$ copies of $S^2\times
  S^1$.
\end{example}

The following theorem was proved in \cite{CT}.

\begin{teo}[\cite{CT}]\label{teo:CT}
  Let $L$ be a fundamental hyperbolic link in $M=\#_{k} S^2\times
  S^1$.  Then $M\setminus L$ admits a complete hyperbolic structure
  whose volume is
  $$2(k-1){\rm VolOct}=16(k-1) \Lambda(\frac{\pi}{4})$$
  where ${\rm VolOct}$ is the volume of a regular ideal octahedron and
  $\Lambda(x)=\int_{0}^{x} -\log|2\sin(t)|dt$ is the Lobatchevsky
  function.  Moreover, for any link $T$ in an oriented 3--manifold $N$
  there exists a fundamental hyperbolic link $L=L'\cup L''\subset \#_k
  S^2\times S^1$ (for some $k$) such that the result of an integral
  surgery on $L'$ is the manifold $N$ and the image of $L''$ is $T$.
\end{teo}

The following theorem is the main result of this subsection.
\begin{teo}[A version of the Volume Conjecture]\label{T:OurVolCon}
  Let $L$ be a fundamental hyperbolic link in $M= \#_k S^2\times S^1$
  colored by $0\in\srcol$.  For each odd integer $r$, let $\coh_r\in
  H^1(M\setminus L;\C/2\Z)$ be a cohomology class such that its associated
  $\C/2\Z$-coloring satisfies $g_{\coh_r}(m_i)=0,$ for all $i$ (where
  $m_i$ is the oriented meridian of the $i^{th}$-component of $L$).
  Then
  $$\lim_{
    \begin{array}{c}
      r\to \infty\\r\  \text{odd}
    \end{array}
  } \frac{2\pi}{r}\log(|\Nr(M,L,\coh_r)|) =Vol(M\setminus L).
  $$
\end{teo}

Before we prove Theorem \ref{T:OurVolCon} we state the following
proposition whose proof is given in Subsection~\ref{S:NrMsl2}.
\begin{prop}\label{prop:encircling}
  Let $r$ be an odd integer and $a\in \C\setminus\Z$.  Let
  $\alpha,\beta,\gamma\in \srcol$ such that $\alpha+\beta+\gamma=0$.
  Then
  \begin{equation}\label{eq:encircling}
    \Nr\left(\epsh{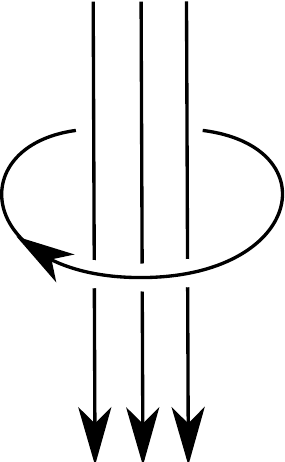}{8ex}\put(-25,-4){\ms{\Omega_{a}}}
      \put(-17,-18){\ms{\alpha}}\put(-13,-18){\ms{\beta}}
      \put(-9,-18){\ms{\gamma}}\right)=r^3 \Nr\left(\epsh{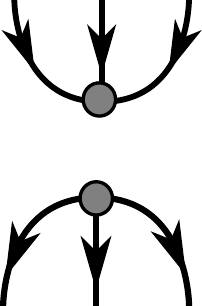}{6ex}\right).
  \end{equation}
\end{prop}

\begin{proof}[Proof of Theorem \ref{T:OurVolCon}]
  Start with a presentation of $(M,L)$ as in Definition
  \ref{def:fundlinks} and let $T,{\Ga}$ be as in that definition.  We
  assume first that this presentation is computable.  First, one has
  $\Nr(M,L,\coh_r)= r^{3k} \Nr(\tilde{{\Ga}})$ where $\tilde{{\Ga}}$
  is the trivalent graph obtained by applying Proposition
  \ref{prop:encircling} to each three-uple of strands in $L$ encircled
  by a $0$-framed meridian.  In particular $\tilde{\Ga}$ contains
  twice as many vertices as there are $0$-framed meridians
  and all edges of $\tilde{\Ga}$ are colored by $0$. One can realize
  $\tilde{\Ga}$ in $S^3$ so that it coincides with the diagram of $L$
  except near the $0$-framed meridians.  The maximal tree $T$ has
  $k-1$ $4$-valent vertices and $k-2$ internal edges.  Now, every
  internal edge of $T$ corresponds to three strands of $\tilde{\Ga}$
  connecting two disjoint subgraphs of $\tilde{\Ga}$ as in the left
  hand side of \eqref{eq:con-sumv}.  Applying Axiom
  \eqref{eq:con-sumv} to all the three-uple of strands of
  $\tilde{\Ga}$ corresponding to internal edges of $T$ (i.e. those not
  encircled by $0$-framed meridians) one gets that
  $\Nr(\tilde{{\Ga}})$ is equal to the product of the values of $\Nr$
  on $k-1$ tetraedral graphs (one for each vertex of $T$).  Axioms
  \eqref{eq:twist} and \eqref{eq:twistv} imply that the modulus of
  $\Nr$ of these tetraedral graphs is independent of the framing and
  of the cyclic ordering at its vertices, and thus is equal to the
  modulus of $\sjtop{0}{0}{0}{0}{0}{0}_r$.  Hence we get
  $$|\Nr(M,L,\coh_r)|=r^{3k}\left|\sjtop{0}{0}{0}{0}{0}{0}_r\right|^{k-1}.$$
  To prove the final statement let's use Lemma 1.15 of \cite{CM} in
  the case when $a=b=c=d=e=f=\frac{r-1}{2}$ (in our notation all the
  colors correspond to $0$).  Then we have:
  $$\sjtop{0}{0}{0}{0}{0}{0}_r=(-1)^{r-1}\sum_{z=0}^{\frac{r-1}{2}}
  \qbin{\frac{r-1}{2}}{z}
  \qbin{\frac{r-1}{2}+z}{\frac{r-1}{2}}^2\qbin{r-1-z}{\frac{r-1}{2}}=$$
  $$=(-1)^{r-1}\sum_{z=0}^{\frac{r-1}{2}}\qbin{\frac{r-1}{2}}{z}
  \qbin{\frac{r-1}{2}}{\frac{r-1}{2}-z}^2\qbin{\frac{r-1}{2}}{z}
  =r^{2}\sum_{z=0}^{\frac{r-1}{2}}(\{z\}!\{\frac{r-1}{2}-z\}!)^{-4}$$
  Where we used the equalities $$\{a\}!\{r-1-a\}!=\sqrt{-1}^{r-1}r,\
  {\rm and}\ \qbin{a}{b}=\qbin{r-1-b}{r-1-a}$$ which hold whenever $a,
  b, a-b\in \{0,\ldots r-1\}$.  It is now a standard analysis to check
  that the summands are all positive real numbers, that the term
  growing faster is that corresponding to
  $z=\lfloor\frac{r-1}{4}\rfloor$ and that its growth rate is
  $\exp(\frac{8r}{\pi}\Lambda(\frac{\pi}{4}))$. One concludes by
  Theorem \ref{teo:CT}.  Finally, the proof when the presentation is
  not computable, that is when one of the $\C/2\Z$ color of the framed
  meridian is $0$ or $1$, can be deduced by continuity from Subsection
  \ref{S:hol}.
\end{proof}

\begin{question}[A version of the Volume Conjecture for links in manifolds]\label{C:GVolConj}
  Let $L$ be a $0$-colored link in a compact, oriented 3--manifold
  $M$ such that $M\setminus L$ has a complete hyperbolic metric with
  volume $\operatorname{Vol}(M\setminus L)$.
  For each odd integer $r$, let $\coh_r\in H^1(M\setminus L;\C/2\Z)$
  be the zero cohomology class.  Does the equality
  $$\lim_{r\to \infty} \frac{2\pi}{r}\log|\Nr(M,L,\coh_r)|=
  \operatorname{Vol}(M\setminus L)$$
  hold?
\end{question}
The above question includes the usual volume conjecture which
corresponds to the case $M=S^3$.  Indeed, $\Nr(S^3,L,\coh)$ is the ADO
invariant of $L$ (for a link in $S^3$ or in a homology sphere, the
$\C$ coloring of the link uniquely determines a cohomology class
$\coh\in H^1(M\setminus L;\C/2\Z)$
such that the triple is compatible).  With our notation, the color $0$
corresponds to the specialization of ADO invariant which H. Murakami
and J. Murakami showed is
the Kashaev invariant (see \cite{MM}).
\begin{rem}
  Observe that Theorem \ref{T:OurVolCon} states something stronger
  than what is asked in the above question. Indeed the cohomology
  classes considered in the theorem are not necessarily zero: only the
  coloring they induce on $L$ is zero. So one may ask a stronger
  question considering this larger class of cohomology classes.
\end{rem}
\begin{rem}
  The restriction to odd $r$ is coherent both with the statement of
  Theorem \ref{T:OurVolCon} and with the results of \cite{LT}
  concerning the standard version of the volume conjecture.
\end{rem}

At this point the following natural question is not supported by any evidence.
\begin{question}[A version of the Volume Conjecture for closed manifolds]\label{C:GVolConjClosed}
  Let $M$ be a closed oriented 3--manifold admitting a complete
  hyperbolic metric with volume $\operatorname{Vol}(M)$.
  For each odd integer $r$, let $\coh_r\in H^1(M;\C/2\Z)$
  be the zero cohomology class.  Does the equality
 $$\lim_{r\to \infty} \frac{2\pi}{r}\log|\Nr^0(M,\emptyset,\coh_r)|
 =\operatorname{Vol}(M)$$
  hold?
\end{question}

%
\subsection{Surgeries on $T(2,2n+1)$ torus knots}\label{SS:torusKnots} 
Let $K^f_{2n+1}$ be the closure of the two strands braid
$\sigma_{1,2}^{2n+1}$ with $f$ additional full twists with respect to
the blackboard framing.  In other words, $K^f_{2n+1}$ is a torus knot
of type $T(2,2n+1)$ with $f$ additional twists.  For example, when
$f=-2$ and $n=1$ then $K^{-2}_3$ is the trefoil knot with framing
$+1$.  Let $K^f_{2n+1}(\alpha)$ be $K^f_{2n+1}$ colored by $\alpha\in
\C\setminus \Z$.

Applying Relation \eqref{eq:fusion} to the two parallel strands of
$K^f_{2n+1}(\alpha)$ and then applying Relations \eqref{eq:twistv} and
\eqref{eq:twist} we have
\begin{equation}\label{eq:trefoil}
\Nr(K^f_{2n+1}(\alpha))=\sum_{k\in \Hr}
q^{ft_\alpha+\frac{2n+1}{2}(-2t_\alpha+t_{2\alpha+k})}\qd(2\alpha+k)
\end{equation}
where $t_\alpha=\dfrac{\alpha^2-(\ro-1)^2}2$.  In particular, for the
trefoil $K^{-2}_3$ and $\ro =5$ we have:
\begin{equation}\label{eq:trefoil5}
\Nr(K^{-2}_3(\alpha))=5q^{\alpha^2/2}
\frac{q^3\qn{3\alpha}+q\qn{5\alpha}-q^{-1}\qn{9\alpha}}{\qn{5\alpha}}
\end{equation}

Let $M_{f,n}$ be the manifold obtained by doing surgery along
$K^f_{2n+1}$.  We will now compute $\Nr(M_{f,n})$.  Suppose that
$f+2n+1\neq 0$; since $H_1(M_{f,n};\Z)=\Z/(f+2n+1)\Z$ (and a generator
is represented by the meridian of $K^f_{2n+1}$), a cohomology class
$\coh\in H^1(M_{f,n};\C/2\Z)$ is determined by its value
$\wb{\alpha}\in \C/2\Z$ on a generator of $H_1(M_{f,n};\Z)$; clearly
such a value multiplied by $f+2n+1$ must be $0$ in $\C/2\Z$.  Then
choosing $\alpha\in \C$ such that $\alpha+r-1\equiv \frac{2}{f+2n+1}\
({\rm mod}\ 2\Z)$ one can compute $\Nr(M_{f,n},\emptyset,\coh)$ where
$\wb{\alpha}$ is the class of $\alpha+r-1$ in $\C/2\Z$:
$$\Nr(M_{f,n},\emptyset,\coh)=\frac{\sum_{k\in
    \Hr}\qd(\alpha+k) \Nr(K^f_{2n+1}(\alpha+k))}{\Delta_{\sign(f+2n+1)}}.$$
 
%
%
\subsection{Non-triviality of $\Nr^0$ on Poincar\'e sphere}\label{SS:Non-trivPoincare}
In this subsection, assuming that $r$ is odd, we give several formulas
for $\Nr^0(M,\emptyset,0)$ when $M$ is a 3-manifold obtain by a
surgery on a knot. Then we use these formulas to show that
$\Nr^0(P,\emptyset,0)\neq 1$ where $P$ is the Poincar\'e sphere.

Let $M$ be a manifold obtained by surgery on a knot $K$ with framing $f\neq0$.  
Let $K(\alpha)$ be the knot $K$ colored with $\alpha\in \C$.
We have
\begin{prop}\label{prop:doubleslide}
  Define $\PP(\alpha)=\sum_{k\in\Hr}\Nr(K(\alpha+k))$.  Then
  $q^{-\frac f2\alpha^2}\PP(\alpha)$ is a Laurent polynomial in
  $q^{\pm\alpha}$.  In particular, it is continuous at $\alpha\in\Z$.
  Furthermore,
  $$\Nr^0(M,\emptyset,0)=\frac1{\Delta_{\sign(f)}}\sum_{k\in\Hr}q^{k}\PP(k)
  =\frac1{\Delta_{\sign(f)}}\sum_{k\in\Hr}q^{-k}\PP(k).$$
\end{prop}
\begin{proof}
  Let $DK(\alpha,\beta)$ be the 2-cable of $K$ whose components are
  colored with $\alpha$ and $\beta$.  In order to compute
  $\Nr^0(M,\emptyset,0)$, we consider $K$ colored by $0\in\C/2\Z$
  unlinked with the unknot $u_\alpha$ colored by $\alpha$.  Sliding
  $u_\alpha$ over $K$ one obtains the computable presentation
  $DK(\alpha,\Omega_{-\wb{\alpha}})$ (where $\Omega_{-\wb{\alpha}}$ is
  a Kirby color of degree $-\wb{\alpha}$).
  Observe now that (applying the fusion rule \eqref{eq:twistv} twice)
  one gets
  $$\Nr(DK(\alpha,\beta))=\sum_{k\in \Hr}
  \Nr\big(K(\alpha+\beta+k)\big)=\PP(\alpha+\beta).$$ Moreover, one
  can prove that $\Nr(DK(\alpha,\beta)\big)$ is a Laurent polynomial
  in $q^{\alpha}$ and $q^{\beta}$ times $q^{\frac f2(\alpha+\beta)^2}$ (by an argument
  similar to \cite[Theorem 2]{GP0};  see also \cite[Lemma
  11]{CGP2}).  Thus we have for any $h\in\Z$,
  $$\Nr\big(DK(\alpha,h-\alpha)\big)=\lim_{\beta\to h-\alpha} \sum_{k\in \Hr}
  \Nr\big(K(\alpha+\beta+k)\big)=\PP(h).$$ Now by definition of
  $\Omega_{-\wb{\alpha}}$ and of $\Nr^0$, 
  $$\Nr^0(M,\emptyset,0)=\frac{1}{\Delta_{\sign(f)}\qd(\alpha)}\sum_{h\in \Hr}
  \qd(h-\alpha)\Nr \big(DK(\alpha,h-\alpha)\big).$$  
  Thus, 
  $$\Delta_{\sign(f)}\Nr^0(M,\emptyset,0)= \frac{1}{\qn\alpha}\sum_{h\in \Hr}
  \qn{\alpha-h}\PP(h)$$ 
  $$=\frac{q^{\alpha}}{q^\alpha-q^{-\alpha}}\sum_{h\in\Hr}q^{-h}\PP(h)
  -\frac{q^{-\alpha}}{q^\alpha-q^{-\alpha}}\sum_{h\in\Hr}q^h\PP(h).$$ And the
  result follows from the fact that $\Nr^0(M,\emptyset,0)$ does not depend on
  $\alpha$.
\end{proof}
As in the last proposition it can be shown that there exists a Laurent
polynomial $\wt K(X)\in\C[X,X^{-1}]$ such that
  \begin{equation}
  \label{E:NtildeK}
  \Nr(K(\alpha))=\theta_\alpha^f\qd(\alpha)\wt K(q^\alpha)
  \end{equation} 
  where $\theta_\alpha=q^{\frac12(\alpha^2-(\ro-1)^2)}$ is a factor coming
  from the framing.  We will use this fact to give another formula for
  $\Nr^0(M,\emptyset,0)$.
\begin{prop}
  Assume that $\ro$ is odd. Let $\wt K(X)$ be the Laurent polynomial
  discussed above.  Suppose that $\Nr(K(\alpha))=\Nr(K(-\alpha))$ or
  equivalently that $\wt K(X^{-1})=\wt K(X)$, then
  \begin{equation*}
    \label{eq:N0oddKnot}
    \Nr^0(M,\emptyset,0)=\frac{\ro f\theta_0^f}{2\qn1\Delta_{\sign(f)}}
    \sum_{n\in\Hr}q^{2fn^2}\qn{2n}^2\wt K(q^{2n}).
  \end{equation*}
\end{prop}
\begin{proof}
  From Proposition \ref{prop:doubleslide}, we have
  \begin{align*}
    S=\Delta_{\sign(f)}\Nr^0(M,\emptyset,0)&=
    \sum_{\ell\in\Hr}q^{\ell}\PP(\ell)=
    \lim_{\ve\to0}\sum_{k,\ell\in\Hr}q^\ell\Nr(K(\ve+k+\ell))\\
    &=\lim_{\ve\to0}\sum_{n=1-\ro }^{\ro -1}\sum_{\tiny{\begin{array}{c}
          h,\ell\in\Hr\\h+\ell=2n \end{array}}}
    q^\ell\Nr(K(\ve+2n))
  \end{align*}
  The sum of $q^\ell$ for $h,\ell\in\Hr$ with $h+\ell=2n$ has as many
  terms as there are possible values of 
  $\ell$~: $\max(1-\ro ,1-\ro +2n)\le
  \ell\le\min(\ro -1,\ro -1+2n)$.  This sum is equal to
  $q^{n}\dfrac{\qn{\ro -|n|}}{\qn1}=q^{n}\dfrac{\qn{|n|}}{\qn1}$, so
  \begin{align*}
    S&=
    \lim_{\ve\to0}\frac{1}{\qn1}\sum_{n=1-\ro }^{\ro -1}q^{n}{\qn{|n|}}\Nr(K(\ve+2n))
    \\ &
    =\lim_{\ve\to0}\frac{1}{\qn1}\sum_{n=1-\ro }^{\ro -1}q^{n}
    {\qn{|n|}}\frac12\bp{\Nr(K(2n+\ve))+\Nr(K(2n-\ve))}\\    
    &=\lim_{\ve\to0}\frac{1}{2\qn1}\sum_{n=1-\ro }^{\ro -1}q^{n}
    {\qn{|n|}}\bp{\Nr(K(\ve+2n))+\Nr(K(\ve-2n))}
    \\ &
    =\lim_{\ve\to0}\frac{1}{2\qn1}\sum_{n=1}^{\ro -1}
    {\qn{2n}}\bp{\Nr(K(\ve+2n))+\Nr(K(\ve-2n))}\\    
  \end{align*}
  Here the second equality is the average of the right an left limits,
  the third equality comes from the fact that
  $\Nr(K(\alpha))=\Nr(K(-\alpha))$ and the last equality is obtained
  by grouping the terms $n$ and $-n$.  Now, remark that $\qd(\alpha)$
  and $\wt K(q^\alpha)$ depend only of $\alpha$ modulo $2\ro$ so
  $\Nr(K(\alpha\pm2\ro))=
  \bp{\frac{\theta_{\alpha\pm2\ro}}{\theta_{\alpha}}}^f\Nr(K(\alpha))=q^{\pm2\ro
    f\alpha}\Nr(K(\alpha))$.  Using this, the sum of the
  $n$\textsuperscript{th} and the $(r-n)$\textsuperscript{th} terms of
  the previous sum is
  $$
  \begin{array}{l}
    \qn{2n}\bp{\Nr(K(\ve-2n))\bp{1-q^{2\ro f(\ve-2n)}}+
      \Nr(K(\ve+2n))\bp{1-q^{-2\ro f(\ve+2n)}}}\\[1ex]
    =\qn{2n}\qn{\ro
    f\ve}\bp{q^{-\ro f\ve}\Nr(K(\ve+2n))-q^{\ro
      f\ve}\Nr(K(\ve-2n))}.
  \end{array}
  $$
  But $\lim_{\ve\to0}\qn{\ro f\ve}\qd(\ve\pm2n)=
  \lim_{\ve\to0}\ro\frac{\qn{\ro f\ve}}{\qn{\ro \ve}}\qn{\ve\pm2n}
  =\pm \ro f\qn{2n}$.  So,
  $$S=\frac{\ro f\theta_{0}^f}{2\qn1}\sum_{n=2,n\text{ even}}^{\ro -1}q^{2fn^2}
  {\qn{2n}}^2\bp{\wt K(q^{2n})+\wt K(q^{-2n})}$$
  Thus, the proposition follows.
\end{proof}
To show that $\Nr^0$ is non-trivial we will prove that its value for
the Poincar\'e sphere $P$ is not $1$; thus distinguishing it from
$S^3$.  Let $K$ be a trefoil with framing $+1$.  Using the notation of
Subsection \ref{SS:torusKnots}, $K$ is equal to $K^{-2}_{3}$ with the
color $\wb{0}\in \C/2\Z$.  Then for $\ro=5$ (and
$q=\exp(\frac{i\pi}{5})$), using Formula \eqref{eq:trefoil5} and
Equation \eqref{E:NtildeK} one obtains:
$$\wt K(q^\alpha)=
{q\,\qN{3\alpha}+q^{-1}\qN{5\alpha}+q^{2}\qN{9\alpha}}
\text{ where }\qN x=\frac{\qn x}{\qn1}.$$
Then Proposition \ref{eq:N0oddKnot} implies
$$\Nr^0(P,\emptyset,0)=\frac{5 q^{-8}}{\qn1\Delta_{+}}\bp
{q^8\qn4^2\wt K(q^4)+q^{32}\qn8^2\wt K(q^8)}=-(q^2+1)^2\neq 1=
\Nr^0(S^3,\emptyset,0).$$

\begin{question} 
  If $M$ is a closed oriented 3--manifold, is there a relation
  between $\Nr^0(M,\emptyset,0)$ and the Witten-Reshetikhin-Turaev
  invariant of $M$? (See \cite{CGP2} for some evidences in this direction.)
\end{question}

\subsection{A holomorphic function on cohomology}\label{S:hol}
Let $M$ be a 3--manifold and let $T$ be a (possibly empty) framed
trivalent graph in $M$.  By fixing a basis of $H_1(M\setminus T;\Z)$
we have $H_1(M\setminus T,\Z)=\Z^{b}\oplus_{i} \Z/n_i\Z$ for some
positive integers $b, n_i$.  Suppose also that $b\geq1$; then
$H^1(M\setminus T;\C/2\Z)$ is isomorphic to $(\C/2\Z)^{b}\times_{i}
\Z/n_i\Z$ where the isomorphism maps a class $\coh$ to the list of its
values on the elements of the basis.  Note that if an element $x$ has
order $n$ then $\coh(x)$ is an integer multiple of $\frac{2}{n}\in
\C/2\Z$.

Let $(M,T,\coh)$ be a compatible triple.  Let $L$ be a link in $S^3$
which represents a computable presentation of $(M,T,\coh)$.  Let $n_L$
be the number of components of $L$ and $L_i$ be the $i^{\text{th}}$
component of $L$.  For each $i\in \{1,...,n_L\}$, fix a complex number
$\alpha_i$ such that $\wb{\alpha_i}=g_\coh(L_i)$.  By definition
$$
\Nr(M,T,\coh)=\frac{1}{\Delta_+^p \Delta_-^s }\sum_{(k_i)\in
  (\Hr)^{n_L}}\left(\prod_{i=1}^{n_L}\qd(\alpha_i+k_i)\right)
\Nr(L^{(k_i)}\cup T)
$$ 
where $L^{(k_i)}$ is the link $L$ such that $L_i$ is colored by
$\alpha_i+k_i$.  Since $L$ provides a computable presentation then
$\alpha_i\in \C\setminus \Z$ and so $\qd(\alpha_i+k_i)$ is holomorphic
in $\alpha_i$.  Moreover, by Remark \ref{R:admcol} the invariants
$\Nr(L^{(k_i)}\cup T)$ is holomorphic in the colorings. Therefore,
when $\Nr(M,T,\coh)$ is defined it is holomorphic in $\coh$.
Similarly, one can show that $\Nr(M,T,\coh)$ is holomorphic in $\coh$
when defined using a $H$-stabilization.

%
%
\subsection{Behavior under connected sum}\label{S:consum}
We say that a compatible triple $(M, T, \coh)$ is {\em generic} if
$T\neq \emptyset$ or if $\coh$ is not integral.  Notice that
$\Nr(M,T,\coh)$ is defined only if $(M,T,\coh)$ is generic.

Let $(M_1, T_1, \coh_1)$ and $(M_2, T_2, \coh_2)$ be compatible
triples.  Recall the definition of the connected sum defined by
Equation \eqref{eq:consum}.  We have the following three cases with
regards to the connected sum:

\noindent 
$\bullet$ Case 1.  Both $(M_1, T_1, \coh_1)$ and $(M_2, T_2,\coh_2)$
are generic.  Then there exists a computable surgery presentation
$L_i$ of $(M_i, T_i, \coh_i)$ (or a stabilization of $(M_i, T_i,
\coh_i)$), for $i=1,2$.  Clearly, $(M_1,T_1,\coh_1)\#(M_2,T_2,\coh_2)$
can be presented as surgery over $L_1\sqcup L_2$ which is a split link
and thus $\Nr((M_1,T_1,\coh_1)\#(M_2,T_2,\coh_2))=0$.

\noindent
$\bullet$ Case 2.  Exactly one of the triples $(M_1, T_1, \coh_1)$ or
$(M_2, T_2,\coh_2)$ is generic.  Suppose $(M_1,T_1,\coh_1)$ is not
generic and then by Proposition \ref{prop:connsumone0} we have:
$$\Nr((M_1,T_1,\coh_1)\#(M_2,T_2,\coh_2))
=\Nr^0(M_1,T_1,\coh_1)\Nr(M_2,T_2,\coh_2).$$ 

\noindent
$\bullet$ Case 3.  Neither $(M_1,T_1,\coh_1)$ nor $(M_2,T_2,\coh_2)$
are generic.  Then $(M_1,T_1,\coh_1)\#(M_2,T_2,\coh_2)$ is not generic
and Proposition \ref{prop:connsumtwo0} implies:
$$\Nr^0((M_1,T_1,\coh_1)\#(M_2,T_2,\coh_2))
=\Nr^0(M_1,T_1,\coh_1)\Nr^0(M_2,T_2,\coh_2).$$ 

\subsection{Studying the self-diffeomorphisms of a rational homology
  sphere}
Let $M$ be a rational homology sphere.  Let $c_1,\ldots c_n$ and
$d_1,\ldots d_n$ be two distinct minimal sets of generators of
$H_1(M;\Z)$. Suppose that one is given $\coh,\coh'\in H^1(M;\C/2\Z)$
such that $\coh'(c_i)=\coh(d_i)$.  If $\Nr(M,\emptyset,\coh')\neq
\Nr(M,\emptyset,\coh)$ then there exists no positive
self-diffeomorphism $\phi:M\to M$ such that $\phi_*(c_i)=d_i$.  In
particular, one can apply this argument to distinguish generators of
$H_1(M;\Z)$.  For example, consider $H_1(L(5,1);\mathbb{Z})$.  Present
$L(5,1)$ as the surgery over a $5$-framed unknot in $S^3$; a generator
$c_1$ is represented by the meridian of the unknot, and another,
$d_1$, by its double.  Using Formula \eqref{eq:chain} with $\ro=3$ and
a computer one can see that
$$\Nr(L(5,1), \emptyset,\coh)\neq \Nr(L(5,1), \emptyset, \coh')$$
where $\coh(c_1)=\coh'(d_1)=\frac{2}{5}\in \C/2\Z$.
 
%
%
\subsection{Conjugation versus change of orientation}

Let $T\subset S^3$ be a framed, oriented graph with set of vertices
$\V$ and set of edges $\E$.  Fix $\delta:\V\to \Hr$ and assume that
$T$ has a $\srcol$-coloring with boundary $\delta$ (see Remark
\ref{R:admcol}). Let
$R_T=\mathbb{Q}(q)(\{q^{\frac{e_i}{4}},q^{\frac{e_je_k}{4}}|e_i,e_j,e_k\in
\E\})$ and let $i:R_T\to R_T$ be the $\mathbb{Q}$-linear isomorphism
defined by
$i(q)=q^{-1},i(q^{\frac{e_i}{4}})=(q^{\frac{e_i}{4}})^{-1},$ and $
i(q^{\frac{e_je_k}{4}})=(q^{\frac{e_je_k}{4}})^{-1}$ for all
$e_i,e_j,e_k\in \E$. By the axioms of Section \ref{S:axiom}, $\Nr(T)$
can be seen as the evaluation of an element $\Nrd(T)\in R_T$
via the map $\ev:R_T\to \C$ defined by $\ev(q)=\exp(\frac{i\pi}{r})$,
$\ev(q^{\frac{e_i}{4}})=\exp(\frac{\pi i \alpha_i}{4r})$ and
$\ev(q^{\frac{e_i e_j}{4}})=\exp(\frac{\pi i \alpha_i\alpha_j}{4r})$
where $\alpha_i\in\C$ denotes the color of $e_i$.

\begin{lemma}\label{lem:mirror}
  Let $T\subset S^3$ be a framed, oriented, colored graph.  Let $T^*$
  be the mirror image of $T$ equipped with the coloring obtained by
  conjugating all the colors of the edges of $T$ and with the framing
  obtained by taking the mirror image of the framing of $T$.
  Then $\Nr(T^*)=\wb{\ev(i(\Nrd(T)))}=\wb{\Nr(T)}\in \C$. 
\end{lemma}
\begin{proof}
  First remark that if $T$ is planar, then
  $\ev(i(\Nrd(T)))=\Nr(T)$ (and so the second equality is
  proved): indeed using the axioms of Section \ref{S:axiom}
  $\Nrd(T)$ can be expressed as a sum of products of the values
  of $6j$-symbols, of the theta-graphs and of unknots, each of which
  is expressed in terms of quantum binomials and are thus fixed by $i$
  because $i(\{x\})=-\{x\},\ \forall x\in \C$ and each binomial
  contains an even number of such terms. Moreover for planar $T$ if
  $\wb{T}$ is the mirror image of $T$ equipped with the same coloring
  as $T$, then $\Nr(\wb{T})=\Nr(T)$: to show this one may check that
  the value of a tetrahedron and its mirror image are equal and this
  is proved by applying four times the Axiom \eqref{eq:twistv}; then
  expressing again $\Nr(T)$ and $\Nr(\wb{T})$ as sums of products of
  values of $6j$-symbols, theta-graphs and unknots one gets that
  $\Nr(\wb{T})=\Nr(T)$. So to prove the first statement for planar
  $T$, it is sufficient to remark that for any $x\in \C$ it holds
  $\wb{(q^{x})}=q^{-\wb{x}}$ then $\wb{i(\{x\})}=-i(\{\wb{x}\})$ and
  since $q$-binomials contain an even number of factors of the form
  $i(\{x±\})$ it holds
  $\wb{i\Big(\qbin{x}{a}\Big)}=i\Big(\qbin{\wb{x}}{\wb{a}}\Big)$. Thus
  once again since the invariant of a planar graph can be expressed as
  sum of products of $q$-binomials the claim follows and so
  $\wb{\Nr(T)}=\wb{\Nr(\wb{T})}=\Nr(T^*)$.  For non planar $T$, given
  a diagram of it, using axioms \eqref{eq:fusion} and
  \eqref{eq:twistv} one can express $\Nr(T)$ as a sum of products of
  $6j$-symbols, theta graphs, unknots and some powers of $q$ whose
  exponent is a degree $2$-polynomial in the colors of $T$. Then one
  is left to check that switching all the crossings in a diagram of
  $T$ has the same effect on $\Nrd(T)$ as $i$: this is checked
  directly by the axioms \eqref{eq:twist} and \eqref{eq:twistv}
  because changing a factor from $q^{\frac{\alpha^2-(r-1)^2}{4}}$ to
  $q^{-\frac{\alpha^2-(r-1)^2}{4}}$ is exactly applying
  $\ev(i(\cdot))$.
\end{proof}

Given a compatible triple $(M,T,\coh)$ let $\wb{M}$ be $M$ with the
opposite orientation, $-T$ be $T$ with the opposite orientations on
the edges, $\wb{T}$ be $T$ equipped with the conjugated coloring and
$\wb{\coh}$ be the complex conjugate of $\coh$; observe that
$(M,-T,-\coh)$, $(\wb{M},-T,\coh)$ and $(M,\wb{T},\wb{\coh})$ are all
compatible triples and hence also $(\wb{M},\wb{T},-\wb{\coh})$ is.  If
$(M,T,\coh)$ is presented as an integral surgery over a framed link
$L$ in $S^3$ (with $T\subset S^3\setminus L$), then, letting $L^*\cup
T^*$ be the mirror image of $L\cup T$ (as in the statement of Lemma
\ref{lem:mirror}), it can be checked that $(\wb{M},\wb{T},-\wb{\coh})$
can be presented as surgery over $L^*$ (with the framing obtained by
taking the mirror image of the framing of $L$) and with $\wb{T}$ being
presented as $T^*$.  By Lemma \ref{lem:mirror} and the definitions of
$\Nr$ and of $\qd$ it holds
$\Nr(\wb{M},T^*,-\wb{\coh})=\wb{\Nr(M,T,\coh)}$.

\section{The general construction of the 3--manifold invariant}\label{sec:generalconstruction}
In this section we give a general categorical construction of an
invariant of a colored graph in a compact connected oriented
3--manifold and a cohomology class.  In Section
\ref{S:ProofsOfNrManInv} we will show that the invariant $\Nr$ defined
in Subsection \ref{SS:sl2NandN^0} is a special case of the invariant
defined in this section.

\subsection{Ribbon Ab-categories}
We describe the concept of a ribbon Ab-category (for details see \cite{Tu}).
A \emph{tensor category} $\cat$ is a category equipped with a covariant
bifunctor $\otimes :\cat \times \cat\rightarrow \cat$ called the tensor
product, a unit object $\unit$, an associativity constraint, and left and
right unit constraints such that the Triangle and Pentagon Axioms hold.
When the associativity constraint and the left and right unit constraints are
all identities we say the category $\cat$ is a \emph{strict} tensor
category. By Mac Lane's coherence theorem any tensor category is equivalent to
a strict tensor category.

A tensor category $\cat$ is said to be an \emph{Ab-category} if for any pair
of objects $V,W$ of $\cat$ the set of morphism $\Hom(V,W)$ is an additive
abelian group and the composition and tensor product of morphisms are
bilinear.

Let $\cat$ be a (strict) ribbon Ab-category, i.e. a (strict) tensor
Ab-category with duality, a braiding and a twist. Composition of
morphisms induces a commutative ring structure on $\End(\unit)$.  This
ring is called the \emph{ground ring} of $\cat$ and denoted by $\FK$.
For any pair of objects $V,W$ of $\cat$ the abelian group $\Hom(V,W)$
becomes a left $\FK$-module where the action is defined by
$kf=k\otimes f$ where $k\in \FK$ and $f\in \Hom(V,W)$. An object $V$
of $\cat$ is \emph{simple} if $\End(V)= \FK \Id_V$.  We denote the
braiding in $\cat$ by $c_{V,W}:V\otimes W \rightarrow W \otimes V$,
the twist in $\cat$ by $\theta_V:V\to V$ and duality morphisms in
$\cat$ by
\begin{align*}
  b_{V} : \:\: & \unit \rightarrow V\otimes V^{*}, & b'_{V} : \:\: &
  \unit\rightarrow V^*\otimes V, & d_{V}: \:\: & V^*\otimes V\rightarrow
  \unit, & d_{V}':\:\: & V\otimes V^{*}\rightarrow \unit.
\end{align*}
An object $V$ of $\cat$ is a direct sum of the objects $V_1,\ldots
V_n$ if there exist morphisms $\alpha_i:V_i\to V$ and $\beta_i:V\to
V_i$ for $i=1\cdots n$ with $\beta_j\alpha_i=\delta_i^j\Id_{V_i}$ and
$\Id_V=\sum_{i=1}^n\alpha_i\beta_i$.
We say that a full subcategory $\mathscr{D}$ of $\cat$ is {\em
  semi-simple} if every object of $\mathscr{D}$ is a direct sum of
finitely many simple objects in $\mathscr{D}$ and any two
non-isomorphic simple objects $V$ and $W$ of $\mathscr{D}$ have the
property $\Hom_\cat(V,W)=0$.

\subsection{Ribbon graphs}\label{SS:RibbonGraphs} 
Here we recall the notion of the category of $\cat$-colored ribbon
graphs and its associated ribbon functor (for details see \cite{Tu}).
\begin{figure}[b]
  \centering $ \xymatrix{ \ar@< 8pt>[d]^{W_m}_{... \hspace{1pt}}
    \ar@< -8pt>[d]_{W_1}\\
    *+[F]\txt{ \: \; f \; \;} \ar@< 8pt>[d]^{V_n}_{... \hspace{1pt}}
    \ar@< -8pt>[d]_{V_1}\\
    \: } $
 \caption{}\label{F:Coupon}
\end{figure}
A morphism $f:V_1\otimes{\cdots}\otimes V_n \rightarrow
W_1\otimes{\cdots}\otimes W_m$ in $\cat$ can be represented by a box and
arrows as in Figure \ref{F:Coupon}.  Here the plane of the picture is oriented
counterclockwise, and this orientation determines the numbering of the arrows
attached to the bottom and top of the box.  More generally, we allow
such boxes with some arrows directed up and some arrows directed down. For
example, if all the bottom arrows in the picture are redirected upward,
then the box represents a morphism $V_1^*\otimes{\cdots}\otimes V_n^*
\rightarrow W_1\otimes{\cdots}\otimes W_m$. The boxes as above are called {\it
  coupons}.  Each coupon has distinguished bottom and top sides and all
incoming and outgoing arrows can be attached only to these sides.

Let $M$ be an oriented 3--manifold.  A \emph{ribbon graph} in $M$ is
an oriented compact surface in $M$ which decomposed into elementary
pieces: bands, annuli, and coupons (see \cite{Tu}).  A
\emph{$\cat$-ribbon graph} in $M$ is a ribbon graph whose bands and
annuli are colored by objects of $\cat$ and whose coupons are colored
with morphisms of $\cat$.  Recall the definition of a framed colored
graph given in Subsection \ref{SS:NotationNr}.  A $\cat$-ribbon graph
in $\R^2 \times [0,1]$ is a $Ob(\cat)$-colored framed oriented graph
such that all the vertices lying in $\Int(\R^2 \times [0,1])$ are not
univalent and are thickened to coupons colored by morphisms of~$\cat$.
Let $\A$ be a set of simple objects of $\cat$.  By a \emph{$\A$-graph}
in $M$ we mean a $\cat$-ribbon graph in $M$ where at least one band or
annuli is colored by an element of $\A$.  By an \emph{$\A$-graph}, we
mean a closed $\A$-graph in $S^3$ or in $B^3$.

Next we recall the category of $\cat$-colored ribbon graphs $Rib_\cat$
(for more details see \cite{Tu} Chapter~I).  The objects of $Rib_\cat$
are sequences of pairs $(V,\epsilon)$, where $V\in Ob(\cat)$ and
$\epsilon=\pm$ determines the orientation of the corresponding
edge. The morphisms of $Rib_\cat$ are isotopy classes of
$\cat$-colored ribbon graphs in $\R^2 \times [0,1]$ and their formal
linear combinations with coefficients in $\FK$. From now on we write
$V$ for $(V,+)$.

Let $F$ be the usual ribbon functor from $Rib_\cat$ to $\cat$ (see \cite{Tu}).
Let $L$ be a $\A$-graph and let $V$ be the color of an edge of $L$ belonging
to $\A$.  Cutting this edge, we obtain a $\cat$-colored (1,1)-ribbon graph
$T_V$ whose closure is $L$.  Since $V$ is simple $F(T_{V})\in
\End_{\cat}(V)= \FK \Id_V$.  Let $<T_{V}> \, \in \FK$ be such that
$F(T_{V})= \, <T_{V}> \, \Id_V$.  Let $\qd:\A\to\FK^*$ be a map.  We
say $(\A,\qd)$ is an \emph{ambidextrous pair} if, for all $L$,
\begin{equation}\label{E:Def_of_F'}
  F'(L) = \qd(V)<T_{V}> \in \FK
\end{equation} is independent of the choice of the edge 
to be cut and yields a well defined invariant of $L$.

\subsection{Relative $\Gr$-modular categories} \label{SS:RelGModcat}
In this subsection we give the main new categorical notion of this paper.
If $\Gr$ is a commutative group (in this section we use a multiplicative
notation), a \emph{$\Gr$-grading} in $\cat$ is a family $\{\cat_g\}_{g\in\Gr}$
of full subcategories of $\cat$ such that:
\begin{enumerate}
\item If $V\in\cat_g$, $V'\in\cat_{g'}$ then $V\otimes V'\in\cat_{gg'}$.
\item If $V\in\cat_g$ then $V^{*}\in\cat_{g^{-1}}$.
\item If $V\in\cat_g$, $V'\in\cat_{g'}$ and $g\neq g'$ then $\Hom_\cat(V,
  V')=0$.
\end{enumerate}

A set of objects of $\cat$ is said to be {\em commutative} if for any
pair $(V,W)$ of objects in this set we have $c_{V,W}\circ
c_{W,V}=\Id_{W\otimes V}$ and $\theta_V=\Id_V$.  Let $\TT$ be a
commutative group with additive notation.  A {\em realization} of
$\TT$ in $\cat$ is a commutative set of object $\{\ve^{t}\}_{t\in\TT}$
such that $\ve^{0}=\unit$, $\qdim(\ve^{t})=1$ and
$\ve^{t}\otimes\ve^{t'}=\ve^{t+t'}$ for all $t,t'\in\TT$.
\begin{lemma}\label{L:Tacts}
  Let $\{\ve^{t}\}_{t\in\TT}$ be a realization of $\TT$ in $\cat$.  If
  $ t\in\TT$ then $\ve^{t}$ is simple.  Also, for any objects
  $V,W\in\cat$, the map $\Hom_\cat(V,W)\rightarrow
  \Hom_\cat(V\otimes\ve^{t} ,W\otimes\ve^{t})$ given by $f\mapsto
  f\otimes \Id_{\ve^{t}}$ is an isomorphism.  In particular, if $V$ is
  simple then $V\otimes\ve^{t}$ is also simple.
\end{lemma}
\begin{proof}
  We have that for any $t\in\TT$,
  $\Id_{\ve^{t}}\otimes\Id_{\ve^{-t}}=\Id_\unit$. It follows that the
  map $\Hom_\cat(V,W)\to\Hom_\cat(V\otimes\ve^{t} ,W\otimes\ve^{t})$
  which sends $f$ to $f\otimes\Id_{\ve^{t}}$ has an inverse map given
  by $f\mapsto f\otimes\Id_{\ve^{-t}}$.  In particular, if $V$ is
  simple then $\End_\cat(V)=\FK.\Id_V$ implying that
  $\End_\cat(V\otimes\ve^{t})=\FK.\Id_{V\otimes\ve^{t}}$.  Finally,
  taking $V=\unit$ we have that $\ve^t$ is simple.
\end{proof}
A realization of $\TT$ in $\cat$ induces an action of $\TT$ on
isomorphism classes of objects of $\cat$ by
$(t,V)\mapsto\ve^{t}\otimes V\simeq V\otimes \ve^{t}$ where the
isomorphism here is given by the braiding.  We say that
$\{\ve^{t}\}_{t\in\TT}$ is a {\em free realization} of $\TT$ in $\cat$
if this action is free.  This means that for any
$t\in\TT\setminus\{0\}$ and for any object $V$ of $\cat$,
$V\otimes\ve^{t}\nsimeq V$.

For a simple object $V$, we denote by $\wt V$ the set of isomorphism
classes of the set of simple objects $\{V\otimes\ve^t | t\in \TT\}$.
We say that $\wt V$ is a {\em simple $\TT$-orbit}.

\begin{defi}\label{D:GmodularCat}
  A ribbon category $\cat$ is {\em $\Gr$-modular relative to $\X$ with
    modified dimension $\qd$ and periodicity group $\TT$} if
  \begin{enumerate}
  \item \label{I1:DefGmodularCat} the category $\cat$ has a
    $\Gr$-grading $\{\cat_g\}_{g\in\Gr}$,
  \item \label{I2:DefGmodularCat} there is a free realization
    $\{\ve^{t}\}_{t\in\TT}$ of the group $\TT$ in $\cat_1$ (where
    $1\in\Gr$ is the unit),
  \item \label{I3:DefGmodularCat} there is a bilinear pairing
    $\Gr\times\TT\to\FK^{*}$, $(g,t)\mapsto g\pow{t}$ such that for
    any object $V$ of $\cat_g$ we have $c_{V,\ve^{t}}\circ
    c_{\ve^{t},V}=g\pow{t}\Id_{\ve^{t}\otimes V}$, for all $t\in\TT$,
  \item \label{I4:DefGmodularCat} there exists $\X\subset \Gr$ such
    that $\X^{-1}=\X$ and $\Gr$ cannot be covered by a finite number
    of translated copies of $\X$, in other words, for any $ g_1,\ldots
    ,g_n\in \Gr$, we have $\bigcup_{i=1}^n (g_i\X) \neq\Gr $,
  \item \label{I5:DefGmodularCat} there is an ambidextrous pair
    $(\A,\qd)$ where $\A$ contains the set of simple objects of
    $\cat_g$ for all $g\in\Gr\setminus\X$,
  \item \label{I6:DefGmodularCat} for all $g\in\Gr\setminus\X$, the category
    $\cat_g$ is semi-simple and its simple objects form a union of finitely
    many simple $\TT$-orbits,
  \item \label{I7:DefGmodularCat} there exists an element $g\in
    \Gr\setminus \X$ and an object $V\in \cat_g$ such that the scalar
    $\Delta_+$ given in Figure \ref{F:Delta+-} is non-zero (similarly,
    there exists $g\in \Gr\setminus \X$ and $V\in \cat_g$ such that
    $\Delta_- \neq 0$),
  \item \label{I8:DefGmodularCat} $F(H(V,W)) \neq 0$, for all $V,W\in
    \A$, where $H(V,W)$ is the long Hopf link whose long edge is
    colored $V$ and circle component is colored with $W$.
  \end{enumerate}
\end{defi}
\begin{figure}$F\left(\epsh{fig1}{12ex}
    \put(-33,-2){\ms{\Omega_g}}\put(-26,-20){\ms{V}} \right)
  =\Delta_-\Id_V,\quad \quad
  F\left(\epsh{fig2}{12ex}\put(-33,-2){\ms{\Omega_g}}\put(-26,-20){\ms{V}}
  \right)=\Delta_+\Id_V$
  \caption{Here $V$ is in $\cat_g$ and $\Omega_{g}$ is a formal linear
    combination of modules $\sum_{U\in Y}\qd(U)U$ where $Y\subset
    Obj(\cat_g)$ is a finite set representing the simple $\TT$-orbits
    in $\cat_g$.}
  \label{F:Delta+-}
 \end{figure}

 Remark that the bilinearity of the pairing $\Gr\times\TT\to\FK^{*}$
 means that $g\pow{t+t'}=g\pow{t}g\pow{t'}$ and
 $(gh)\pow{t}=g\pow{t}h\pow{t}$.  We can illustrate Condition
 \eqref{I3:DefGmodularCat} with the following skein relation:
\begin{equation}
  \label{E:skeineps}
  F\left(\epsh{fig26}{9ex}\put(-8,17){\ms V}\put(-20,17){${\ve^t}$}\right)
  =g\pow{t}
  F\left(\epsh{fig27}{9ex}\put(-8,17){\ms V}\put(-20,17){${\ve^t}$}\right) 
  \text{ for all }  V\in\cat_g.
\end{equation}
\noindent
\textbf{Notation.}  If $\cat$ is a category satisfying Definition
\ref{D:GmodularCat} then we say $\cat$ is a \emph{relative
  $\Gr$-modular} category.  For such a category let $\{\alpha\}$ be a
set indexing simple objects of $\A$.  Let $V\in \A$ and let $\alpha$
be the corresponding indexing element.  We will denote $ \dg{\alpha}$
as the unique element of $\Gr$ such that $V\in \cat_{ \dg{\alpha}}$.
Also, to simplify notation we will identify $\alpha$ with $V$ and
write $\alpha\in\A$.

\begin{rem} 
  Conditions \eqref{I1:DefGmodularCat} and \eqref{I2:DefGmodularCat}
  of Definition \ref{D:GmodularCat} are not very restrictive once one
  has a free realization of $\TT$ in $\cat$: The long Hopf link given
  by a straight strand colored by $V$ and its meridian colored by
  $\ve^t$ is sent by $F$ to a central isomorphism
  $\phi_t(V)\in\End_\cat(V)$.  For $g\in\Gr'=\Hom_\cat(\TT,\FK^{*})$,
  let $\cat_g'$ be the full subcategory of $\cat$ formed by objects
  $V$ such that $\phi_t(V)=g(t)\Id_V$, for all $ t\in\TT$.  If $V$ is
  a simple object of $\cat$ then $\phi_t(V)$ is a non-zero scalar and
  so $V$ belongs to some $\cat_g'$.  Furthermore, one can easily prove
  that $(\cat_g')_{g\in\Gr'}$ is a grading in $\cat$ and $\ve^t\in
  \cat_1'$, for all $ t\in\TT$.
\end{rem}

\begin{rem}\label{R:Hnonzero}
  Condition \eqref{I8:DefGmodularCat} holds in all the examples of this paper.
  In general this assumption is not true, however the graph $H$ can be
  exchanged with any ribbon graph which does not vanish when colored by
  elements of $\A$.  For such an exchange, in what follows, the process of
  $H$-stabilization below should be replaced by the connected sum with the new
  ribbon graph.
\end{rem}

\subsection{Main results} \label{SS:MainResults} 
Here we give the general definition of the invariants this paper.
Recall the notation and definitions of Subsection
\ref{SS:CohomGcoloring}.  Let $\cat$ be a relative $\Gr$-modular
category relative to $\X$ with modified dimension $\qd$ and
periodicity group $\TT$.  A formal linear combination of objects of
$\cat$ is a \emph{homogeneous} $\cat$-color of degree $g\in\Gr$ if all
appearing objects belong to the same $\cat_g$.  We say a ribbon graph
$T$ has a \emph{homogeneous} $\cat$-coloring if each edge $e$ of $T$ is
colored by a homogeneous $\cat$-color of degree $g_e\in\Gr$ for some
$1$-cycle $\sum_eg_e[e]\in H_1(T,G)$ which is called \emph{the
  $\Gr$-coloring} of $T$.
\begin{defi}\label{def:adm} 
  Let $M$ be a compact connected oriented 3--manifold, $T$ a
  $\cat$-colored ribbon graph in $M$ and $\coh\in H^{1}(M\setminus
  T,\Gr)$.
\begin{enumerate}   
\item   We say that $(M,T,\coh)$ is {\em a compatible triple} if 
  $T$ has a homogeneous $\cat$-coloring given by $\Phi(\coh)\in H_1(T,\Gr)$.
\item A compatible triple is \emph{$T$-admissible} if there exists an
  edge of $T$ colored by $\alpha\in \A$.
\item A surgery presentation via $L\subset S^3$ for a compatible
  triple $(M,T,\coh)$ is \emph{computable} if one of the two following
  conditions holds:
  \begin{enumerate}
  \item $L\neq \emptyset$ and $g_\coh(L_i) \in {\Gr}\setminus \X$ for
    all $L_i$ or
  \item $L=\emptyset$ and there exists an edge of $T$ colored by
    $\alpha\in \A$.
  \end{enumerate}
\end{enumerate}
\end{defi}

From now on we assume that $(M,T,\coh)$ is a compatible triple.

\begin{defi}
  The formal linear combination $\Omega_g=\sum_i\qd(V_i)V_i$ is a
  \emph{Kirby color of degree $g\in\Gr$} if the isomorphism classes of
  the $\{V_i\}_i$ are in one to one correspondence with the simple
  $\TT$-orbits of $\cat_g$.
\end{defi}

\begin{teo}\label{T:coh_adm}
  If $L$ is a link which gives rise to a computable surgery
  presentation of $(M,T,\coh)$ then
  $$\Nr(M,T,\coh)=\dfrac{F'(L\cup T)}{\Delta_+^{p}\ \Delta_-^{s}}$$
  is a well defined topological invariant (i.e. depends only of the
  diffeomorphism class of the triple $(M,T,\coh)$), where $(p,s)$ is
  the signature of the linking matrix of the surgery link $L$ and each
  component $L_i$ is colored by a Kirby color
  $\Omega_{g_{\coh}(L_i)}$.
\end{teo}
The proof of Theorem \ref{T:coh_adm} will be given in Section
\ref{S:ProofsT-coh-adm}.
 
\begin{prop}\label{P:coh_adm}
  Let $(M,T,\coh)$ be a compatible triple and consider $\coh$ as a map
  on $H_1(M\setminus T,\Z)$ with values in $\Gr$.  Suppose that $\coh$
  takes a value $g\in\Gr$ such that for each $ x\in \X$ there exists
  an $n(x)\in \Z$ such that $xg^{n(x)} \notin \X$. Then there exists a
  computable surgery presentation of $(M,T,\coh)$.
\end{prop}
The proof of Proposition \ref{P:coh_adm} will be given in Section
\ref{S:ProofsT-coh-adm}.\\

\begin{defi}\label{D:H-stab}[$H$-Stabilization]
  Let $H(\alpha,\beta)$ be a long Hopf link in $\mathbb{R}^3$ whose
  circle component is colored by $\alpha\in \A$ and whose long
  component is colored by $\beta \in \A$.  Let $(M,T,\coh)$ be a
  $T$-admissible triple, $e$ be an edge of $T$ colored by $\beta\in
  \A$, and $m$ be the meridian of $e$.  A $H$-\emph{stabilization} of
  $(M,T,\coh)$ along $e$ is a compatible triple $(M,T_H,\coh_H)$
  where:
  \begin{itemize}
  \item $T_H=T\cup m$, and $m$ is colored by $\alpha \in \A$,
  \item $\coh_H$ is the unique element of $H^1(M\setminus (T\cup
    m);{\Gr})$ such that $\coh_H(m)=\dg{\alpha}$ and
    $(\coh_H)|_{M\setminus (T\cup D)}=\coh$ where $D$ is a disc
    bounded by $m$ which $e$ intersects once.
  \end{itemize}
\end{defi}

\begin{teo}\label{T:T_adm}
  If $(M,T,\coh)$ is $T$-admissible then there exists a
  $H$-stabilization of $(M,T,\coh)$ admitting a computable
  presentation.  Let $(M,T_H,\coh_H)$ be such a $H$-stabilization and
  let $L$ be a link which gives rise to a computable surgery
  presentation of $(M,T_H,\coh_H)$ then
  $$\Nr(M,T,\coh)=\dfrac{F'(L\cup T_H)}{\brk{H}\Delta_+^{p}\ \Delta_-^{s}}$$
  is a well defined topological invariant (i.e. depends only of the
  diffeomorphism class of the triple $(M,T,\coh)$), where $(p,s)$ is
  the signature of the linking matrix of the surgery link $L$, each
  component $L_i$ is colored by a Kirby color
  $\Omega_{g_{\coh}(L_i)}$, $H=H(\alpha,\beta)$ is the long Hopf-link
  used in the stabilization and $\brk{H}$ is defined by the equality
  $F(H)=\brk{H}Id_{\beta}$.  Moreover, if $(M,T,\coh)$ has a
  computable presentation, then the invariant of this theorem is equal
  to the invariant of Theorem \ref{T:coh_adm}.
\end{teo}
  
The proof of Theorem \ref{T:T_adm} will be given in Section
\ref{S:ProofsT-coh-adm}.  Next we define another invariant which can
be non-zero when $\Nr$ is zero.  Before we do this we need the
following result.

Let $T$ and $T'$ be $\cat$-colored ribbon graphs and let $e\subset T,e'\subset
T$ be edges colored by $\alpha \in \A$.  Let $T\#_{(e,e')} T'$ be the
connected sum of $T$ and $T'$ along $e$ and $e'$, then
\begin{equation}\label{E:F'connectedSum}
F'(T\#_{(e,e')}T')=\frac{1}{\qd(\alpha)}F'(T)F'(T').
\end{equation}

Let $T$ be a $\cat$-colored ribbon graph with an edge $e$ colored by
$\alpha\in \A$.  Let $m$ be a meridian of $e$ colored with $\beta\in
\A$.  If $T'=T\cup m$ then
\begin{equation}
  \label{E:F'stabl}
  F'(T')=F'(T)\brk{H}  
\end{equation} 
where $H$ is the long Hopf link whose long edge is colored $\alpha$
and circle component is colored with $\beta$ and $\brk H$ is defined
by the equality $F(H)=\brk H \Id_\alpha$.
\begin{prop}\label{P:Ndirctsum}
  If $(M,T,\coh)$ is $T$-admissible and $(M',T',\coh')$ is
  $T'$-admissible and $e\subset T,e'\subset T$ are edges colored by
  $\beta \in \A$, then
  $$
  \Nr(M\# M',T\#_{(e,e')} T',\coh\#_{(e,e')}
  \coh')=\frac{1}{\qd(\beta)}\Nr(M,T,\coh)\Nr(M',T',\coh')
  $$ 
  where the connected sum is taken over balls intersecting $T$ and
  $T'$ along $e$ and $e'$, $T\#_{(e,e')} T'$ is the connected sum of
  $T$ and $T'$ along $e$ and $e'$ and $\coh\#_{(e,e')}\coh'$ is the
  cohomology class acting as $\coh$ and $\coh'$ on the images through
  the natural maps from $H_1(M\setminus T,\mathbb{Z})$ and
  $H_1(M'\setminus T',\mathbb{Z})$ in $H_1(M\#M' \setminus
  T\#_{(e,e')}T',\mathbb{Z})$.
\end{prop}
\begin{proof}
  Let $(M,T_H,\coh_H)$ and $(M',T'_{H'},\coh'_{H'})$ be
  $H$-stabilizations of $(M,T,\coh)$ and $(M',T',\coh')$ along $e$ and
  $e'$, respectively (here $H=H(\alpha,\beta)$ and
  $H'=H(\alpha',\beta)$ are long Hopf links and $\alpha, \alpha'\in
  \A$).  By definition
  $$
  \Nr(M,T,\coh)=\dfrac{F'(L\cup T_H)}{\brk{H}\Delta_+^{p}\
    \Delta_-^{s}}\;\; \text{ and }\;\;
  \Nr(M,T',\coh')=\dfrac{F'(L'\cup T'_{H'})}{\brk{H'}\Delta_+^{p'}\
    \Delta_-^{s'}}
  $$
  where $L$ and $L'$ are links colored as in Theorem \ref{T:T_adm} and
  $(p,s)$ (resp. $(p',s')$) is the signature of the linking matrix of
  $L$ (resp. $L'$).

  Using Equation \eqref{E:F'connectedSum} we have
  $$
  \dfrac{F'(L\cup T_H)}{\brk{H}\Delta_+^{p}\ \Delta_-^{s}}\cdot
  \dfrac{F'(L'\cup T'_{H'})}{\brk{H'}\Delta_+^{p'}\ \Delta_-^{s'}}
  =\dfrac{\qd(\beta)F'\left((L\cup T_H)\#_{(e,e')}(L'\cup
      T'_{H'})\right)}{\brk{H}\brk{H'}\Delta_+^{p+p'}\
    \Delta_-^{s+s'}}.
  $$
  Now
  \begin{align*}
    F'\left((L\cup T_H)\#_{(e,e')}(L'\cup T'_{H'})\right)
    &= \brk{H'}F'\left((L\cup T_H)\#_{(e,e')}(L'\cup T')\right)\\
    &=\brk{H'}F'\left((L\cup L')\cup (T\#_{(e,e')}T')_H\right).
  \end{align*}
  where the first equality follows from Equation \eqref{E:F'stabl} and
  the second from the fact that the two graphs in $F'$ are isotopic.
  By definition
  $$
  \Nr(M\# M',T\#_{(e,e')} T',\coh\#_{(e,e')}
  \coh')=\dfrac{F'\left((L\cup L')\cup
      (T\#_{(e,e')}T')_H\right)}{\brk{H}\Delta_+^{p+p'}\
    \Delta_-^{s+s'}}
  $$ 
  thus the result follows from the last two equations.
\end{proof}
  
Let $(M,T,\coh)$ be compatible triple.  Let $u_\alpha$ be an unknot in
$S^3$ colored by $\alpha\in \A$.  Let $\coh_\alpha$ be the unique
element of $H^1(S^3\setminus u_\alpha,\Gr)$ such that
$(S^3,u_\alpha,\coh_\alpha)$ is a compatible triple.  Recall the
definition of the connected sum given by Equation \eqref{eq:consum}.
\begin{teo}\label{T:N0}
  Define 
  $$\Nr^{0}(M,T,\coh)=\frac{\Nr((M,T,\coh)\#(S^3,u_\alpha,\coh_\alpha))}
  {\qd(\alpha)}.$$ Then $\Nr^{0}(M,T,\coh)$ is a well defined
  topological invariant (i.e. depends only of the diffeomorphism class
  of the compatible triple $(M,T,\coh)$).
\end{teo}
\begin{proof}
  We must show the definition of $\Nr^0$ does not depend on the color
  $\alpha\in \A$.  Let $\alpha, \beta \in \A$.  Let $u_\beta$ be the
  unknot with color $\beta$ and edge $e_\beta$.  Let $\coh_\beta$ be
  the unique element of $ H^{1}(S^3\setminus u_\beta,\Gr)$ such that
  $(S^3, u_\beta,\coh_{\beta})$ is a compatible triple.  Let $H$ be
  the Hopf link whose edges $e_1$ and $e_2$ are colored with $\alpha$
  and $\beta$, respectively.  Let $\coh_{\alpha,\beta}$ be the unique
  element of $ H^{1}(S^3\setminus H,\Gr)$ such that $(S^3,
  H,\coh_{\alpha,\beta})$ is a compatible triple.  Consider the
  cohomology class $\coh\sqcup \coh_\beta$ of
  $(M,T,\coh)\#(S^3,u_\beta,\coh_\beta)$.  Then
  \begin{align*}
     \Nr((M,T,\coh)\#(S^3,H,\coh_{\alpha,\beta}))
     &=\Nr(M \#S^3,(T\sqcup u_\beta)\#_{(e_\beta,e_2)}H, 
     (\coh\sqcup \coh_\beta) \#_{(e_\beta,e_2)} \coh_{\alpha,\beta})\\
     &=\frac{\Nr((M,T,\coh)\#(S^3,u_\beta,\coh_\beta))
       \Nr(S^3,H,\coh_{\alpha,\beta})}{\qd(\beta)}
   \end{align*}
   where the first equality from follows the fact that the two triples
   are diffeomorphic and the second equality comes from Proposition
   \ref{P:Ndirctsum}.  Similarly,
   $$ \Nr((M,T,\coh)\#(S^3,H,\coh_{\alpha,\beta}))=
   \frac{\Nr((M,T,\coh)\#(S^3,u_\alpha,\coh_\alpha))
     \Nr(S^3,H,\coh_{\alpha,\beta})}{\qd(\alpha)}.$$ Thus, the theorem
   follows from the last two equations.
\end{proof}

\begin{rem}\label{R:RTrecoved} Any modular category satisfying
  Condition \eqref{I8:DefGmodularCat} of Definition
  \ref{D:GmodularCat} give rise to trivial examples of relative
  $\Gr$-modular categories where $\Gr$ and $\TT$ are both the trivial group,
  $\X=\emptyset$, $\A$ is the set of simple objects and $\qd=\qdim$ is
  the quantum dimension.  In this case the construction proposed in
  this paper reduces to the original Reshetikhin-Turaev construction
  of \cite{RT, Tu} and then $\Nr=\Nr^0$ is the usual
  Witten-Reshetikhin-Turaev invariant. 
\end{rem}

Next we show how $\Nr$ and $\Nr^0$ behave under the connected sum
defined in Equation \eqref{eq:consum}.
If $T$ is a $\A$-graph in $\R^3$ and $T'$ is any $\cat$-colored ribbon graph
with coupons, then 
\begin{equation}\label{E:F'T'T}
  F'(T\sqcup T')=F'(T)F(T').
\end{equation}
Recall that the categorical dimension of an object $V$ in $\cat$ is
the element $d'_V \circ b_V\in \End(\unit)=\FK$.  There are many
interesting examples where the categorical dimension is zero on all of
the objects of $\A$ (see Section \ref{S:otherQG} and the examples of
\cite{GP2, GPT}).  For such a category we have $F(T)=0$ for any
$\A$-graph and $F'$ vanishes on a disjoint union of $\A$-graphs (see
Lemma 16 and Proposition 19 of \cite{GPT}).
Similarly, we have the following proposition:
\begin{prop}\label{prop:connsumone0}
  Suppose that the categorical dimension of any object in $\A$ is
  zero.
  If $(M,T,\coh)$ is $T$-admissible or has a computable surgery presentation
  then $\Nr^0(M,T,\coh)=0$.  In addition, if $(M',T',\coh')$ is a compatible
  triple, then
  $$\Nr((M,T,\coh)\#(M',T',\coh'))=\Nr(M,T,\coh)\Nr^0(M',T',\coh').$$   
\end{prop}
\begin{proof}
We first prove the last equality.

First if $(M,T,\coh)$ is $T$-admissible then let $(M,T_H,\coh_H)$ be a
$H$-stabilization of $M$ with computable presentation given by $L\cup
T_H$.  As in Definition \ref{D:H-stab}, we call $\alpha$ the color of
the new meridian $m\subset T_H$.  If $\alpha$ as been chosen
appropriately then $(M',T',\coh')\#(S^3,u_\alpha,\coh_\alpha)$ as in
Theorem \ref{T:N0} admits a computable presentation given by $L'\cup
T'\cup u_\alpha$.  Take a connected sum in $S^3$ of these two
presentations along $m$ and $u_\alpha$.  The property of $F'$ given in
Equation \eqref{E:F'connectedSum} implies that $F'((L\cup
T_H)\#_{(m,u_\alpha)}(L'\cup T'\cup u_\alpha))
=\frac1{\qd(\alpha)}F'(L\cup T_H)F'(L'\cup T'\cup u_\alpha)$.  But
$(L\cup T_H)\#_{(m,u_\alpha)}(L'\cup T'\cup u_\alpha)$ is a computable
presentation of a $H$-stabilization of $(M,T,\coh)\#(M,T',\coh')$.
Thus $\Nr((M,T,\coh)\#(M',T',\coh'))=\Nr(M,T,\coh)\Nr^0(M',T',\coh').$

Now if $(M,T,\coh)$ has a computable surgery presentation, let $L$ be
a link which gives this surgery presentation.  Let $K$ be any sublink
of $L$ and let $L_1=L\setminus K$ then clearly $F'(L\cup
T)=F'(L_1\cup(T\cup K))$.  If $S^3_{L_1}$ is the manifold given by
surgery on $L_1$ then there exists a $T$-admissible triple
$(S^3_{L_1}, T\cup K,\coh)$ (where $K$ is colored by the Kirby color
$\Omega_{g_{\coh}(K)}$) and $a,b\in \Z$ such that
\begin{equation}\label{E:MM'K1}
  \Nr(M,T,\coh)=\frac{\Nr(S^3_{L_1}, T\cup K,\coh'')}{\Delta_+^{b}\Delta_-^{a}}.
\end{equation}
This process can be thought of as taking the sublink $K$ out of the
surgery presentation and ``adding'' it to the graph $T$.  To do this
one changes the original manifold and so the signature changes; this
is reflected in the constant mentioned above.  Similarly, for the
$a,b\in \Z$ used in Equation \eqref{E:MM'K1} we have
\begin{equation}\label{E:MM'K2}
  \Nr((M, T,\coh)\# (M',T',\coh'))= \frac{\Nr((S^3_{L_1}, T\cup K,\coh'')\#
    (M',T',\coh'))}{\Delta_+^{b}\Delta_-^{a}}. 
\end{equation}
Now since $(S^3_{L_1}, T\cup K,\coh'')$ is $T$-admissible then the
argument in the previous paragraph shows that
$$ \Nr((S^3_{L_1}, T\cup K,\coh'')\# (M',T',\coh'))=\Nr(S^3_{L_1},
T\cup K,\coh'')\Nr^0(M',T',\coh').$$ The desired result follows from
Equations \eqref{E:MM'K1} and \eqref{E:MM'K2}.

Finally if $(M,T,\coh)$ is $T$-admissible or has a computable surgery
presentation, consider the previous equality with
$(M',T',\coh')=(S^3,u_\alpha,\coh_\alpha)$ the unknot in the
3--sphere colored by $\alpha\in\A$.  Then we have
$$\Nr(M,T,\coh)\Nr^0(S^3,u_\alpha,\coh_\alpha)
=\Nr((M,T,\coh)\#(S^3,u_\alpha,\coh_\alpha))
=\Nr^0(M,T,\coh)\Nr(S^3,u_\alpha,\coh_\alpha).$$  But
$\Nr(S^3,u_\alpha,\coh_\alpha)=\qd(\alpha)\neq0$ whereas
$\Nr^0(S^3,u_\alpha,\coh_\alpha)=0$ because $F'$ vanishes on $\A$-colored
split links. Thus $\Nr^0(M,T,\coh)=0$. 
\end{proof}
\begin{prop}\label{prop:connsumtwo0}
If $(M,T,\coh)$  and $(M',T',\coh')$ are compatible triples, then 
$$\Nr^0((M,T,\coh)\#(M',T',\coh'))=\Nr^0(M,T,\coh)\Nr^0(M',T',\coh')$$ 
where the connected sum is taken over balls not intersecting $T$ and $T'$. 
\end{prop}
\begin{proof}
  Let $L\cup T\cup u_\alpha$ and $L'\cup T'\cup u_\alpha$ be
  computable presentations respectively of
  $(M,T,\coh)\#(S^3,u_\alpha,\coh_\alpha)$ and
  $(M',T',\coh')\#(S^3,u_\alpha,\coh_\alpha)$. Then since
  $u_\alpha\#u_\alpha=u_\alpha$, by the property of $F'$ for connected
  sum it holds:
\begin{equation}\label{eq:splitn0}
  F'(L\cup L'\cup T\cup T'\cup u_\alpha)=\qd(\alpha)^{-1}F'(L\cup T\cup
  u_\alpha)F'(L'\cup T'\cup u_\alpha)
\end{equation}
But $L\cup L'\cup T\cup T'\cup u_\alpha$ provides a computable
presentation for
$(M,T,\coh)\#(M',T'\coh')\#(S^3,u_\alpha,\coh_\alpha)$ and thus, using
formula (\ref{eq:splitn0}), $\Nr^0((M,T,\coh)\#(M',T',\coh')) $ can be
computed from this presentation as
$$
\qd(\alpha)^{-1}\Nr((M,T,\coh)\#(M',T'\coh')\#(S^3,u_\alpha,\coh_\alpha))
=\frac{F'(L\cup T\cup u_\alpha)F'(L'\cup T'\cup u_\alpha)}
{\qd(\alpha)^2\Delta_+^{s+s'}\Delta_-^{p+p'}}
$$ 
where $s,p$ (resp. $s',p'$) are the indices of inertia of the linking
matrix of $L$ and of $L'$ respectively.  But the latter is by
definition $\Nr^0(M,T,\coh)\Nr^0(M',T',\coh')$.
\end{proof}

\section{Proofs of Theorems \ref{T:coh_adm} and \ref{T:T_adm} and
  Proposition \ref{P:coh_adm}
}\label{S:ProofsT-coh-adm}

\subsection{Idea of the proofs.}\label{sub:kirbygompf}
Recall that Kirby's theorem \cite{Ki} allows one to relate any two
presentations of an oriented 3--manifold as surgery over a framed
link in $S^3$ by means of handle-slides, blow-up and blow-down
moves. \emph{Handle-slides} are depicted schematically in the proof of
Lemma~\ref{L:handle-slide}: they consist in modifying one component of
a link by replacing a chord of the component by one which is ``slid''
over or follows another component (and lies horizontal with respect to
the framing).  In this paper, \emph{blow-up moves} consist in adding
an unknot with framing $\pm 1$ which is linked with one component of
the link where the framing of this component is changed by $\pm 1$.  A
blow-up move with framing 1 or $-1$ can be depicted by the local
replacement:
$$\epsh{fig10}{12ex}\to \epsh{fig2}{12ex} \;\;\; \text{ or } \;\;\; 
\epsh{fig10}{12ex}\to \epsh{fig1}{12ex}. $$
\emph{Blow-down moves} are the inverse moves. The standard definition
of a blow-up move is simply adding a $\pm 1$-framed unknot which is
unlinked with the rest of the link; it is a standard fact that as soon
as the link is not empty the standard definition is equivalent to our
definition.
 
As remarked by R. Gompf and A. Stipsicz in \cite{GS} (see Theorem
5.3.6 and subsequent comments) every move $L\to L'$ describes an
isotopy class of diffeomorphisms between the surgered manifolds
$S^3_L$ and $S^3_{L'}$.  So one can check that Kirby's theorem
actually proves the following:
\begin{teo}[\cite{Ki}]\label{teo:refinedkirby} Let $M_1$ and $M_2$ be
  oriented 3--manifolds and $f:M_1\to M_2$ be an orientation
  preserving diffeomorphism. Any two surgery presentations $L_1$ and
  $L_2$ of $M_1$ and $M_2$, respectively can be connected by a
  sequence of handle-slides, blow-up moves and blow-down moves such
  that the induced diffeomorphism between $M_1=S^3_{L_1}$ and
  $M_2=S^3_{L_2}$ is isotopic to $f$.
\end{teo}
 
The above theorem can be refined to the case of manifolds containing
graphs (see \cite{Ro}).  In particular, we have the following:

\begin{teo}
  Let $M_1$ and $M_2$ be oriented, closed 3--manifolds containing
  framed graphs $T_1\subset M_1$ and $T_2\subset M_2$, respectively.
  Let $f:M_1\to M_2$ be an orientation preserving diffeomorphism such
  that $f(T_1)=T_2$ as framed graphs.  Let $L_i$ be a link in $S^3$
  which is a surgery presentation of $M_i$ such that $T_i\subset
  S^3\setminus L_i$.  There exists a sequence of handle-slides,
  blow-up moves and blow-down moves on the components of $L_1$ as well
  as handle slides moving an edge of $T_1$ over a component of $L_1$
  and blow-up and blow-down moves around edges of $T_1$, transforming $L_1\sqcup T_1$ into $L_2\sqcup T_2$ and such that the
  induced diffeomorphism between $M_1=S^3_{L_1}$ and $M_2=S^3_{L_2}$
  is isotopic to $f$.
\end{teo}

We will now discuss the idea of the proofs of Theorem \ref{T:coh_adm}
and \ref{T:T_adm}.  Suppose $(M,T,\coh)$ and $(M',T',\coh')$ are two
compatible triples which are diffeomorphic through a map $f:
(M,T,\coh)\to (M',T',\coh')$ (i.e. $f(T)=T'$ as framed graph and
$f^*(\coh')=\coh$).  Pick any two computable surgery presentations of
$(M,T,\coh)$ and $(M',T',\coh')$ through links $L$ and $L'$ in $S^3$.
Use Theorem \ref{teo:refinedkirby} to realize the map $f$ through a
sequence of handle-slides, blow-up moves and blow-down moves; then try
to follow the sequence and prove that the values of the invariants
before and after each move does not change. In order to do so, we have
to deal not only with framed links in $S^3$ but also with their
colorings and how they change during the moves. We will show in Lemma
\ref{L:handle-slide} that when a component of a link slides over a
second one, the color of the latter is modified. A similar phenomenon
happens when a blow-up/down move is applied (see Lemma
\ref{L:Stab}). But our invariants are defined only when the colors of
all the components of the links are ``generic'' (for instance in the
case of $\mathfrak{sl}_2$ when they belong to $\C\setminus \Z$); so,
if during a sequence of moves a non-generic color is produced the
corresponding invariant is not defined and the invariance cannot be
proven directly.
To bypass this problem we initially apply a $H$-stabilization along
an edge $e$ of $T$ which is colored by an element of $\A$ and color
the newly created meridian of $e$ by a ``sufficiently generic''
color. Then whenever a move in the sequence would produce a
non-generic color, before doing the move we slide the meridian over it
to change its color to a generic one. Then, at the end of the sequence
we slide back the meridian in its ``original position'' around $e$ and
remove it.

In the proof of Theorem \ref{T:coh_adm} we do not have an edge $e$ of
$T$ to create such a useful meridian, but we may apply a blow-up move
over a component of the surgery link. This produces a meridian colored
by a linear combination of generic colors.  We then consider this
meridian as part of the graph $T$ and our invariant as a linear
combination of invariants where $T$ is non-empty and has a generically
colored component; we thus may apply the preceding argument.

\subsection{Useful lemmas.}  
Before proving Theorems \ref{T:coh_adm} and \ref{T:T_adm} we prove a
series of lemma which will be used in the proof.  If $U,V,W$ are
objects of $\cat$, the duality morphisms of $\cat$ induce natural
isomorphisms
\begin{equation}
  \label{E:dualHom}
  \begin{array}{c}
  \Hom_\cat(U\otimes V,W)\cong\Hom_\cat(U,W\otimes
  V^{*})\cong\Hom_\cat(\unit,W\otimes V^{*}\otimes
  U^{*})\\\cong\Hom_\cat(\unit,V^{*}\otimes U^{*}\otimes W).  
  \end{array}
\end{equation}
Furthermore, the value of $F'$ on the planar theta-graph with coupons
and edges colored by $U$, $V$ and $W$ (see Figure \ref{F:theta})
\begin{figure}
  \centering
  $\brk{\vp,\vp'}=F'\left(\epsh{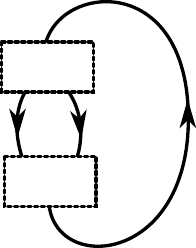}{15ex}
    \put(-54,1){\ms U}\put(-28,1){\ms V}\put(-11,1){\ms W}
    \put(-43,16){$\vp$}\put(-43,-15){{$\vp'$}}\right)$
  \caption{Pairing of morphisms in $\cat$.}
  \label{F:theta}
\end{figure}
gives a pairing
$$\brk{\cdot,\cdot}:\Hom_\cat(U\otimes V,W)\otimes\Hom_\cat(W,U\otimes V)
\rightarrow \FK$$ which is compatible with these isomorphisms.

By definition of $F'$ and this pairing, we have 
if $W\in\A$,
$$\vp\circ\vp'=\qd(W)^{-1}\brk{\vp,\vp'}\Id_W$$
for any $\vp\in\Hom_\cat(U\otimes V,W)$ and
$\vp'\in\Hom_\cat(W,U\otimes V)$.

\begin{lemma} Let $V\in\A$ and $t\in\TT$ be such that
  $V\otimes\ve^t\in\A$.  Then $\qd(V\otimes\ve^t)=\qd(V)$.  In other
  words, $\qd$ factors through a map on the simple $\TT$-orbits.
\end{lemma}
\begin{proof} Let $V_t=V\otimes\ve^t$ and consider the following
  $\cat$-colored ribbon graph with coupons:
  $$\Gamma=\epsh{fig25}{10ex}\put(-37,0){\ms{V}}\put(-9,0){\ms{V_t}}
  \put(-18,0){\ms{\ve^t}}\put(-28,10){\ms{\Id}}\put(-28,-10){\ms{\Id}}.$$
  To compute $F'(\Gamma)$, one can cut the edge colored by $V_t$.  The
  image under $F$ of the obtained tangle is $\Id_{V_t}$ so
  $F'(\Gamma)=\qd(V_t)$.  On the other hand, one can cut the edge
  colored by $V$ and so
  $$F'(\Gamma)\Id_V=
  \qd(V)F\left(\epsh{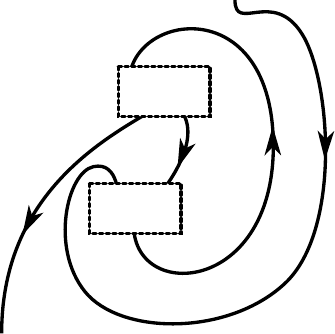}{10ex}\put(-45,-6){\ms{V}}
    \put(-15,2){\ms{V_t}}\put(0,2){\ms{V}}\put(-25,2){\ms{\ve^t}}\
  \right)=\qd(V)F\left(\epsh{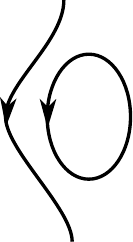}{10ex}\put(-29,2){\ms{V}}
    \put(-13,2){\ms{\ve^t}}\ \right)=\qd(V)\qdim(\ve^t)\Id_V=\qd(V)\Id_V.$$
  Hence $\qd(V_t)=F'(\Gamma)=\qd(V)$.
\end{proof}

The following proposition gives a relationship between vanishings of
$F$ and $F'$.
\begin{prop}\label{P:FvsF'} 
  Let $\wb V=((V_1,\ve_1),\ldots,(V_k,\ve_k))$ and $\wb
  W=((W_1,\ve'_1),\ldots,(W_l,\ve'_l))$ be objects of $Rib_\cat$ such
  that at least one of the $V_i$ belongs to $\A$ and for any $1\leq
  i\leq k$ and $1\leq j\leq l$ the objects $V_i$ and $W_j$ are in
  different graded pieces of $\cat$.  Suppose that $T$ is a
  $\cat$-colored ribbon graph in $\Hom_{Rib_\cat}(\wb V,\wb W)$.  Then
  the following two are equivalent
  \begin{itemize}
  \item $F(T)=0$
  \item For all $T'\in\Hom_{Rib_\cat}(\wb W,\wb V)$, one has
    $F'(\tr(T\circ T'))=0$ where $\tr(T\circ T')$ is the trace in
    $Rib_\cat$, i.e.  the braid closure of the tangle formed from $T'$
    on top of $T$.
  \end{itemize}
\end{prop}
\begin{proof}
  Assume that $F(T)=0$ and that $T'$ is as in the proposition.  Then one can
  compute $F'(\tr(T\circ T'))$ by cutting the edge of $T$ colored by $V_i\in
  A$.  This produce a 1-1-tangle containing $T$ as a subgraph and thus the
  functor $F$ vanish on it.  Hence $F'(\tr(T\circ T'))=0$.
  
  Suppose now that $F(T)=f\neq0$.  Let $g\in\Gr$ be such that $F(\wb
  V)=V_1\otimes \cdots \otimes V_k\in\cat_g$.  As $f\neq0$, the
  $\Gr$-grading impose that $F(\wb W)=W_1\otimes \cdots \otimes
  W_l\in\cat_g$.  Let $h\in\Gr\setminus(\Xr\cup g^{-1}\Xr)$ and let
  $U$ be a simple object of $\cat_h$.  Then $\wb V \otimes U$ and $\wb
  W\otimes U$ are semi-simple objects because they belongs to
  $\cat_{gh}$ where $gh\notin\Xr$.  Now $f\otimes\Id_U\neq0$ so there
  is a simple object $U'\in \cat_{gh}$ and maps $i:U'\to \wb V\otimes
  U$, $p:\wb W\otimes U\to U'$ such that $p\circ(f\otimes\Id_U)\circ
  i=\Id_{U'}$.  Let $T''\supset T$ be the ribbon graph with coupons
  colored by $p$ and $i$ corresponding to the expression
  $p\circ(f\otimes\Id_U)\circ i$.  Then $F(T''-\Id_{(U',+)})=0$ hence
  the braid closures satisfy
  $F'(\tr(T''))=F'(\tr(\Id_{(U',+)}))=\qd(U')$.  Now for $T'$ the
  complement of $T$ in $\tr(T'')$, we have $F'(\tr(T\circ
  T'))=\qd(U')\neq0$.
\end{proof}

\begin{lemma}[Fusion lemma] \label{L:fusion} Let $g\in\Gr\setminus
  \X$.  Let $U,V\in \cat$, with $U\otimes V\in\cat_g$ and let $W$ be a
  simple object of $\cat_g$ (in particular $W\in \A$).  Then the pairing
  $\brk{\cdot,\cdot}:\Hom_\cat(U\otimes
  V,W)\otimes\Hom_\cat(W,U\otimes V)\to\FK$ is non degenerate.
  Furthermore,
  \begin{equation}\label{E:DepIDofUoV}
    \Id_{U\otimes V}=\sum_{W}\sum_{i}\qd(W)x^{W,i}\circ x_{W,i}
  \end{equation}
  where $W$ runs through a representative set of isomorphism classes
  of simple objects of $\cat_{g}$, $(x_{W,i})_i$ is a basis of
  $\Hom_\cat(U\otimes V,W)$ and $(x^{W,i})_i$ is the dual basis of
  $\Hom_\cat(W,U\otimes V)$.  Let $(W_j)_{j=1\cdots n}$ be a finite
  set of simple modules of $\cat_g$ whose isomorphism classes
  represent all simple $\TT$-orbits of $\cat_g$.  Then the above sum
  can be rewritten as
  $$\Id_{U\otimes V}=\sum_{j=1}^{n}\qd(W_j)\sum_{t\in\TT}\sum_i x^{j,t,i}\circ
  x_{j,t,i}$$ where $(x_{j,t,i})_i$ is a basis of $\Hom_\cat(U\otimes
  V,W_j\otimes \ve^{t})$ and $(x^{j,t,i})_i$ is the dual basis of
  $\Hom_\cat(W_j\otimes \ve^{t},U\otimes V)$.
\end{lemma}
\begin{proof}
  As $\cat_{g}$ is semi-simple, we can write $U\otimes
  V\simeq\bigoplus_k W_k^{\oplus n_k}$ where $W_k$ are non-isomorphic
  simple objects of $\cat_{g}$.  So there are maps $x_{W_k,i}:U\otimes
  V\to W_k$ and $y^{W_k,i}:W_k\to U\otimes V$ such that
  $x_{W_k,i}\circ y^{W_k,j}=\delta_i^j\Id_{W_k}$ and $\Id_{U\otimes
    V}=\sum_{k}\sum_i y^{W_k,i}\circ x_{W_k,i}$ Now if
  $\vp\in\Hom_\cat(U\otimes V,W)$ is not zero, then
  $\vp=\sum_{k}\sum_i \vp\circ y^{W_k,i}\circ x_{W_k,i}$ so there
  exist $k,i$ with $\vp\circ y^{W_k,i}\neq0\in\Hom_\cat(W_k,W)$.  This
  implies that $W_k$ and $W$ are isomorphic and if
  $\vp'=y^{W_k,i}\circ(\vp\circ y^{W_k,i})^{-1}\in
  \Hom_\cat(W,U\otimes V)$, then we have $\vp\circ\vp'=\Id_W$ implying
  $\brk{\vp,\vp'}=\qd(W)\neq0$.  Hence the pairing is not degenerate.

  Furthermore, $x_{W_k,i}\circ y^{W_k,j}=\delta_i^j\Id_{W_k}$ so
  $\brk{x_{W_k,i}, y^{W_k,j}}=\delta_i^j\qd({W_k})$ which implies that the
  basis $(\qd(W_k)^{-1}y^{W_k,i})_{i=1\cdots n_k}$ is the dual basis of
  $(x_{W_k,i})_{i=1\cdots n_k}$.  Hence the identity of $U\otimes V$ can be
  decomposed as in Equation \eqref{E:DepIDofUoV}.  Remark that for all but a
  finite number of simple modules $W$ the corresponding term in the sum is
  zero so that the sum in Equation \eqref{E:DepIDofUoV} is finite.

  We can then rewrite the sum by grouping the simple objects $W$ that are in
  the same simple $\TT$-orbit.  Here we use that the action of $\TT$ is free
  so that any $W$ is isomorphic to an unique module of the form
  $W_j\otimes\ve^t$.  Then using that $\qd(W_j\otimes\ve^t)=\qd(W_j)$ we
  obtain the second decomposition of $\Id_{U\otimes V}$.
\end{proof}

Let $T$ be a homogeneous $\cat$-colored ribbon graph with coloring
$\gcolor$, whose associated $\Gr$-coloring we denote $\dg\gcolor$.
Let $K$ be an oriented knot in the complement of $T$ and let $\Sigma$
be a Seifert surface for $K$.  Then the homological intersection of
the homology class relative to $K$ represented by $\Sigma$ with
$\dg\gcolor$ is an element of $\Gr$ depending on $K$.  Let
$$\lkG(K,\dg{\gcolor})=\prod_{p\in \Sigma\cap T} \dg{\gcolor}(e_p)^{\sign(p)}$$
where $e_p$ is the edge of $T$ containing $p$ and $sign(p)$ is the
sign of the intersection of $\Sigma$ and $e_p$ at $p$.  Clearly
$\lkG(K,\dg\gcolor)$ depends only on $[K]\in H_1(S^3\setminus T;\Z)$
and it can be computed from a diagram of $T$ by the following rules:
(1) if $e$ is an edge of $T$ which is below $K$ then $e$ does not
contribute to $\lkG(K,\dg\gcolor)$, (2) if $e$ is above $K$ and the
pair $(K,e)$ form a positive crossing then $e$ contributes
$\dg\gcolor(e)$ to $\lkG(K,\dg\gcolor)$, (3) if $e$ is above $K$ and
the pair $(K,e)$ form a negative crossing then $e$ contributes
$\dg\gcolor(e)^{-1}$ to $\lkG(K,\dg\gcolor)$. These rules can be
summarized by:
\begin{equation}
  \label{E:lkG}
  \epsh{fig26}{7ex}\put(-6,16){\small$e$}\put(-18,16){\ms{K}}\to 
  \dg\gcolor(e), \quad
  \epsh{fig27}{7ex}\put(-6,16){\ms{K}}\put(-18,16){\small$e$}\to 
  \dg\gcolor(e)^{-1},\quad
  \epsh{fig27}{7ex}\put(-6,16){\small$e$}\put(-18,16){\ms{K}}\to 1, 
  \quad \epsh{fig26}{7ex}\put(-6,16){\ms{K}}\put(-18,16){\small$e$}\to 1,
\end{equation}
where we are using a multiplicative notation for the group ${\Gr}$. 
We extend the above definition to the case when $K\subset T$ by
defining $\lkG(K,\dg\gcolor)$ to be $\lkG(K_\sslash,\dg\gcolor)$ where
$K_\sslash$ is a parallel copy of $K$ given by the framing of $T$.  In
particular, if $L\cup T$ is a $\cat$-colored ribbon graph with
$\Gr$-coloring $\dg\gcolor$ and linking matrix $\lk$, and if $K_i$ is
the $i$\textsuperscript{th} component of $L$ then
$$\lkG(K_i,\dg\gcolor)=\prod_j\dg\gcolor(K_j)^{\lk_{ij}}
=\prod_j\dg\gcolor(K_j)^{\lk_{ji}}$$
even if $\lk$ is not symmetric.  Thus, if $\dg\gcolor$ is in the
image of $\Phi$ then $\lkG(K_i,\dg\gcolor)=1\in\Gr$.
In particular, we have the following remark.
\begin{rem}\label{R:colorlkG}
  Let $(M,T,\coh)$ be a compatible triple with a computable surgery
  presentation via $L$.  Let $g_\coh$ be the $\Gr$-coloring on $L\cup
  T$ coming from $\coh$.  If $L_i$ is any component of $L$ then
  $\lkG(L_i,g_\coh)=1$.
\end{rem}

\begin{lemma}\label{L:lifts}
  Let $T$ be a $\A$-graph with corresponding $\Gr$-coloring
  $\dg\gcolor$.  Suppose $K$ is a circle component of $T$ colored by
  $\ve^t$ for some $t\in\TT$.  Let $T'$ be the $\A$-graph $T$ with the
  component $K$ removed.  Then
  $$F'(T)=\lkG(K,\dg\gcolor)\pow{t}F'(T').$$
\end{lemma}
\begin{proof}
  We consider a diagram representing $T$ and use the skein relation
  \eqref{E:skeineps} to pull the knot $K$ above the rest of the
  diagram, obtaining the disjoint union of $T'=T\setminus K$ with $K$.
  Using Equation \eqref{E:lkG}, one can see that
  $F'(T)=\lkG(K,\dg\gcolor)\pow{t}F'((T\setminus K)\sqcup K)$.
  Combining this equality with Equation \eqref{E:F'T'T} we have
  $$F'(T)=\lkG(K,\dg\gcolor)\pow{t}F'(T'\sqcup K)=
  \lkG(K,\dg\gcolor)\pow{t}F'(T')F(K)$$ and the lemma follows from the
  fact that $F(K)=1$ which can be easily deduced from the braiding and
  modified dimension of $\ve^t$ since $c_{\ve^t,\ve^t}\circ
  c_{\ve^t,\ve^t}=\Id_{\ve^t\otimes \ve^t}$ and $\qdim(\ve^t)=1$.
\end{proof}

\begin{lemma}\label{L:lifts2}
  Let $T$ be a $\A$-graph with $\Gr$-coloring $\dg\gcolor$.  Suppose
  $K$ is a circle component of $T$ colored by $V\in\A$ with
  $\lkG(K,\dg\gcolor)=1$.  Let $T'$ be the $\A$-graph $T$ where the
  color of the component $K$ is changed to $V\otimes\ve^{t}$ for some
  $t\in\TT$ then $F'(T')=F'(T).$
\end{lemma}
\begin{proof}
  Let $T''=T\cup K_\sslash$ where $K_\sslash$ is a parallel copy of
  $K$ colored with $\ve^t$.  Then basic properties of $F$ imply
  $F'(T')=F'(T'')$.  The $\Gr$-coloring $\dg\gcolor$ of $T$ extends to a
  $\Gr$-coloring $\dg{\gcolor'}$ of $T''$ with value $1$ on $K_\sslash$.
  Now one can apply Lemma \ref{L:lifts} to $T''$ which gives
  $F'(T'')=\lkG(K_\sslash,\dg{\gcolor'})\pow{t}F'(T)$.  But
  $\lkG(K_\sslash,\dg{\gcolor'})=\lkG(K,\dg\gcolor)=1$ so
  $F'(T')=F'(T)$.
\end{proof}

\begin{lemma}[Handle slide]\label{L:handle-slide}
  Let $T$ be a $\cat$-colored ribbon graph with a homogeneous
  $\cat$-coloring $\gcolor$.  Suppose $K$ is a circle component of $T$
  colored by a Kirby color $\Omega_g$ of degree $g\in\Gr\setminus\X$.
  Let $e$ be an oriented edge of $T\setminus K$ homogeneously colored
  by $\gcolor(e)\in \cat$.  Let $T'$ be a $\cat$-colored ribbon graph
  obtained from $T$ by a handle-slide of $e$ along $K$ with the color
  $\Omega_g$ of $K$ replaced by a Kirby color $\Omega_h$ of degree
  $h=\dg{\gcolor}(e)g$.  If $\lkG(K,\dg\gcolor)=1$ and $h\notin\X$
  then $F'(T')=F'(T)$.
\end{lemma}
\begin{proof}
  Let $O(g)$ and $O(h)$ be the sets of colors that appear in
  $\Omega_g$ and $\Omega_h$, respectively.  Since
  $\lkG(K,\dg\gcolor)=1$, Lemma \ref{L:lifts2} implies that $T$ can be
  modified by adding a parallel copy of $K$ colored by $\ve^t$ without
  changing its value under $F'$.  We use this fact with the fusion of
  Lemma \ref{L:fusion} to prove the lemma.  In particular, let $U$ be
  one of the modules appearing in $\gcolor(e)$.  We will show that the
  edge $e$ colored by $U$ can slide over $K$, then the result follows
  since $\gcolor(e)$ is a the formal linear combination of modules.
  In the following equation the symbol $\eg$ displays the equality of
  the values of the corresponding local diagrams under $F'$.
  \begin{align*}
    \sum_{V\in O(g)}\qd(V)\epsh{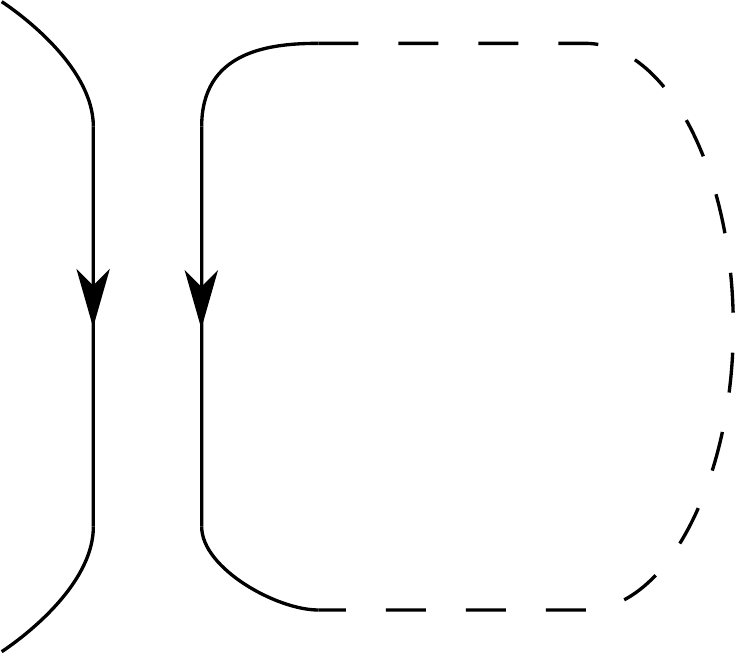}{14ex}\put(-68,5){\ms U}\put(-45,5){\ms V}
    &\eg\sum_{\ms{
        \begin{array}{c}
          (V,W)\in O(g)\times O(h)\\
          t\in\Z,\,W_t=W\otimes\ve^{t}
        \end{array}}}\hspace{-7ex}\qd(V)\qd(W)\sum_i
  \epsh{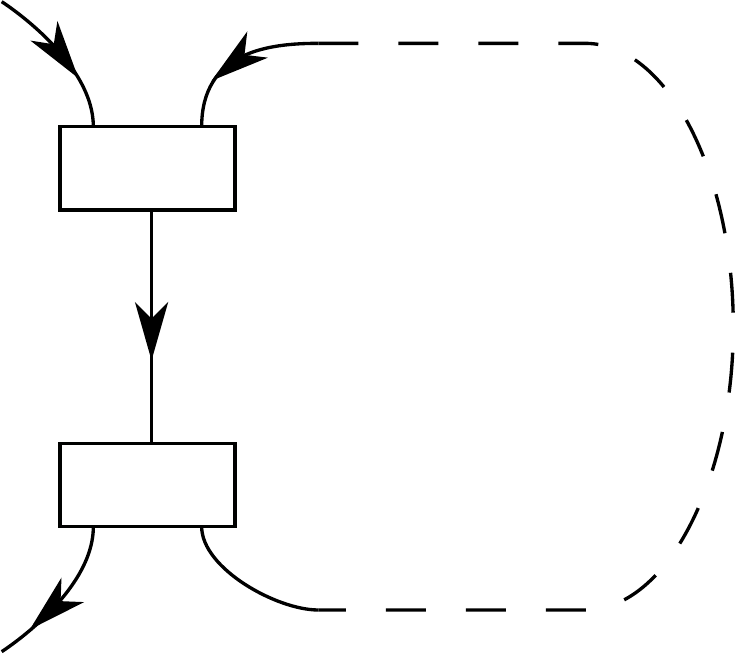}{14ex}\put(-68,22){\ms U}\put(-68,-22){\ms U}
    \put(-45,22){\ms V}\put(-66,0){\ms{W_t}}
    \put(-57,16){\ms{x_i}}\put(-57,-15){\ms{x^i}}
    \\
    &\eg\sum_{\ms{
        \begin{array}{c}
          (V,W)\in O(g)\times O(h)\\
          t\in\Z,\,W_t=W\otimes\ve^{t}
        \end{array}}}\hspace{-7ex}\qd(V)\qd(W)\sum_i
    \epsh{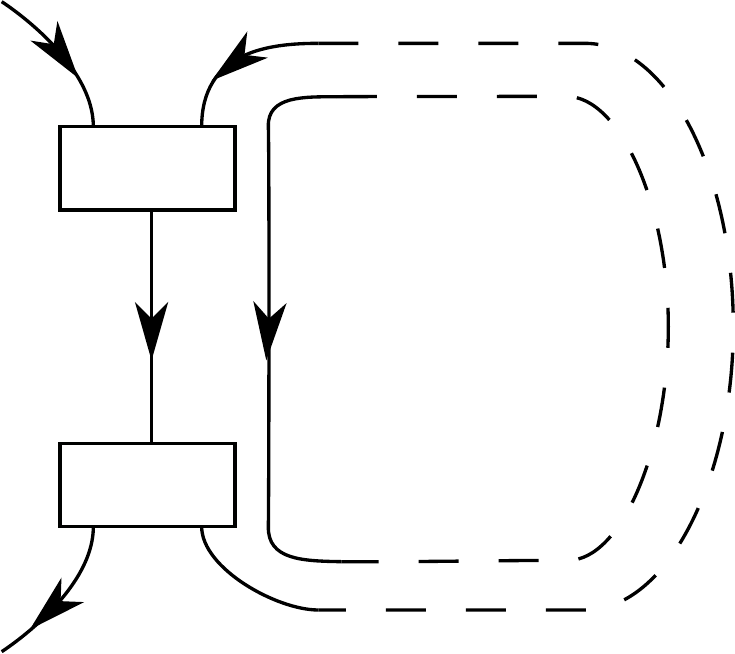}{14ex}\put(-68,22){\ms U}\put(-68,-22){\ms U}
    \put(-52,27){\ms V}\put(-66,0){\ms{W_t}}\put(-40,0){\ms{\ve^{-t}}}
    \put(-57,16){\ms{x_i}}\put(-57,-15){\ms{x^i}}
    \\
    &\eg\sum_{\ms{
        \begin{array}{c}
          (V,W)\in O(g)\times O(h)\\
          t\in\Z,\,V_t=V\otimes\ve^{-t}
        \end{array}}}\hspace{-7ex}\qd(V)\qd(W)\sum_i
    \epsh{K2b}{14ex}\put(-68,22){\ms U}\put(-68,-22){\ms U}
    \put(-44,22){\ms{V_t}}\put(-64,0){\ms{W}}
    \put(-57,16){\ms{y_i}}\put(-57,-15){\ms{y^i}}
    \\
    &\eg\sum_{\ms{
        \begin{array}{c}
          (V,W)\in O(g)\times O(h)\\
          t\in\Z,\,V_t=V\otimes\ve^{-t}
        \end{array}}}\hspace{-7ex}\qd(V)\qd(W)\sum_i
    \epsh{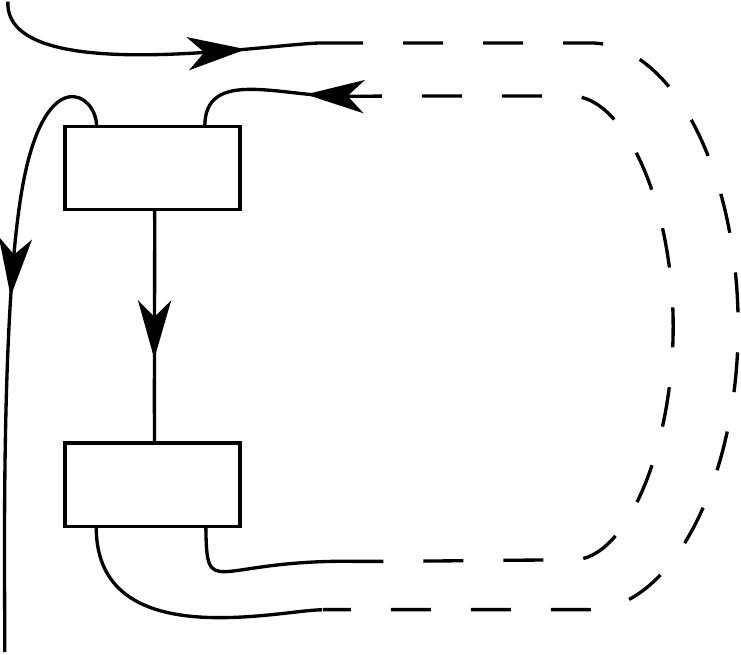}{14ex}\put(-50,30){\ms U}\put(-66,3){\ms U}
    \put(-50,2){\ms{V_t}}\put(-43,17){\ms{W}}
    \put(-57,16){\ms{z_i}}\put(-57,-15){\ms{z^i}}
    \\
    &\eg\sum_{W\in O(h)}\qd(W)
    \epsh{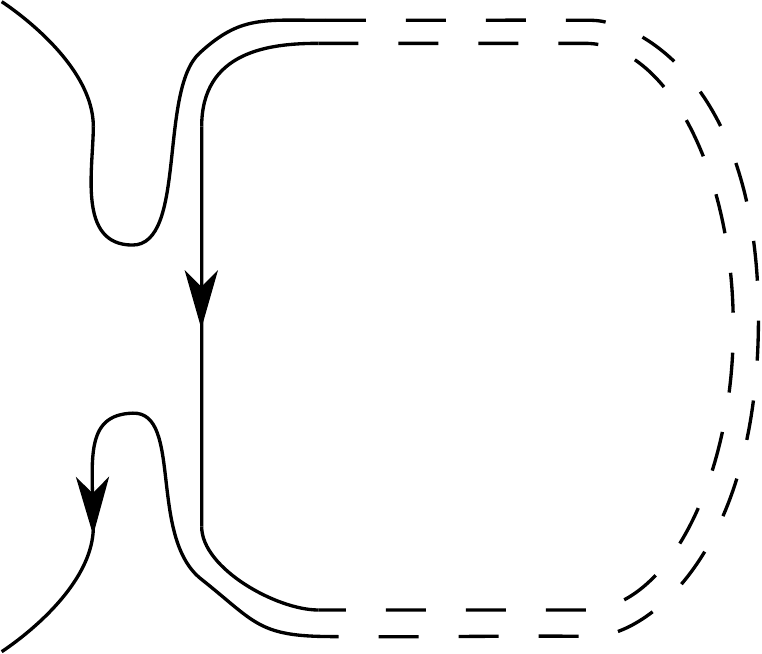}{14ex}\put(-70,-20){\ms U}\put(-50,2){\ms{W}}
  \end{align*}
  In these sums, $\{x_i\}_i$ and $\{x^i\}_i$ are arbitrary dual bases
  of the multiplicity modules $\Hom(W_t, U\otimes V)$ and
  $\Hom(U\otimes V,W_t)$ respectively; $\{y_i\}_i$ is a base deduced
  from $\{x_i\}_i$ using Lemma \ref{L:Tacts}, $\{y^i\}_i$ is its dual
  base; $\{z_i\}_i$ is a basis obtained from $\{y^i\}_i$ using the
  isomorphisms of the multiplicity modules and $\{z^i\}_i$ denotes its
  dual bases.  The first and last equality are obtained from
  Lemma~\ref{L:fusion}, the second from Lemma~\ref{L:lifts}, the third
  comes from Lemma~\ref{L:Tacts} and the fourth is an isotopy plus a
  local modification of the coupons.
\end{proof}

\begin{lemma}\label{L:Stab}
  There exists scalars $\Delta_{\pm}\in\FK$ such that for any module
  $V$ of degree $g\in \Gr\setminus \X$ we have
  \begin{equation}\label{E:FTwistedH}
    F\left(\epsh{fig1}{12ex}\put(-33,-2){\ms{\Omega_g}}\put(-26,-20){\ms{V}}
    \right)=\Delta_-\Id_V\quad\text{ and }\quad
    F\left(\epsh{fig2}{12ex}\put(-33,-2){\ms{\Omega_g}}\put(-26,-20){\ms{V}}
    \right)=\Delta_+\Id_V.
  \end{equation}
\end{lemma}
\begin{proof}
  Suppose first $V\in \A$ and let $\Delta_{\pm}(V)\in\FK$ be the
  quantities determined by Equation \eqref{E:FTwistedH}.  We want to
  show that these quantities do not depend on $V$.  We will do this
  for $\Delta_-$; a similar argument implies the result for
  $\Delta_+$.  If $g'\in \Gr\setminus\X$ then there exists
  $h\in\Gr\setminus (\X\cup (g^{-1}\X)\cup ({g'}^{-1}\X))$.  Let $W$
  be a simple object of $\cat_h$.  Now for any $V\in\A$ of degree $g$,
  we can use the handle slide property of Lemma \ref{L:handle-slide}
  and Proposition \ref{P:FvsF'} to obtain
 \begin{align*}
   F\left(\epsh{fig1}{12ex}\put(-33,-2){\ms{\Omega_g}}\put(-26,-20){\ms{V}}
     \epsh{fig10}{12ex}\put(-16,-2){\ms{W}}\right)
   =F\left(\epsh{fig1}{12ex}\put(-37,-2){\ms{ \Omega_{gh} }}
  \put(-16,-20){\ms{V\hspace{-2pt}\otimes\hspace{-2pt} W}}\right)
 =F\left(\epsh{fig10}{12ex}\put(-15,-2){\ms{V}}\epsh{fig1}{12ex}
  \put(-33,-2){\ms{\Omega_h}}\put(-27,-20){\ms{W}}\right).
  \end{align*}
  The left (resp. right) hand side of the last equation is equal to
  $\Delta_{-}(V)\Id_{V\otimes W}$ (resp. $\Delta_{-}(W)\Id_{V\otimes
    W}$) and so $\Delta_{-}(V)=\Delta_{-}(W)$.  Similarly, for any
  $V'\in \A$ of degree $g'$ we have $\Delta_{-}(V')=\Delta_{-}(W)$ and
  so $\Delta_{-}(V)=\Delta_{-}(V')$.

  Now let $V$ be any object of the semi-simple category $\cat_g$ and
  let $f_V\in\End_\cat(V)$ be the endomorphism represented by the
  first ribbon graph in the lemma.  Then $V=\bigoplus_i V_i$ with
  $V_i\in \A$.  Let $\alpha_i:V_i\to V$ and $\beta_i:V\to V_i$ for
  $i=1,\dots, n$ such that $\Id_V=\sum_{i=1}^n\alpha_i\beta_i$.  Then
  $f_V=\sum_{i=1}^nf_V\alpha_i\beta_i=
  \sum_{i=1}^n\alpha_i\Delta_-(V_i)\beta_i=\Delta_-\Id_V$.
\end{proof}
Remark that in the previous Lemma $V$ can be a tensor product of $n$
homogeneous objects.  The corresponding diagrammatic relation can be
obtained by replacing the edge colored by $V$ in Lemma \ref{L:Stab} by
$n$ parallel strands.
\begin{lemma}\label{L:val_coh}
  Let $(M,T,\coh)$ be a compatible triple and consider $\coh$ as a map
  on $H_1(M\setminus T,\Z)$ with values in $\Gr$.  Suppose that
  $g\in\Gr$ is a value of $\coh$, then there exists a surgery
  presentation $L\cup T$ of $(M,T,\coh)$ for which the $\Gr$-color of
  a component of $L$ is $g$.
\end{lemma}
\begin{proof}
  Start with any presentation $L'\cup T$ of $(M,T,\coh)$.
  $H_1(M\setminus T,\Z)$ is generated by the meridians around the
  edges of $L'\cup T$.  So there exists integers $n_i\in\Z$ such that
  $\coh(\sum_i n_im_i)=\prod_i g_\coh(e_i)^{n_i}=g$ where $\{m_i\}_i$
  is the set of meridians around the edges $\{e_i\}_i$ of $L'\cup T$.
  We add to $L'$ a disjoint unknot $U$ with framing $1$.  The
  $\Gr$-color of $U$ is $1\in\Gr$. Then we can slide $n_i$ times the
  edge $e_i$ on $U$ so that the resulting $\Gr$-coloring of $U$
  becomes $g_\coh(U)=g$.  For $L=L'\cup U$, we get the desired
  presentation of $(M,T,\coh)$.
\end{proof}

\subsection{Proofs.}
\begin{proof}[Proof of Theorem \ref{T:T_adm}.]
  Since $(M,T,\coh)$ is $T$-admissible then $T\neq \emptyset$ and has
  an edge $e$ colored by $\beta\in \A$.  Pick a surgery presentation
  of $(M,T,\coh)$ over a link $L\subset S^3$.  If $L=\emptyset$ then
  since $e\in T$ is colored by $\beta\in \A$, then $T\cup L$ is
  computable.  Suppose now that the presentation is not computable.
  Then $L\neq \emptyset$; let $L_i,\  i=1,...,n$ be the components of $L$ and $S=\{i\in \{1,...,n\}\ {\rm s.t.}\ g_\coh(L_i)\in \X\}$.
  Apply a $H$-stabilization along $e$ to create a new component $m$
  colored by $\alpha\in \A$ where $\alpha$ is chosen so that
  $\dg{\alpha}g_\coh(L_i)\notin \X$ for all $i\in S$.  Then perform
  handle-slides to slide $m$ over all the components $L_i$ with $i\in
  S$.
  The resulting presentation is computable and this proves the first
  statement.

  Let $T'$ be $T$ or a $H$-stabilization of $T$ such that there are
computable presentations $(L\cup T')$ and $(L'\cup T')$ of $(M,T',\coh')$. Next we prove that $L$ and
  $L'$ give the same
  invariant; then at the end of the proof we will show that two different $H$-stabilizations of $T$ yield the same invariant.

  \newcommand{\Th}{{T_H}} \newcommand{\cohh}{{\omega_H}}
Let $\Th=T'\cup m$ be a graph obtained from a $H$-stabilization over an edge $e\subset T'$ colored by $\beta  \in \A$.    Here $m$ is colored by $\alpha \in \A$.   Then $(L\cup \Th)$ and  $(L'\cup \Th)$ are computable surgery presentations of  $(M,\Th,\cohh)$ where $\cohh$ is given in the definition of a $H$-stabilization.  In both presentations, we assumed that the component $m$ is a standard meridian of $e$ in $S^3\setminus  (L\cup T')$ and $S^3\setminus (L'\cup T')$ respectively.
Furthermore, one has
  $$F'(L\cup \Th)=\bk H F'(L\cup T')\text{ and }
  F'(L'\cup \Th)=\bk H F'(L'\cup T'),$$ 
  where $\langle H\rangle$ is the value of the long Hopf-link whose
  closed component is colored by $\alpha$ and whose long component is
  colored by $\beta$.  Definition \ref{D:GmodularCat} implies $\langle
  H\rangle\neq 0$ so it is enough to show that there exists an
  $\alpha$ for which
  $$ 
  \frac{F'(L\cup \Th)}{\langle
    H\rangle\Delta_+^r\Delta_-^s}=\frac{F'(L'\cup \Th)}{\langle
    H\rangle\Delta_+^{r'}\Delta_-^{s'}}.
  $$
  where $r,s$ (resp. $r',s'$) the number of positive and negative
  eigenvalues of the linking form of $L$ (resp. $L'$).  For this, we
  are going to show that for a suitable value of $\alpha$, the two
  presentation $L\cup \Th$ and $L'\cup \Th$ are related through
  computable presentations by elementary moves.  Then applying Lemmas
  \ref{L:handle-slide} and \ref{L:Stab} will give the desired equality
  of invariants.

  Let $L=L^0\stackrel{s_1}{\to}\ldots \stackrel{s_k}{\to}L^k=L'$ be a
  sequence of handle-slides, blow-up moves and blow-down moves
  connecting the two presentations and inducing a diffeomorphism $f$
  between $S^3_{L}\setminus \Th$ and $S^3_{L'}\setminus \Th$ such that
  $f^*(\cohh)=\cohh$. 
  It may happen that $L_i$ is not computable for some $i$.  We will
  use handle-slides of $m$ to modify $L_i$ to a computable
  presentation.  We choose $\alpha$ ``generically'' with respect to
  the sequence $L^0\stackrel{s_1}{\to}\ldots \stackrel{s_k}{\to}L^k$
  where generically means that $\alpha$ is chosen by the finite list
  of conditions given below.
  
  Without loss of generality we can assume that all the blow-ups
  are at the beginning of the sequence and all the blow-downs at its end.  So suppose $s_1,...,s_t$ are blow-ups,
  $s_{t+1},...,s_n$ are handle-slides and $s_{n+1},...,s_k$ are
  blow-downs.

  Suppose the blow-up $s_1$ is on an edge $e_1$ of $L^{0}\cup T$
  colored by an object of degree $g_\cohh(e_1)\in \Gr\setminus \X$.
  Then $L^1$ is a computable surgery presentation.  Then we keep the
  move unchanged.  If instead $g_\cohh(e_1)\in \X$ then $L^1$ is not
  computable.  Suppose that $\alpha$ is such that $\dg{\alpha},\dg{\alpha}
  g_\cohh(e_1)\in \Gr\setminus \X$; this imposes two conditions on
  $\alpha$ which are part of our genericity hypothesis.  Then we can
   do the blow-up on $m$ and then slide the edge $e$ on the newly
  created component.  This leads to a computable surgery presentation
  similar to $L^1$ except for the position of $m$ (which is now linked with the component containing $e_1$ in $L^1$) and the color of the
  new component which is now $\dg{\alpha} g_\cohh(e_1)$.  This
  presentation is obtained from $L^0$ by two elementary moves between
  three computable presentations (a computable blow up and a computable sliding).
  Similarly, for $i=2,...,t$, the presentations $L^{i-1}$ and $L^{i}$
  can be made computable by sliding $m$ on the created
  components.

  Suppose now the handle-slide $s_{t+1}$ is on an edge of $L^{t}\cup T$
  colored by $\gamma$
  over a component $L^t_i$ of $L^t$ colored by $g_\cohh(L^t_i)\in
  {\Gr}$ (recall that $g_\cohh(L^t_i)\in {\Gr}$ was defined as the
  value of $\cohh$ on the meridian of $L^t_i$; see Definition
  \ref{def:adm}). If $g_\cohh(L^{t}_i)\dg{\gamma}\in \Gr\setminus \X$
  then $L^{t+1}$ is a computable presentation.
    Otherwise, suppose that $\alpha$ is such that
  $\dg{\alpha} g_\cohh(L^{t}_i), \dg{\alpha}
  g_\cohh(L^{t}_i)\dg{\gamma}\in \Gr\setminus \X$; this imposes two
  conditions on $\alpha$ which are part of our genericity hypothesis.
  Slide $m$ over $L^t_i$ and perform the handle slide $s_{t+1}$.  The
  result is $L^{t+1}$ where $m$ has moved but is still colored by $\alpha$ and
  where the color of $L^{t+1}_i$ changed to
  $g_{\coh_1}(L^{t+1}_i)\dg{\alpha}$; the rest of $L^{t+1}$ is unchanged. 
  
  Proceed as above to follow the rest of the sequence of handle-slides
  $s_{t+2},...,s_n$: each time we need to slide $m$ over a component
  to make a handle-slide computable we add some conditions on $\alpha$
  to our list. The list of conditions imposed will be finite and by
  Hypothesis \eqref{I4:DefGmodularCat} of Definition
  \ref{D:GmodularCat} there exists an $\alpha$ such that all of them
  are fulfilled.

  Thus, after performing the last handle-slide $s_n$ we arrive at a
  computable presentation $L_{n}$ (note here we may have used $m$).
  Since $s_{n+1}$ is a blow-down then $L_{n+1}$ is computable.
  Similarly, we can perform all the rest of the blow-downs.

  After the last blow-down one gets a computable presentation via a
  link $L'$ of the compatible triple $(M,\Th,\cohh)$ where $\Th$ is
  $T'$ together with the new component $m$ colored by $\alpha$.
  The new component $m$ could be linked with $L'$, but  
  since $m$ is by construction isotopic
  to a meridian of $e$ in $M$ one can find an isotopy of $m$ in $M$
  which brings it in the position of a small meridian around $e$ in
  the presentation $L'$ where $L'$ is colored by $g_{\cohh}$.  This
  isotopy can be decomposed into a sequence $L'\cup
  m\stackrel{h_1}{\to}\ldots \stackrel{h_l}{\to}L'\cup m$ of handle
  slides of $m$ over the components of $L'$.  At each step $s\in
  \{1,\dots ,l\}$ the color of a component $L'_i$ of $L'$ will be of the
  form $g_{\cohh}(L'_i)\dg{\alpha}^{x^i _s}$ for some $x^i_s\in \Z$
  depending on both the component $L'_i$ and the step $s$.  At the end
  of the isotopy the color of the component $L'_i$ of $L'$ is
  $g_{\cohh}(L'_i)\in {\Gr}\setminus\dg{X}$ and so $x^i_l=0$ for all $i$.  Thus, in order to be
  able to perform the sequence of slidings $h_s$ of $m$ over the
  components of $L'$ bringing $m$ back to its position of meridian of
  $e$ it is sufficient\footnote{see erratum in Appendix \ref{A:err}} to add to the above list of genericity
  conditions the conditions
  $$g_{\cohh}(L'_i)\dg{\alpha}^{x^i_s}\notin \dg{X}$$
  for every component $L'_i$ of $L$ and for all $s\in \{1,...,l-1\}$.
  The union of these conditions and those found precedingly can be
  satisfied for a suitable choice of $\alpha$ since no union of
  finitely many translates of $\dg{X}$ covers ${\Gr}$ (see Definition
  \ref{D:GmodularCat}).  Thus, we obtain a sequence of moves
  connecting the two presentations provided by $L$ and $L'$ through
  computable presentations and the two computable
  presentations give the same invariant:
  \begin{equation}
    \frac{F'(L\cup T')}{\Delta_+^r\Delta_-^s}=
    \frac{F'(L'\cup T')}{\Delta_+^{r'}\Delta_-^{s'}}.
  \end{equation}
  
  This proves that the invariant of a $T$-admissible triple
  $(M,T',\coh')$ is well defined as soon as there is a computable
  surgery presentation.  Now starting with a $T$-admissible triple
  $(M,T,\coh)$, we consider two different $H$-stabilizations over two
  edges $e$ and $e'$, both leading to a computable presentation.  One
  can apply a $H$-stabilization to both $e$ and $e'$.  A computable
  presentation of this new triple can be obtained from either
  of the two
  computable presentation by adding a meridian to $e'$ or $e$.  
  Since the  invariants of these two presentations 
  are equal, one sees that
  the values of $\dfrac{\Nr(M,T_H,\coh_H)}{\bk H}$ are the same for
  the two $H$-stabilizations.
  \renewcommand{\qedsymbol}{\fbox{\ref{T:T_adm}}}
\end{proof}
\renewcommand{\qedsymbol}{\fbox{\theteo}}

\begin{proof}[Proof of Proposition \ref{P:coh_adm}.]
  By Lemma \ref{L:val_coh}, there exists a surgery presentation $L\cup
  T$ of $(M, T,\coh)$ with an edge $e$ such that for each $x\in \X$
  there exists $n(x)\in \Z$ such that $xg_\coh(e)^{n(x)}\notin \X$.
  If $L_i$ is a component of $L$ such that $g_\coh(L_i)\in \X$ then
  sliding $e$ over $L_i$ $n(g_\coh(L_i))$-times the color of $L_i$ is
  changed to a color not in $\X$. Thus, by sliding $e$ when needed we
  obtain a computable presentation.
  \renewcommand{\qedsymbol}{\fbox{\ref{P:coh_adm}}}
\end{proof}
\renewcommand{\qedsymbol}{\fbox{\theteo}}

\begin{proof}[Proof of Theorem \ref{T:coh_adm}.]
  If $L$ and $L'$ are two computable presentations of $(M, T,\coh)$
  then there exists a sequence of handle-slides, blow-up moves and
  blow-down moves $L=L^0\stackrel{s_1}{\to},\ldots
  \stackrel{s_k}{\to}L^k=L'$ connecting them and inducing a
  diffeomorphism $f$ such that $f^*(\coh)=\coh$.  
 As in the proof of Theorem \ref{T:T_adm}, if each $L_i$ is computable then since each $s_i$ is an elementary move we have that Lemmas \ref{L:handle-slide}
  and  \ref{L:Stab} imply that the invariants associated to $L$ and $L'$ are equal.  However, it may happen that $L_i$ is not computable for some $i$.  Now if $(M, T,\coh)$ is $T$-admissible then Theorem \ref{T:T_adm} implies that  invariants of associated to $L$ and $L'$ are equal.  The general idea of the proof here is to create a $T$-admissible diagram from $L$, apply Theorem \ref{T:T_adm} then ``undo'' the $T$-admissible part we created.

  Pick a component $L_0^0$ of $L^0$ colored by $g_0$ and apply two
  blow-up moves to create meridians $m_+$ and $m_-$ with framing
  $\pm1$, respectively.  
  Here both meridians $m_{\pm}$ are colored by $g_0\in {\Gr}\setminus
  \dg{X}$ 
  and oriented as in Lemma \ref{L:Stab}.
  Since we did a positive and negative blow-up the framing of $L_0^0$
  is unchanged.  Without loss of generality one can assume that, if
  $L^0_0$ is destroyed by a blow-down move during the sequence, then
  this happens at the very last step $s_k$.  
  In the case when $s_k$ is a blow-down, we can replace $L'$ with $L^{k-1}$ which also gives a
  computable presentation and then Lemma \ref{L:Stab} together with the
  general case will imply the theorem.

Thus, we can assume $L^0_0$ is not destroyed by a blow-down move during the sequence.
Follow the sequence ``ignoring'' $m_+$ and $m_-$: when a component
$L_i$ slides over $L_0^0$ it gets linked (or unlinked) with $m_+$ and
$m_-$.  At the end of the sequence we have a link $L'\cup m_+ \cup
m_-$ where the components of $L'$ are colored by $g_\coh(L'_i)$ and
both $m_+$ and $m_-$ are colored by $g_0$.  The Kirby color
$\Omega_{g_0}$ of $m_\pm$ is a formal finite linear combination of 
colors: $\Omega_{g_0}=\sum_{\alpha}\qd(V_\alpha)V_\alpha$,
where $\alpha\in \A$ because $g_0\notin \X$.  Therefore, $F'(L\cup
T\cup m_+ \cup
m_-)=\sum_{\alpha,\beta}\qd(V_\alpha)\qd(V_\beta)F'(L\cup T\cup
m^{\alpha}_+\cup m^{\beta}_-),$ where $m^\alpha_+$ and $m^{\beta}_-$
are the meridians $m_+$ and $m_-$ colored by $\alpha$ and $\beta$,
respectively.  Let $T_{\alpha,\beta}=T\cup m^{\alpha}_+\cup
m^{\beta}_- \subset M$ and $\omega''$ be the compatible cohomology
class on $M\setminus T_{\alpha,\beta}$ induced by $\omega$ and $g_0$.
Since $L$ and $L'$ are both computable presentations of the
$T$-admissible triple $(M,T_{\alpha,\beta},\coh'')$ then
Theorem~\ref{T:T_adm} (proved above) implies $\frac{F'(L\cup T\cup
  m^{\alpha}_+\cup m^{\beta}_-)}{\Delta_+^r\Delta_-^s}=
\Nr(M,T_{\alpha,\beta},\coh'')= \frac{F'(L'\cup T\cup m^{\alpha}_+\cup
  m^{\beta}_-)}{\Delta_+^{r'}\Delta_-^{s'}}$.  Summing over $\alpha$
and $\beta$ with coefficients $\qd(\alpha)\qd(\beta)$ one gets
$\frac{F'(L\cup T\cup m_+ \cup
  m_-)}{\Delta_+^{r}\Delta_-^{s}}=\frac{F'(L' \cup T\cup m_+ \cup
  m_-)}{\Delta_+^{r'}\Delta_-^{s'}}$ where $ m_+$ and $ m_-$ are
colored by the Kirby color $\Omega_{g_0}$.  Now in the graph of
$L'\cup T\cup m_+ \cup m_-$, both of the meridians $m_+$ and $m_-$
bounds discs intersecting some components of $L'$.  So $m_+$ can be
linked with several strings of $L'$, then Lemma \ref{L:Stab} implies
the blow-down move to eliminate $m_+$ introduce a negative full twist
of these strings and a multiple $\Delta_+$ (remark that the lemma can
be applied as the overall color of the strands linked with $m_+$ is 
$g_0\notin \X$).  Similarly, eliminating $m_-$ introduce a positive full
twist of these strings and a multiple $\Delta_-$.  Thus,
$\frac{F'(T\cup L)}{\Delta_+^r\Delta_-^s}=\frac{F'(T\cup
  L')}{\Delta_+^{r'}\Delta_-^{s'}}$.
\renewcommand{\qedsymbol}{\fbox{\ref{T:coh_adm}}}
\end{proof}
\renewcommand{\qedsymbol}{\fbox{\theteo}}

\section{Proofs of the theorems of Section
  \ref{S:sl2NrInv}} \label{S:ProofsOfNrManInv}

\subsection{A quantization of $\slt$ and its associated ribbon
  category}\label{SS:QUantSL2H}
In this subsection we consider a ribbon category which underlies the
combinatorial invariants defined in Section \ref{S:sl2NrInv}.
Fix a positive integer $r$ and let $q=e^\frac{\pi\sqrt{-1}}{r}$ be a
$2r^{th}$-root of unity.  We use the notation
$q^x=e^{\frac{\pi\sqrt{-1} x}{r}}$.  Here we give a slightly
generalized version of quantum $\slt$.  Let $\UsltH$ be the
$\C(q)$-algebra given by generators $E, F, K, K^{-1}, H$ and
relations:
\begin{align*}
  HK&=KH, & HK^{-1}&=K^{-1}H, & [H,E]&=2E, & [H,F]&=-2F,\\
  KK^{-1}&=K^{-1}K=1, & KEK^{-1}&=q^2E, & KFK^{-1}&=q^{-2}F, &
  [E,F]&=\frac{K-K^{-1}}{q-q^{-1}}.
\end{align*}
The algebra $\UsltH$ is a Hopf algebra where the coproduct, counit and
antipode are defined by
\begin{align*}
  \Delta(E)&= 1\otimes E + E\otimes K, 
  &\varepsilon(E)&= 0, 
  &S(E)&=-EK^{-1}, 
  \\
  \Delta(F)&=K^{-1} \otimes F + F\otimes 1,  
  &\varepsilon(F)&=0,& S(F)&=-KF,
    \\
  \Delta(H)&=H\otimes 1 + 1 \otimes H, 
  & \varepsilon(H)&=0, 
  &S(H)&=-H,
  \\
  \Delta(K)&=K\otimes K
  &\varepsilon(K)&=1,
  & S(K)&=K^{-1},
    \\
  \Delta(K^{-1})&=K^{-1}\otimes K^{-1}
  &\varepsilon(K^{-1})&=1,
  & S(K^{-1})&=K.
\end{align*}
Define $\Ubar$ to be the Hopf algebra $\UsltH$ modulo the relations
$E^\ro=F^\ro=0$.

Let $V$ be a finite dimensional $\Ubar$-module.  An eigenvalue
$\lambda\in \C$ of the operator $H:V\to V$ is called a \emph{weight}
of $V$ and the associated eigenspace is called a \emph{weight space}.
We call $V$ a \emph{weight module} if $V$ splits as a direct sum of
weight spaces and $\qr^H=K$ as operators on $V$.  Let $\cat$ be the
tensor Ab-category of finite dimensional weight $\Ubar$-modules (here
the ground ring is $\C$).

We will now recall that the category $\cat$ is a ribbon Ab-category.
Recall the notation of Section \ref{S:sl2NrInv}.  Let $V$ and $W$ be
objects of $\cat$.  Let $\{v_i\}$ be a basis of $V$ and $\{v_i^*\}$ be
a dual basis of $V^*$.  Then
\begin{align*}
  b_{V} :& \C \rightarrow V\otimes V^{*}, \text{ given by } 1 \mapsto \sum
  v_i\otimes v_i^* & d_{V}: & V^*\otimes V\rightarrow \C, \text{ given by }
  f\otimes w \mapsto f(w)
\end{align*}
are duality morphisms of $\cat$.  
In \cite{Oh} Ohtsuki defines an $R$-matrix operator defined on
$V\otimes W$ by
\begin{equation}
  \label{eq:R}
  R=\qr^{H\otimes H/2} \sum_{n=0}^{\ro-1} \frac{\{1\}^{2n}}{\{n\}!}\qr^{n(n-1)/2}
  E^n\otimes F^n.
\end{equation}
where $q^{H\otimes H/2}$ is the operator given by  
$$q^{H\otimes H/2}(v\otimes v') =q^{\lambda \lambda'/2}v\otimes v'$$
for weight vectors $v$ and $v'$ of weights of $\lambda$ and $\lambda'$.  Thus,
the action of $R$ on the tensor product of two objects of $\cat$ is well
defined and induces an endomorphism on such a tensor product.  Moreover, $R$
gives rise to a braiding $c_{V,W}:V\otimes W \rightarrow W \otimes V$ on
$\cat$ defined by $v\otimes w \mapsto \tau(R(v\otimes w))$ where $\tau$ is the
permutation $x\otimes y\mapsto y\otimes x$ (see \cite{jM,Oh}).  Also, in
\cite{Oh} Ohtsuki defines an operator $\theta$ given by
\begin{equation}
\theta=K^{\ro-1}\sum_{n=0}^{\ro-1}
\frac{\{1\}^{2n}}{\{n\}!}\qr^{n(n-1)/2} S(F^n)\qr^{-H^2/2}E^n
\end{equation}
where $q^{-H/2}$ is an operator defined by on a weight vector $v_\lambda$ by
$q^{-H^2/2}.v_\lambda = q^{-\lambda^2/2}v_\lambda.$ The twist
$\theta_V:V\rightarrow V$ in $\cat$ is defined by $v\mapsto \theta^{-1}v$ (see
\cite{jM,Oh}).

Remark that the ribbon structure on $\cat$ induce right duality morphisms 
\begin{equation}\label{E:d'b'}
  d'_{V}=d_{V}c_{V,V^*}(\theta_V\otimes\Id_{V^*})\text{ and }b'_V
  =(\Id_{V^*}\otimes\theta_V)c_{V,V^*}b_V
\end{equation}
which are compatible with the left duality morphisms $\{b_V\}_V$ and
$\{d_V\}_V$.

\subsection{The invariant of oriented trivalent framed graphs $\Nr$
  through $\UsltH$}\label{SS:NfromSl2}
In this subsection show that the categories define in the previous
subsection gives rise to an invariant of ribbon graphs which recovers
the invariants of trivalent graphs defined in Subsection
\ref{S:axiom}.
We say a simple weight module is \emph{typical} if its highest weight
minus $(r-1)$ is in the set $(\C \setminus \Z) \cup \{kr : k\in
\Z\}=\srcol$, otherwise we say it is \emph{atypical}.  A typical
module is $r$ dimensional.  For $\alpha\in\srcol$, we denote
$V_\alpha$ by the simple weight module with highest weight
$\alpha+\ro-1$.

Let $F$ be the usual ribbon functor from $Rib_\cat$ to $\cat$.  Let
$\A$ be the set of typical modules.  Let $\qd:\A\rightarrow \C$ given
by $\qd(V_\alpha)=\qd(\alpha)$ where $\qd(\alpha)$ is defined in
Equation \eqref{E:Def_qd}.  In \cite{GPT}, it is shown that map
$F':\{\text{$\A$-graphs}\}\rightarrow \C$ given by Equation
\eqref{E:Def_of_F'} is a well defined invariant.  In particular,
$(\A,\qd)$ is an ambidextrous pair.

Next we will show that $F'$ can be used to define $\Nr$.  In
particular, we will recall how $F'$ extends to an invariant of
trivalent framed graphs whose edges are colored by element of $\srcol$
(for more details see \cite{GPT2}).  This extension requires the
choice of a certain family of morphisms in $\cat$.  Such a family is
given in \cite{GP} when $r$ is odd.  Here we will show that another
family can be deduced from the computation of \cite{CM} for any
$r\ge2$.

Let $U$ be the quantization of quantum $\slt$ considered in \cite{CM}.
The algebras $U$ and $\UsltH$ have the same underlying structure.
However, they differ in two main ways: 1) the element $K$ in $U$
should be considered the square root of the corresponding element $K$
in $\UsltH$, 2) $\UsltH$ has the additional generator $H$.  The first
difference essentially has no effect on the corresponding topological
invariants.  As explained in Subsection \ref{SS:QUantSL2H}, the
generator $H$ allows the category $\cat$ to be braided.

Let $U$-cat be the category of weight $U$-modules.  Consider the
functor $\cat \to U$-cat which is the identity at the level of vector
spaces and linear maps and sends a weight $\Ubar$-module to the weight
$U$-module determined by the action of the generators $K', E', F'$ of
$U$ given by $q^{H/2}, q^{-H/2}E, Fq^{H/2}$, respectively.  This
functor sends $V_\alpha$ to $V^a$ where $V^a$ is the highest weight
$U$-module of highest weight $a=(\alpha+r-1)/2$ considered in
\cite{CM}.  Let $\cap_{a,r-1-a}:V^a\otimes V^{r-1-a}\to\C$ be the map
defined in Equation (1.2) of \cite{CM}.  Let
$\cap_{\alpha,-\alpha}:V_\alpha\otimes V_{-\alpha}\to\C$ be the
corresponding morphism in $\cat$.  For $\alpha \in \srcol$ define the
isomorphism $w_\alpha: V_\alpha \to V^*_{-\alpha}$ by
$$w_\alpha=\big(\cap_{\alpha,-\alpha}\otimes\Id_{V^{*}_{-\alpha}}\big)
\circ\big(\Id_{V_\alpha}\otimes b_{V_{-\alpha}}\big).$$
\begin{prop}
  The family $\{w_\alpha\}_{\alpha \in \srcol}$ 
 satisfy 
\begin{equation}
  \label{eq:turnw}
  d_{V_\alpha}\circ\big(w_{-\alpha}\otimes\Id_{V_\alpha}\big)=
  d'_{V_{-\alpha}}\circ\big(\Id_{V_{-\alpha}}\otimes w_{\alpha}\big)
  \in\Hom_\cat(V_{-\alpha}\otimes V_{\alpha},\C)
\end{equation}
for all $\alpha \in \srcol$.
\end{prop}
\begin{proof}
  Using the naturality of the braiding and the formula for
  $d'_{V_{-\alpha}}$ given in Equation \eqref{E:d'b'} we can rewrite
  Equation \eqref{eq:turnw} as
  $$\cap_{-\alpha,\alpha}=\cap_{\alpha,-\alpha}c_{V_{-\alpha},V_{\alpha}}
  (\theta_{V_{-\alpha}}\otimes \Id_{V_\alpha}).$$
 This equation is equivalent to the following equation in $U$-mod:
 $$\cap_{b,a}=q^{2ab}\cap_{a,b}\circ ({_{b}^aR})$$
 where $b=r-1-a$ and ${_{b}^aR}$ is the map defined in Equation (1.4)
 of \cite{CM}.  The last equality is an easy consequence of
 \cite[Proposition 3.1]{CM}.
\end{proof}
Consider the map $Y_c^{a,b} : V^c \to V^a\otimes V^b$ define in
Theorem 1.7 of \cite{CM} where $a,b,c\in \C\setminus \frac12\Z$ and
$a+b-c\in \{0,1,...,r-1\}$.  Using the functor above the map
$Y_c^{a,b}$ corresponds to a non-zero morphism
$Y_{-\gamma}^{\alpha,\beta}:V_{-\gamma}\to V_{\alpha}\otimes
V_{\beta}$ in $\cat$ where $ \alpha,\beta,\gamma\in \srcol$ with
$\alpha+\beta+\gamma\in\Hr$.  Define $W^{\alpha,\beta,\gamma}:\C\to
V_\alpha\otimes V_\beta\otimes V_\gamma$ as the morphism
$W^{\alpha,\beta,\gamma}=(Y_{-\gamma}^{\alpha,\beta}\otimes
w_{\gamma}^{-1})b_{V_{-\gamma}}$.  Now Lemma 1.8 of \cite{CM} implies
that
\begin{equation}
  \label{eq:turnW2}
  \big(d_{V_\alpha}\otimes\Id_{V_\beta\otimes V_\gamma\otimes V_\alpha}\big) \circ 
  \big(\Id_{V_\alpha^*}\otimes W^{\alpha,\beta,\gamma}\otimes 
  \Id_{V_\alpha}\big)\circ b'_{V_\alpha}=W^{\beta, \gamma,\alpha}.
\end{equation}

Since the family $\{V_\delta, w_\delta\}_{\delta\in \srcol}$ satisfies
Equation \eqref{eq:turnw} then in the terminology of \cite{GPT2} it is
called \emph{basic data}.  Moreover, with this basic data the pair
$(\srcol, \qd)$ is \emph{trivalent-ambidextrous} (see Definition 1 of
\cite{GPT2}).  Thus, as explained after Lemma 2 in \cite{GPT2} the
invariant $F'$ extends to an invariant of oriented trivalent framed
graphs whose edges are colored by elements of the set $\srcol$.  This
extension can be summarized as follows.

Let $\Gamma$ be an oriented trivalent framed graph in $S^{3}$ and
$\Sigma$ be a thickening of $\Gamma$ to a surface using the framing.
We construct a framed $\cat$-colored ribbon graph with coupons
$\Gamma'$ embedded in $\Sigma$.  First we decompose $\Sigma$ in bands,
discs and annuli corresponding respectively to the edges, 3-valent
vertices and loops of $\Gamma$.  Then we replace a band colored by
$\alpha$ by two edges colored by $V_\alpha$ and $V_{-\alpha}$ and a
bivalent coupon filled with $w_\alpha$ as shown in Figure
\ref{F:triv}a.  Then we put in each disk a coupon filled with one of
the morphisms $W^{\alpha,\beta,\gamma}$ as shown in Figure
\ref{F:triv}b.  The color of an $\alpha$-colored loop is replaced by
$V_\alpha$.  The construction of $\Gamma'$ involve some choice but the
quantity $F'(\Gamma')$ does not depend of these choices.  Define
$\Nr(\Gamma)=F'(\Gamma')$.
\begin{figure}
 \begin{minipage}[b]{.4\linewidth}
  \centering
    $\epsh{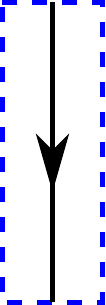}{12ex} \put(-18,1){$\alpha$} \longrightarrow
    \epsh{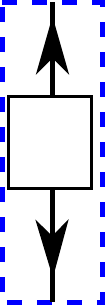}{12ex}\put(-6,-14){\ms{V_\alpha}}\put(-6,16){\ms{V_{-\alpha}}}
    \put(-15,2){${w_\alpha}$}
    $
  \\[1ex]{a. Replacing an edge}
 \end{minipage} 
 \begin{minipage}[b]{.4\linewidth}
  \centering
  $\epsh{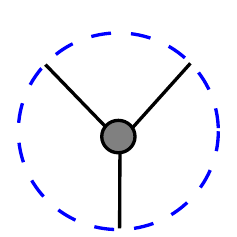}{12ex} \longrightarrow \epsh{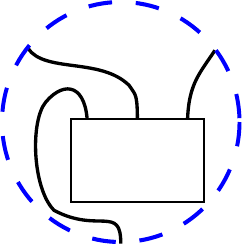}{12ex}
  \put(-34,-8){\ms{W^{\alpha,\beta,\gamma}}}$
  \\[1ex]{b. Replacing a vertex}
 \end{minipage}
 \caption{Construction of a $\cat$-colored ribbon graph with a
   trivalent graph.\label{F:triv}}
\end{figure}

By definition the extension of $F'$ can be computed using the formulas
of \cite{CM} via the functor $\cat\to U$-mod discussed above.  In
particular, the relations of Section \ref{S:axiom} are then
consequences of computations given in \cite{CM} and of properties of
$F'$.  Let use be more precise.  Suppose $T$ and $T'$ are two 1-1
$\cat$-colored ribbon graphs whose colors belongs to $\A$ such that
$F(T)=F(T')$ then Proposition \ref{P:FvsF'} implies
$F'(\tr(T))=F'(\tr(T'))$ where $\tr$ is the braid closure of the
tangle.  An analogous relation between $F$ and $\Nr$ exists.  We now
explain why the axioms of Section \ref{S:axiom} hold.

The invariant $F$ satisfies a property analogous to Axiom
\eqref{eq:ch-orient} and so this axiom is a direct consequence of this
fact.
The endomorphism set $\Hom_\cat(\C,V_\alpha\otimes V_\beta\otimes
V_\gamma)$ is zero if $\alpha+\beta+\gamma\notin\Hr$ and so Axiom
\eqref{eq:heights} follows.

The endomorphism represented by $T$ in Axiom \eqref{eq:con-sum} is in
$\Hom_\cat(V_\alpha, V_\beta)$ which is equal to $\{0\}$ if
$\alpha\neq\beta$ and $\C\cdot\Id_{V_\alpha}$ if $\alpha=\beta$.  In
the latter case, $F(T)$ is a scalar times $\Id_{V_\alpha}$ and this
scalar is by definition $\qd(\alpha)^{-1}$ times $\Nr(\hat T)$ where
$\hat T$ is the closure of $T$ which appears in the right hand side of
the axiom.  Replacing $T$ by this scalar times a strand representing
the identity of $V_\alpha$, we get the equality in \eqref{eq:con-sum}.

The first identity of \eqref{eq:unknot-Theta} follows directly from
the definition of $F'$.  Lemma 1.10 of \cite{CM} implies
$\Nr\left(\epsh{fig13}{4ex}\right)=1$ and thus proves the second
equalities in Axiom \eqref{eq:unknot-Theta}.  The last equality
follows from the fact that $\Hom_\cat(V_\alpha,\C)=\{0\}$.

The invariant $F$ vanishes on any closed graph $T$ colored by modules
in the set $\{V_\alpha\}_{\alpha\in\srcol}$.  This implies that $F'$
and $\Nr$ vanish on any split graph $T\sqcup T'$ where both $T$ and
$T'$ are colored by the modules in the set
$\{V_\alpha\}_{\alpha\in\srcol}$ and thus Property
\eqref{I:zeroSplitGraphs} in the list of axioms holds.

The space $\Hom_\cat(\C,V_\alpha\otimes V_\beta\otimes V_\gamma)$ is
one dimensional and generated by $W^{\alpha,\beta,\gamma}$ (assuming
that $\alpha+\beta+\gamma\in\Hr$).  Hence the morphisms represented by
$T$ and $T'$ in \eqref{eq:con-sumv} are equal to
$\lambda.W^{\alpha,\beta,\gamma}$ and
$\lambda'.W^{-\alpha,-\beta,-\gamma}$ respectively.  The scalars
$\lambda$ and $\lambda'$ are the factors of the right hand side
because $\Nr$ is $1$ one the $\Theta$ graphs.

Equation (3.1) of \cite{CM} implies Axioms \eqref{eq:twist} and
\eqref{eq:twistv}.  A computation of $F'$ of the Hopf link can be
found in \cite{GPT}.  Equivalently, Axiom \eqref{eq:hopf} can be
deduced from the other axioms.

Axiom\eqref{eq:fusion} follows from the decomposition of modules
$V_\alpha\otimes V_\beta\simeq\oplus_{k\in\Hr}V_{\alpha+\beta+k}$.
Here we use that if $\gamma=\alpha+\beta+k$, then
$\Hom_\cat(V_\gamma,V_\alpha\otimes V_\beta)$ is generated by the
element corresponding to $W^{\alpha,\beta,-\gamma}$ through the
isomorphism $\Hom_\cat(V_\gamma,V_\alpha\otimes V_\beta)\cong
\Hom_\cat(\C,V_{\alpha}\otimes V_{\beta}\otimes V_{-\gamma})$ and
similarly for $\Hom_\cat(V_\alpha\otimes V_\beta,V_\gamma)\cong
\Hom_\cat(\C,V_{-\beta}\otimes V_{-\alpha}\otimes V_\gamma)$.
Finally, the $6j$-symbols computed in \cite{CM} are by definition the
coefficients that appear in \eqref{eq:6j}.

\subsection{The relative $\Gr$-modular structure on $\cat$}
Here we show that the categories considered earlier in this section
are relative $\Gr$-modular categories.
Let $\Gr$ be the additive group $\C/2\Z$,
$\X=\mathbb{Z}/2\mathbb{Z}\subset \Gr$ and $\TT=\Z$ (here we use
additive notation).  We will now show that $\cat$ is a relative
$\Gr$-modular category relative to $\X$ with modified dimension $\qd$
and periodicity group $\TT$.  To do this we will show that Conditions
(1)--\eqref{I8:DefGmodularCat} of Definition \ref{D:GmodularCat} hold.

For $g\in\Gr$, define $\cat_{g}$ as the full sub-category of weight
modules with weights congruent to $g$ mod $2$.  Then it is easy to see
that $\{\cat_g\}_{g\in \Gr}$ is a $\Gr$-grading in $\cat$.  Moreover,
$\X^{-1}=\X$ and $\Gr$ can not be covered by a finite number of
translated copies of $\X$.  Thus, Conditions \eqref{I1:DefGmodularCat}
and \eqref{I4:DefGmodularCat} are satisfied.
    
Recall that $\A$ is the set of typical modules and $\qd:\A\rightarrow
\C$ is the function given by $\qd(V_\alpha)=\qd(\alpha)$ where
$\qd(\alpha)$ is defined in Equation \eqref{E:Def_qd}.  Also, in
\cite{GPT} it is shown that $(\A,\qd)$ is an ambidextrous pair.  Thus,
if $g\in \Gr\setminus \X$ then by definition the simple modules of
$\cat_g$ are all typical and Condition \eqref{I5:DefGmodularCat}
holds.  Moreover, it follows that the category $\cat_g$ is semi-simple
if $g\in \Gr\setminus \X$ (see Lemma \ref{L:Dg_semi-simple} where we
prove a general statement).

For $t\in\Z$, let $\ve^{t}$ be the one dimensional vector space $\C$ endowed
with the $\Ubar$-action determined by
$$Ev=Fv=0,\quad Kv=v,\quad Hv=2\ro t v$$
for any $v\in \ve^{t}$.  Then $\ve^{t}$ is a weight module in
$\cat_{0}$.  Since, the action of $E$ and $F$ on $\ve^t$ is zero, it
is easy to see that $\{\ve^{t}\}_{t\in \Z}$ is commutative set of
objects in $\cat$ and $\ve^{t}\otimes\ve^{t'}=\ve^{t+t'}$.  Moreover,
\begin{equation}\label{E:ve_otimesV}
  \ve^{t}\otimes V_\alpha=V_\alpha\otimes\ve^{t}=V_{\alpha+2\ro t}
\end{equation}
and it follows that $\{\ve^{t}\}_{t\in\TT}$ is a free realization of
$\Z$ in $\cat$, i.e. Condition \eqref{I2:DefGmodularCat} holds.  For
$g\in \Gr$, the simple modules of $\cat_g$ are all the typical modules
$V_\alpha$ such that $\alpha+r-1 \equiv g \mod 2$.  Equation
\eqref{E:ve_otimesV} implies this set of typical modules is the union
the simple $\Z$-orbits $\wt{V}_{\alpha +i}$ where $i$ runs over the
set $\{0,1,...,2\ro-1\}$ and $\alpha$ is a complex number such that
$\alpha \equiv g \mod 2$.  Therefore, Condition
\eqref{I6:DefGmodularCat} holds.

Let $\Gr\times \Z\rightarrow \C^*$ be the map given by $(g,t)\mapsto
q^{2rt\alpha}$ where $\alpha$ is any complex number such that $\alpha
+r-1 \equiv g \mod 2$ Note this this map is well define since $q$ is a
$2r$th root of unity.  If $V$ is a weight module then
\begin{equation}
  \label{eq:ve}
  c_{V,\ve^{t}}=\tau\circ({K^{\ro t}\otimes\Id})
  \quad\text{ and }\quad
  c_{\ve^{t},V}=\tau\circ({\Id\otimes K^{\ro t}})
\end{equation}
where $\tau$ is the flip map $x\otimes y\mapsto y\otimes x$.  Thus, Condition
\eqref{I3:DefGmodularCat} holds.  

Finally, the computations given at the end of Subsection \ref{S:axiom}
shows that Condition \eqref{I7:DefGmodularCat} holds.  A similar
computation is given in \cite{GPT} to show that $F(H(V,W))\neq 0$
where $H(V,W)$ is the long Hopf link whose long edge is colored by an
object $V\in \A$ and whose circle component is colored an object $W\in
\A$.  Thus, Condition \eqref{I8:DefGmodularCat} holds.

\subsection{The 3--manifold invariant $\Nr$}\label{S:NrMsl2}
In this subsection we prove Theorems \ref{T:sl2CompSurgInv},
\ref{T:T_adm-Nr} and \ref{T:N0sl2} and Proposition
\ref{P:ExistCompSurgSl2}.
In Section \ref{S:sl2NrInv} we only considered compatible triples
$(M,T,\coh)$ where $T$ was a framed trivalent graph in a 3--manifold
$M$ whose edges were colored by elements of $\srcol$.  We only
consider such triples $(M,T,\coh)$ in this section.  If $T\neq
\emptyset$ then by definition $(M,T,\coh)$ is $T$-admissible.  Also,
if $(M,T,\coh)$ is computable as defined in Section \ref{S:sl2NrInv}
then $(M,T,\coh)$ is computable as defined in Section
\ref{SS:MainResults}.

In the preceding section we showed that $\cat$ was a relative $\Gr$-modular
category.  Therefore, the general theory of Subsection
\ref{SS:MainResults} imply the existence of 3--manifold invariants
$\Nr$ and $\Nr^0$.  The main results of Section \ref{S:sl2NrInv}
follow from this general theory and the fact that $\Nr$ is defined as
an extension of $F'$.  In particular,
\begin{align*}
  \text{Theorem } \,
  & \text{\ref{T:coh_adm}  implies Theorem \ref{T:sl2CompSurgInv} } \\
  \text{Theorem } \, & \text{\ref{T:T_adm} implies Theorem  \ref{T:T_adm-Nr} }\\
   \text{Theorem } \, & \text{\ref{T:N0} implies Theorem  \ref{T:N0sl2} }
\end{align*}
Note here that if $H=H(\alpha,\beta)$ is the long Hopf-link in Theorem
\ref{T:T_adm-Nr} then
$\brk{H}=F'(H)/\qd(\beta)=(-1)^{\ro-1}rq^{\alpha\beta}/\qd(\beta)$.
\begin{proof}[Proof of Proposition \ref{P:ExistCompSurgSl2}.] 
  Let $(M,T,\coh)$ be a compatible triple and $L$ be a link which
  gives rise to a surgery presentation of $M$.  Then the image of
  $\coh\in\Hom(H_1(M\setminus T,\Z),\C/2\Z)$ is generated by the
  values of $\coh$ on the meridian of $L\cup T$.  As $\coh$ is not
  integral, its image is not contained in $\Z/2\Z$, and $L\cup T$ has
  an edge $e$ with $g_\coh(e) \in \C/2\Z\setminus \Z/2\Z$.

  Suppose $L_i$ is a component of $L$ such that $g_\coh(L_i)\in
  \X=\Z/2\Z$.  If we slide the component $e$ over $L_i$, after this
  sliding the $\Gr$-color of $L_i$ is $g_\coh(e)g_\coh(L_i)$ which is
  not in $\X$ (this is imposed by the fact that the cohomology class
  must be unchanged after the sliding).  Thus, by sliding $e$ when
  needed we obtain a computable presentation.
  \renewcommand{\qedsymbol}{\fbox{\ref{P:ExistCompSurgSl2}}}
\end{proof}
\renewcommand{\qedsymbol}{\fbox{\theteo}}

\begin{rem}
  $\Nr^0(M,T,\coh)$ can be extended by allowing the tangle part $T$ to
  contain any $\cat$-colored link.  In particular, let $L^J$ denote the
  coloring of a link by the two dimensional representation with highest weight
  $1$.  Then $\Nr^0(S^3,L^J,0)$ is just the Kauffman bracket of $L$ evaluated
  at $A^2=q$ and for a 3--manifold, $\Nr^0(M,L^J,0)$ is a generalization of it
  for links in 3--manifolds. 
\end{rem}

\subsection{Proof of Proposition  \ref{prop:encircling}}
\begin{proof}[Proof of Proposition \ref{prop:encircling}.]
  \newcommand{\Ma}{f_a}
  Let $W=V_\alpha\otimes V_\beta\otimes V_\gamma$ and
  $\Ma=F\left(\epsh{fig31}{6ex}\put(-19,-3){\ms{\Omega_{a}}}
  \right)\in\End_{U_q(sl_2)}(W)$.  We will show the image of $\Ma$ is
  the trivial module.  By Lemma
  \ref{L:handle-slide} we can do handle-slides, blow-up moves and blow-down
  moves on the circle component of the graph representing $\Ma\otimes
  \Id_{V_0}$ to obtain via Proposition \ref{P:FvsF'} the equality of
  morphisms:
  \begin{equation}
    \label{eq:transparent}
    c_{W,V_0}\circ (\Ma\otimes Id_{V_0})=c_{V_0,W}^{-1}\circ (\Ma\otimes Id_{V_0})
  \end{equation}
  where the braidings $c_{W,V_0}, c_{V_0,W}^{-1}: W\otimes V_0\to
  V_0\otimes W$ are given by
  $$c_{W,V_0}=\tau\circ R =  \tau\circ q^{H\otimes H/2}(\Id \otimes \Id + 
  (q-q^{-1})E\otimes F + \cdots )$$ and
  $$c_{V_0,W}^{-1}=R^{-1}\circ \tau = (\Id \otimes \Id + (-q+q^{-1})E\otimes F 
  + \cdots )q^{-H\otimes H/2}\circ\tau. $$ 
  Here the dots ``$\cdots$'' are linear combination of power
  $(E\otimes F)^k$ with $k\ge2$.
  \\
  Let $x\in V_\alpha\otimes V_\beta\otimes V_\gamma$ and set
  $y=\Ma(x)$.  The module $V_0$ has its weights in $\Hr$ and, as $r$
  is odd, $V_0$ has a non zero weight vector $v_0$ of of weight $0$.
  The vectors $\{E^k.v_0, F^k.v_0:k=1\cdots\frac{r-1}2\}$ form with
  $v_0$ a basis of $V_0$.  Let $V'$ be the vector space generated by
  $\{E^k.v_0, F^k.v_0:k=2\cdots\frac{r-1}2\}$.  Applying
  \eqref{eq:transparent} to $x\otimes v_0$ we have
  $$c_{W,V_0}(y\otimes v_0)=c_{V_0,W}^{-1}(y\otimes v_0)$$
  with
  \begin{align}
    c_{W,V_0}(y\otimes v_0) & =  \tau\circ q^{H\otimes H/2}(y\otimes v_0 
    + (q-q^{-1})Ey\otimes Fv_0 +Y_1')\\
    & = v_0\otimes y + (q-q^{-1})Fv_0\otimes K^{-1}Ey+Y_2'
  \end{align}
  where $Y_1'\in W\otimes V'$ and $Y'_2\in V'\otimes W$. The last
  equality comes from the facts that $Hv_0=0$ and $HFv_0=-2Fv_0$.
  Similarly,
  \begin{align}\label{E:cMa1}
    c_{V_0,W}^{-1}(y\otimes v_0)&= v_0\otimes y-(q-q^{-1})Ev_0\otimes Fy + Y'_3
  \end{align}
  where $Y'_3\in V'\otimes W$.  Setting the above equations equal we
  have $K^{-1}Ey=Fy=0$.  So $Ey=0$ and also,
  $(q-q^{-1})(EF-FE)y=(K-K^{-1})y=0$.  Thus, $K^2y=y$, $Ey=0$ and
  $Fy=0$ and this holds for any choice of $a$ and of $x$; but, since
  $K$ acts as $q^H$ and the weights of $W$ are in $2\Z$, we have that
  the eigenvalues of $K$ are in $q^{2\Z}\not\ni-1$.  Thus $Ky=y$ and
  $\Ma(x)$ is an invariant vector of $W$.

  Then as $\Hom_{U_q(sl_2)}(\C,V_\alpha\otimes V_\beta\otimes
  V_\gamma)\simeq\Hom_{U_q(sl_2)}(V_\alpha\otimes V_\beta\otimes
  V_\gamma,\C)\simeq\C$, 
    there exists $\lambda\in\C$ such that   
  \begin{equation}\label{eq:encircling2}
    F\left(\epsh{fig31}{8ex}\put(-25,-4){\ms{\Omega_{a}}}
      \put(-17,-18){\ms{\alpha}}\put(-13,-18){\ms{\beta}}
      \put(-9,-18){\ms{\gamma}}\right)=\lambda F\left(\epsh{fig32}{6ex}\right).
  \end{equation}
  To compute $\lambda$, we consider the value by $\Nr$ of the braid closure of
  the graphs in this equality.  The braid closure of the right side is
  $\lambda$ times the $\Theta$-graph on which $\Nr$ has value $1$.  Thus
  $\lambda$ is equal to $\Nr$ of the braid closure of the graph in the left
  hand side.  But this graph is a connected sum of 3 Hopf links and its value
  is thus given by
  $$\lambda=\sum_{k\in\Hr}
  \qd(a+k)^{-1}r^3q^{(a+k)\alpha}q^{(a+k)\beta}q^{(a+k)\gamma}=
  \sum_{k\in\Hr} \frac{r^3}{\qd(a+k)}.$$ 
  But for any $b\in\srcol$, we have
  $\qd(b)^{-1}=r^{-1}\sum_{l\in\Hr}q^{lb}$ thus
  $$\lambda=r^2\sum_{k\in\Hr}\sum_{l\in\Hr}q^{l(a+k)}=
  r^2\sum_{l\in\Hr}q^{la}\sum_{k\in\Hr}q^{lk}.$$
  Now if $l\in \Hr\setminus\{0\}$,  then $\sum_{k\in\Hr}q^{lk}=0$ thus only the
  term for $l=0\in\Hr$ contributes and $\lambda=r^3$.  
  Finally because of Proposition \ref{P:FvsF'}, Equation
  \eqref{eq:encircling2} also holds for any closure and $F$ replaced with
  $\Nr$.  
\renewcommand{\qedsymbol}{\fbox{\ref{prop:encircling}}}
\end{proof}
\renewcommand{\qedsymbol}{\fbox{\theteo}}

\section{The other quantum groups}\label{S:otherQG}
In this section we recall the results of \cite{GP2} and show that they imply
the existence of relative $\Gr$-modular categories associated with the quantum
group of any simple Lie algebra. 

Let $\g$ be a simple finite-dimensional complex Lie algebra of rank $\rk$ and
dimension $2N+\rk$ with a root system.  Fix a set of simple roots
$\{\alpha_1,\ldots,\alpha_\rk\}$ and let $\roots^{+}$ be the corresponding set
of positive roots.  Also, let $A=(a_{ij})_{1\leq i,j\leq \rk}$ be the Cartan
matrix corresponding to these simple roots.  There exists a diagonal matrix
$D=\text{diag}(d_1,\ldots,d_\rk)$ such that $DA$ is symmetric and positive
definite and $\min\{d_i\}=1$. 
Let $\h$ be the Cartan subalgebra of $\g$ generated by the vectors
$H_1,\ldots,H_\rk$ where $H_j$ is determined by $\alpha_i(H_j)=a_{ji}$.  Let
$L_R$ be the root lattice which is the $\Z$-lattice generated by the simple
roots $\{\alpha_i\}$.  Let $\brk{\;,\;}$ be the form on $L_R$ given by
$\brk{\alpha_i,\alpha_j}=d_ia_{ij}$.  Let $L_W$ be the weight lattice which is
the $\Z$-lattice generated by the elements of $\h^*$ which are dual to the
elements $H_i$, $i=1\cdots\rk$.  Let
$\rho=\frac12\sum_{\alpha\in\roots^{+}}\alpha\in L_W$.

Let $\ro$ be an odd integer such that $\ro\geq 3$ and $\ro\notin3\Z$
if $\g=G_2$.
Let $q=\e^{2i\pi/\ro}$ and for $i=1,\ldots,\rk$, let $q_i=q^{d_i}$.
For $x\in \C$ and $k,l\in \N$ we use the notation:
$$
q^x=\e^{\frac{2i\pi x}\ro},\quad \qn x_q=q^x-q^{-x},\quad \qN x_q=\frac{\qn
  x_q}{\qn 1_q},\quad\qN k_q!=\qN1_q\qN2_q\cdots\qN k_q,\quad{k\brack
  l}_q=\frac{\qN{k}_q!}{\qN{l}_q!}.
$$    
Remark for $x\in\C$, $\qn x=0$ if and only if $ x\in\frac\ro{2}\Z$.

The \emph{unrolled quantum group} $\UqgH$ is the algebra generated by
$K^\beta, X_i,X_{-i},H_i$ for $\beta\in \La, \,i=1,\ldots,\rk$ with
relations
\begin{eqnarray}\label{eq:rel1}
  & K^0=1, \quad K^\beta K^\gamma=K^{\beta+\gamma},\, \quad K^\beta X_{\sigma
    i}K^{-\beta}=q^{\sigma \brk{\beta,\alpha_i}}X_{\sigma i}, & \\
  \label{eq:rel2} &[X_i,X_{-j}]=
  \delta_{ij}\frac{K^{\alpha_i}-K^{-\alpha_i}}{q_i-q_i^{-1}}, &  \\
  \label{eq:rel3} & \sum_{k=0}^{1-a_{ij}}(-1)^{k}{{1-a_{ij}} \brack
    k}_{q_i} X_{\sigma i}^{k} X_{\sigma j} X_{\sigma i}^{1-a_{ij}-k} =0,
  \text{ if }i\neq j & 
\end{eqnarray}
\begin{equation}
  \label{eq:relH}
  [H_i,X_{\epsilon j}]=\sigma a_{ij}X_{\sigma j},\quad[H_i,H_j]=[H_i,K^\beta]=0
\end{equation}
where $\sigma=\pm 1$.

The algebra $\UqgH$ is a Hopf algebra with coproduct $\Delta$, counit
$\epsilon$ and antipode $S$ defined by
\begin{align*}
  \Delta(X_i)&= 1\otimes X_i + X_i\otimes K^{\alpha_i}, &
  \Delta(X_{-i})&=K^{-\alpha_i} \otimes
  X_{-i} + X_{-i}\otimes 1,\\
  \Delta(K^\beta)&=K^\beta\otimes K^\beta, & \epsilon(X_i)&=
  \epsilon(X_{-i})=0, &
  \epsilon(K^{\alpha_i})&=1,\\
  S(X_i)&=-X_iK^{-\alpha_i}, & S(X_{-i})&=-K^{\alpha_i}X_{-i}, &
  S(K^\beta)&=K^{-\beta},\\
  \Delta(H_i)&= 1\otimes H_i + H_i\otimes 1,& \epsilon(H_i)&=0,&
  S(H_i)&=-H_i.
\end{align*}
$\UqgH$ has an Hopf ideal $I$ which contains the
$\ro$\textsuperscript{th} powers of the roots vectors (see
\cite{GP2}).

Also in \cite{GP2}, a full subcategory $\D^\theta$ of the category of
finite dimensional representations of $\UqgH$ is shown to be ribbon.
Let us describe briefly its modules.

A weight vector of weight $\lambda\in\h^{*}$ in a $\UqgH$-module is a
vector on which $H_i$ acts by $\lambda(H_i)$.  A weight vector $v$ is
an highest weight vector if $E_i.v=0,\,\forall i=1\cdots\rk$.  A
weight module is a $\UqgH$-module which satisfy:
\begin{enumerate}
\item it is finite dimensional over $\C$,
\item it has a base of weight vectors,
\item the elements $K^{\sum_i\lambda_i\alpha_i}$ act on it as
  $q^{\sum_i\lambda_i H_i}$,
\item elements of $I$ vanish on it.  
\end{enumerate}
Any weight module $V$ has an highest weight vector and it is unique
(up to a scalar) if $V$ is irreducible.  Moreover the set of
isomorphic classes of irreducible weight modules is in bijection with
$\h^{*}$.  We will write $V_\lambda$ for an irreducible module with
highest weight $\lambda+(\ro-1)\rho$.

$\D^\theta$ is a full sub-tensor category of the category $\D$ of weight
modules (conjecturally, $\D^\theta=\D$).  The category $\D$ (and also
$\D^\theta$) is $\Gr$-graded where $\Gr=\h^{*}/L_R\cong(\C^*)^\rk$.  The weights
of a module in $\D_g$ are all the same modulo $L_R$.  For any $\lambda\in
L_W$, $K^{\ro\lambda}$ acts as the same scalar denoted $g(K^{\ro\lambda})$ on
any module of $\D_{g}$.

Let $\A$ be the family of irreducible weight modules with highest weight
$\lambda$ such that $q^{2\brk{\lambda+\rho,\beta}+m\brk{\beta,\beta}}\neq1$
for all $\beta\in\roots^{+} $ and $m\in \{0,\ldots\ro-1\}$.  Modules in $\A$
are called typical, they are the $\ro^{N}$-dimensional simple modules, they
all belongs to $\D^\theta$ and their categorical dimension vanishes.
\begin{lemma}\label{L:Dg_semi-simple}
  Typical modules are the simple projective modules of $\D$.  Hence the
  subcategory $\D_g$ is semi-simple iff all its simple modules are typical.
\end{lemma}
\begin{proof}
  Every simple module $V_\lambda\in\D_g$ is a quotient of a
  $\ro^{N}$-dimensional module $\wb M(\lambda)$ (which are the ``Verma
  module'' for $\D$ see \cite[Section 3.1]{DK}) generated by an highest weight
  vector of weight $\lambda$.  In particular, if $V_\lambda$ is not typical,
  then $\wb M(\lambda)$ is not semi-simple and the category $\cat_g\ni
  V_\lambda$ is not semi-simple.
  \\
  Let $N^+$ (resp $N^-$) be the subalgebra of $\UqgH$ generated by the
  elements $X_i$ (resp $X_{-i}$) for $i=1\cdots\rk$.  Then $N^+/(I\cap N^+)$
  (resp $N^-/(I\cap N^-)$) possess an unique highest weight vector $x_+$ (resp
  an unique lowest weight vector $x_-$).  As $\dim_\C(N^+/(I\cap
  N^+))=\dim_\C(N^-/(I\cap N^-))=\ro^N$, we have that for a typical module
  $V_\lambda$ with highest weight vector $v_+$ and lowest weight vector $v_-$,
  $x_-.v_+\in\C^*v_-$ and $x_+.v_-\in\C^*v_+$.  Now if $W\to V_\lambda$ is an
  epimorphism, it is surjective and $v_+$ has a preimage $w$.  Let
  $w_+=x_+x_-.w$.  Then $w_+$ is an highest weight vector of $W$ (by
  maximality of $x_+$) which is sent to a non zero multiple of $v_+$.  The
  usual property of Verma modules apply to $\wb M(\lambda)$ and there is a
  unique map $V_\lambda=\wb M(\lambda)\to V$ which sends the highest weight
  vector of $V_\lambda$ to $w_+$.  This map gives a section of the epimorphism
  and thus the typical module $V_\lambda$ is projective.  When all simple
  modules of $\D_g$ are typical, by an easy induction (considering an
  irreducible quotient) we get that any finite dimensional module of $\D_g$ is
  completely reducible.
\end{proof}
In particular this is true if $g\notin \X$ where $\X\subset\h^{*}/L_R$ is
formed by the weights $\lambda$ such that $\exists \beta\in\roots^+$ with
$2\brk{\lambda,\beta}\in\Z$.  Remark that for these $g$, $\D_g\subset\D^\theta$
because $\D^\theta$ contains all typical modules. 

If $V_\lambda\in\A$ has highest weight $\lambda+(\ro-1)\rho$, let 
$$\qd(V_\lambda)=\prod_{\alpha\in\roots^+}\frac{
  \qn{\brk{\lambda,\alpha}}}{ \qn{\ro\brk{\lambda,\alpha}}}$$ then
$(\A,\qd)$ is an ambi pair.  The ingredient in \cite{GP2} to compute
$\qd$ is the computation of the image by $F'$ of the Hopf link $H$
colored by $V_\lambda$,$V_\mu$ which is
$F'(H)=q^{2\brk{\lambda,\mu}}$.

For $t\in\TT=\ro L_R\cong(\ro\Z)^\rk$, let $\ve^t$ be the vector space
$\C$ endowed with the action of $\UqgH$ given by $X_{\pm i}=0$, and by
being a weight space of weight $t$.  Then $(\ve^t)_{t\in\TT}$ is a
free realization of $\TT$ in $\D^\theta_1$ and if $V\in\D^\theta_g$,
the square of braiding on $V\otimes \ve^t$ is given by $g\pow
t=q^{2\brk{t,\lambda}}$ for any weight $\lambda\in V$ (two weights in
$V$ differ by an element of $L_R$ and $\brk{\TT,L_R}\subset\ro\Z$).

Finally, the twist on $V_\lambda$ is given by the scalar
$q^{\brk{\lambda,\lambda}-(\ro-1)^2\brk{\rho,\rho}}$ so we have after
a computation similar to that of Section \ref{S:axiom}:
$$\Delta_+=q^{-2(\ro-1)^2\brk{\rho,\rho}} 
\sum_{k\in L_R/(\ro
  L_R)}q^{\brk{k,k+2\rho}}=q^{-3(\ro-1)^2\brk{\rho,\rho}} 
\sum_{k\in L_R/(\ro
  L_R)}q^{\brk{k+(1-\ro)\rho,k+(1-\ro)\rho}}$$
$$
\Delta_+=q^{-3(\ro-1)^2\brk{\rho,\rho}}\sum_{k\in L_R/(\ro
  L_R)}q^{\brk{k,k}}.
$$ 
Thus at least if $\ro$ is coprime with $\det(a_{i,j})$ then
$|\Delta_+|=\ro^{\frac\rk2}\neq0$.  In general we have the following theorem:
\begin{teo}
  If $\displaystyle{\sum_{k\in L_R/(\ro L_R)}q^{\brk{k,k}}\neq0}$ then $\D^\theta$ is
  $\Gr$-modular relative to $\X$ with modified dimension $\qd$ and periodicity
  group $\TT$.
\end{teo}
\newpage
\appendix
\section{Erratum : correction to the proof of Theorem \ref{T:T_adm}
}\label{A:err}
Marco De Renzi wrote a survey (\cite{A-mdR}) on the quantum invariants
defined in the present paper and on the fact that they classify lens
spaces proved in a later publication (\cite{A-BCGP}) by Christian
Blanchet and the authors.  This led him to point us a mistake in the
proof of Theorem \ref{T:T_adm} 
 of this paper.
We stress that the statement of the theorem is correct as it is;
furthermore also the proof is correct in all the explicit examples we
know of (i.e. when $\cat$ is the category of representations of an
unrolled quantum group). Still, for general $G$-modular categories
relative to $X$ with modified dimension $d$ and periodicity group
$\TT$ (Definition \ref{D:GmodularCat}
) our proof contains a mistake in the following
statements at page 36 ``it is sufficient to add to the above list of
genericity conditions the condition $g_{\coh_H}
(L'_i)\tilde{\alpha}^{x^i _s} ̃\notin \tilde{X}$, for every component
$L'_i$ of $L$ and for all $s\in \{1,\ldots, l-1\}$. The union of these
conditions and those found previously can be satisfied for a suitable
choice of $\alpha$ since no union of finitely many translates of
$\tilde{X}$ covers $G$ (see Definition 4.2).''  Indeed in our
construction, $F'$ is defined by a $G$-graded ribbon category where
$G$ is an abelian group which contains a bad set $\X\subset G$.  The
assumption is that $\X$ is small in the sense that no finite number of
translated copies of $\X$ can cover $G$:
\begin{equation}
  \label{eq:X1}
  \forall g_1,\ldots,g_N\in G, \bigcup_{i=1}^Ng_i\X\neq G,
\end{equation}
 while in the above sentence we were implicitly
using the following stronger hypothesis
\begin{equation}
  \label{eq:X2}
  \forall n\in\N,\,\forall g_1,\ldots,g_N\in G, 
  \bigcup_{i=1}^Ng_i\X\neq G_n=\{x^n:x\in G\}.
\end{equation}
Remark that this stronger hypothesis is satisfied for the main known
application coming from quantum groups.  But the original hypothesis
was already used in a previous paper of the last two author and
Vladimir Turaev to construct 3-manifold invariant ``\`a la
Turaev-Viro'' (\cite{A-GPT2,A-GP2}).

We provide here below an account of how to correct this; in short the idea of the proof is the same : i.e. use $H$-stabilizations to modify possible critical Kirby moves, but, as opposed to what we wrote, in general one needs more than one $H$-stabilization. 


We give in the last section of this appendix an alternative proof of this theorem
(essentially reproducing the first part of the proof of the existence
of a computable presentation that was not affected by the mistake).
In the next section we also take the chance to give a list of related works that appeared
after the publication of the present paper.

\subsection{Related works}
\newcommand{\Nc}{{\mathsf N}_\cat}
For a fixed abelian group $G$, and $\cat$ a $G$-modular category
relative to $\X\subset G$, we have the notion compatible triple
$(M,T,\coh)$ where $T$ is a ribbon graph in the closed 3-manifold $M$,
$\coh\in H^1(M\setminus T,G)$ and the edges of $T$ are colored by objects of
$\cat$ of degree equal to the value of the cohomology class on
meridians of these edges.

The most studied example is the $\C/2\Z$-modular category
of weight modules of the unrolled quantum group $\UsltH$ at
$q=\e^{i\pi/r}$ where $r\in\Z_{\ge2}$, $r\notin4\Z$.  But it is shown
that for any simple Lie algebra $\g$ and for many root of unity $q$,
the unrolled quantum group of $\g$ at $q$ gives a relative $G$-modular
category and invariants of compatible triple.

The set of compatible triples splits as the disjoint union of the set
of admissible triples (having a computable presentation) and its
complement formed by non admissible presentations.  
In the present paper,
two invariants $\Nc$ and $\Nc^0$ are defined on respectively
the set of admissible triples and the set of all triples.  It should
be noticed that in most examples, $\Nc^0$ vanishes on the set of
admissible presentations so the two invariants are in a sense
complementary.  For the preferred example of $\UsltH$ at
$q=\e^{i\pi/r}$, $\Nc$ is just denoted by $\Nr$.
\begin{enumerate}
\item In \cite{A-CGP2}, it is conjectured and shown in many cases that
  the invariant $\Nr^0$ is equivalent to the Kirby-Melvin version of
  the Witten-Reshetikhin-Turaev invariants for rational homology
  spheres.
\item The paper \cite{A-CGP3}, is a survey of known and new algebraic
  results on the category of weight modules of the unrolled quantum
  group $\UsltH$.  Some of these results are used in \cite{A-BCGP}.
\item The paper \cite{A-BCGP} is building and studying the remarkable
  properties of the TQFT associated to $\Nr$.  This TQFT is obtained
  by using a universal construction process.
\item In \cite{A-BCGP2}, it is shown that the category of weight modules
  of the unrolled quantum group $\UsltH$ at $q=\e^{i\pi/r}$ where
  $r\in4\Z_{>0}$ leads to similar families of invariants $\Nr$, $\Nr^0$
  but for compatible triple where the cohomology class $\coh$ is
  replaced by a kind of spin structure.
\item Finally, \cite{A-mdR} is a survey of results of 
  this paper and \cite{A-BCGP}.
\end{enumerate}

\subsection{Corrected proof of Theorem \ref{T:T_adm}
}
For a computable presentation $P=(L\cup T,g_\coh)$ of a compatible
triple $(M,T,\coh)$, let $$\Nc(P)=\dfrac{F'(L\cup T)}{\Delta_+^{p}\
  \Delta_-^{s}}.$$

The idea of the proof is to modify a sequence of moves relating two
computable presentations to a path relating them through computable
presentations.  This is not always possible using only Kirby moves and
we need 
another move, called $H$-stabilization in the text, 
that changes the compatible triple
presented but does not change the value by $\Nc$ of the presentations
up to a controlled constant, as follows.
Starting from a link presentation of a compatible triple with an edge of
$T$ colored by $\beta\in\A$, one can remove a 3-ball containing a
portion of this edge and replace it with the $(\alpha,\beta)$ long
Hopf link whose long strand is colored by $\beta$, and whose circle
component is colored by $\alpha\in \A$.  We then say that the second
link is obtained from the first by {\em a positive
  $(\alpha,\beta)$-Hopf move}.  We call the inverse move a negative
$(\alpha,\beta)$-Hopf move.  Remark that a positive Hopf move gives a
link presentation of a $H$-stabilization of the original compatible
triple.
\\
Also, if a presentation $P_1$ is computable, and a presentation $P_2$
is obtained from $P_1$ by a positive $(\alpha,\beta)$-Hopf move then
clearly the presentation $P_2$ is computable and
$$\Nc(P_2)=\brk{H(\alpha,\beta)}\Nc(P_1).$$
In particular, as $P_2$ is a presentation of a $H$-stabilization of
the triple, this will prove the last part of the theorem.
\\

The proof of the fact that there exists a computable presentation of
an $H$-stabilization of a $T$-admissible three-uple $(M,T,\coh)$ is
correct.  Let's rapidly recall it for the sake of self-containedness:
since $(M,T,\coh)$ is $T$-admissible then $T\neq \emptyset$ and has an
edge $e$ colored by $\beta\in \A$.  Pick a surgery presentation $P$ of
$(M,T,\coh)$ over a link $L\subset S^3$.  If $L=\emptyset$ then since
$e\in T$ is colored by $\beta\in \A$, then $T\cup L$ is computable.
If the presentation is not computable then there are components of $L$
whose color is in $\X$.  Let $C$ be this finite set of colors.  Apply
a positive $(\alpha,\beta)$-Hopf move on $P$ and slide the
$\alpha$-colored circle on each component of $L$ that have a color in
$C$.  These Kirby II moves change the colors of the involved
components of $L$ to colors in $\wb \alpha C$ where $\wb\alpha$ is the
degree of $\alpha$.  Hence as soon as $\wb\alpha\notin\bigcup_{c\in
  C}c^{-1}\X\neq G$, the resulting presentation is a computable
presentation of a $H$-stabilisation of $(M,T,\coh)$.

Next, we want to prove that $\Nc$ takes the same value on two computable
presentations $P,P'$ of the same compatible triple.  Let first prove
the theorem in the case where $P$ and $P'$ are ``isotopic in $S^3_L$'':
they have the same surgery link $L$ and their $T$-part graphs are related by a
sequence of Kirby II moves corresponding to an isotopy of the graph in
the surgered manifold $S^3_L$.  In this case, the sequence is formed
by isotopy in $S^3\setminus L$ and sliding of edges of the $T$-graph on
components of $L$ and the proof given in the text is correct as is; let's us sketch it here for the sake of completeness.  
During this sequence, $L$ is unchanged but the
$G$-colors of the components of $L$ are multiplied by the degree of
the objects coloring the edges of the graph sliding on them.  Let
$C\subset G$ be the finite set of colors of components of $L$ that
appear during this sequence.  If $C\cap \X=\emptyset$, then the
sequence passes only through computable presentations and there is no problem. 
Else, we first apply a positive $(\alpha,\beta)$-Hopf
move on $P$.  Then we slide the newly created $\alpha$-colored
circle on each component of $L$.  This multiplies the colors of the
components of $L$ by the degree $\wb \alpha$ of $\alpha$.  Then we
apply the sequence of Kirby II moves ignoring $\alpha$ and finally we
slide back the $\alpha$-colored circle from each component of $L$ to get
a presentation which is related to $P'$ by a negative
$(\alpha,\beta)$-Hopf move.  The set of $G$-colors that appear during
this process is $\wb\alpha.C$.  So it is enough to choose $\alpha$
such that $\wb\alpha\notin\bigcup_{c\in C}c^{-1}\X$ to ensure that the
new path is done among computable presentations.  And it follows that
$\Nc(P)=\Nc(P')$.

We now consider the general case where two presentations
$P=L\cup T$ and $P'=L'\cup T$  are two
computable presentations of diffeomorphic triples through a diffeomorphism $f$.
Then let $L=L^0\stackrel{s_1}{\to}\ldots \stackrel{s_k}{\to}L^k=L'$ be
a sequence of handle-slides, blow-up moves and blow-down moves
connecting the two presentations and inducing the diffeomorphism $f$
between $S^3_{L}\setminus T$ and $S^3_{L'}\setminus T$.
Now we modify the moves $s_i$ for $i$ from $1$ to $k$ in the following
way:
\\
1) If $s_i$ is a Kirby II move generating a color in $\X$ for a
component $L_j$ of $L$, we first do a positive $(\alpha,\beta)$-Hopf
move and slide the $\alpha$-colored circle on $L_i$.  Here $\alpha$ is
chosen so that no color in $\X$ appear during this Kirby II move.  In
the sequel, we just ``ignore" the component $\alpha$.
\\
2) If $s_i$ is a blow up (positive Kirby I) move creating a $\pm1$
framed meridian component around an edge with color $x\in \X$, then the
compatibility condition implies that the $G$ color of this new
component of the surgery link is $x$.  In this case, we first do a
positive $(\alpha,\beta)$-Hopf move for some $\alpha$ with $\wb
\alpha\notin \X\cup x^{-1}\X$, then do the blow up around the new
$\alpha$-colored edge creating a meridian $m$ and then slide the
$x$-colored edge on $m$.  The result is a presentation similar to the
result of $s_i$ except for the additional $\alpha$-colored circle
colored by $\alpha$.  The condition on $\wb\alpha$ ensures that the new
presentation is computable.
\\
3) If $s_i$ is a blow down (negative Kirby I) move removing a $\pm1$
framed unknot $m$, first remark that this unknot might be linked with
more than one edge: indeed it is possible that some of the components
created by the positive H-stabilizations are linked with $m$.  Hence to
perform this blow down we might need first to move some component of
the graph by sliding them on $m$ to unlink them from $m$. But this process
changes the $G$-color of $m$ that might pass through elements of $\X$.
To solve this problem we first do a positive $(\alpha,\beta)$-Hopf move,
slide the edges linked to $m$ on $m$ to unlink them and slide the
$\alpha$-colored circle on $m$. This final presentation is computable
for any $\alpha$ such that $\wb \alpha\notin \X$ but the intermediate
steps might not.  Nevertheless, the starting and final presentations
are ``isotopic in $S^3_L$'' thus related through computable presentation as proved above.

At the end of this sequence of moves, we get the presentation $P'$
with many circle components added by the positive Hopf moves.  But
this presentation is ``isotopic in $S^3_L$'' to a presentation obtained by
performing some positive $(\alpha,\beta)$-Hopf moves on $P'$ (for some
values of $\alpha$ in the same subset of $\A$).

Lemmas 
\ref{L:handle-slide} and \ref{L:Stab} ensure that the Kirby moves do
not change the invariant while positive and negative
$(\alpha,\beta)$-Hopf moves multiply it by the inverse quantity.
Hence we get that $\Nc(P)=\Nc(P')$.
\renewcommand{\refname}{References of the appendix} 

\newpage
\renewcommand{\refname}{References}

\end{document}